\documentclass[reqno, oneside, 12pt]{amsart}
\usepackage[letterpaper]{geometry}
\usepackage{amsmath,amssymb}
\usepackage{amsthm}
\usepackage{thmtools}
\usepackage{mathtools}
\usepackage{dsfont}
\usepackage{color}
\usepackage{enumitem}
\usepackage{hyperref}
\usepackage{appendix}

\usepackage{comment}
\usepackage[normalem]{ulem}

\usepackage{cases}
\usepackage{multirow}
\usepackage{hhline}
\usepackage{mathrsfs}

\newcommand{\Z}{{\mathbb{Z}}}

\newcommand{\R}{{\mathbb{R}}}
\newcommand{\C}{{\mathbb{C}}}

\newcommand{\mR}{{\mathcal{R}}}

\newcommand{\HH}{{\mathbb{H}}}

\DeclareMathOperator{\rank}{rank}
\newcommand{\be}{\begin{equation}}
	\newcommand{\ee}{\end{equation}}

\DeclareMathOperator{\Arg}{Arg}

\newcommand{\defeq}{\vcentcolon=}
\newcommand{\floor}[1]{\left\lfloor #1 \right\rfloor}

\newcommand{\ceil}[1]{\left\lceil #1 \right\rceil}

\renewcommand{\(}{\left(}
\renewcommand{\)}{\right)}

\newcommand{\Mod}[1]{\ (\mathrm{mod}\ #1)}
\DeclareMathOperator{\sgn}{sgn}

\newtheorem{theorem}{Theorem}

\newtheorem{lemma}[theorem]{Lemma}
\newtheorem{corollary}[theorem]{Corollary}
\newtheorem{proposition}[theorem]{Proposition}

\newtheorem{condition}[theorem]{Condition}

\theoremstyle{remark}
\newtheorem*{remark}{Remark}

\numberwithin{equation}{section}
\numberwithin{theorem}{section}
\numberwithin{lemma}{section}
\numberwithin{proposition}{section} 
\numberwithin{example}{section}
\numberwithin{definition}{section}
\numberwithin{corollary}{section}
\numberwithin{condition}{section}
\numberwithin{notation}{section}
\numberwithin{claim}{section}

\numberwithin{table}{section}

\author{Qihang Sun}
\title{Vanishing properties of Kloosterman sums and Dyson's conjectures}

\date{\today}

\begin{document}

	\begin{abstract}
		In a previous paper \cite{QihangExactFormula}, the author proved the exact formulae for ranks of partitions modulo each prime $p\geq 5$. In this paper, for $p=5$ and $7$, we prove special vanishing properties of the Kloosterman sums appearing in the exact formulae. These vanishing properties imply a new proof of Dyson's rank conjectures. Specifically, we give a new proof of Ramanujan's congruences $p(5n+4)\equiv 0\Mod 5$ and $p(7n+5)\equiv 0\Mod  7$. 
	\end{abstract}

	\maketitle

	\section{Introduction}

	Let $p(n)$ denote the integer partition function. Ramanujan obtained the famous congruence properties of $p(n)$:
	\begin{equation}\label{Ramanujan congruences}
		p(5n+4)\equiv 0\Mod 5,\quad p(7n+5)\equiv 0\Mod 7,\quad p(11n+6)\equiv 0\Mod{11}. 
	\end{equation}  
	In 1944, Dyson \cite{Dyson} defined the rank of a partition and conjectured a beautiful explanation for Ramanujan's congruences. Suppose $\Lambda=\{\Lambda_1\geq \Lambda_2\geq \cdots\geq \Lambda_\kappa\}$ is a partition of $n$, i.e. $\sum_{j=1}^\kappa \Lambda_j=n$. Then the rank of $\Lambda$ is defined by
	\[\rank(\Lambda)\defeq\Lambda_1-\kappa.\]
	Let the quantities $N(m,n)$ and $N(a,b;n)$ be defined by
	\begin{equation}
		N(m,n)\defeq \#\{\Lambda \text{ is a partition of }n: \rank \Lambda=m\}
	\end{equation} 
	and 
	\begin{equation}
		N(a,b;n)\defeq \#\{\Lambda \text{ is a partition of }n: \rank \Lambda\equiv a\Mod b\}.
	\end{equation}
	Let $q=\exp(2\pi i z)=e(z)$ for $z\in \HH$ and $w$ be a root of unity. The generating function of $N(m,n)$ is given by (see e.g. \cite[p. 245]{BrmOno2006ivt})
	\begin{equation}\label{rankGeneratingFunction}
		\mR(w;q)\defeq 1+\sum_{n=1}^\infty\sum_{m=-\infty}^\infty N(m,n) w^mq^n=1+\sum_{n=1}^\infty\frac{q^{n^2}}{(wq;q)_n(w^{-1}q;q)_n},
	\end{equation}
	where $(a;q)_n\defeq\prod_{j=0}^{n-1}(1-aq^j)$. 
	Dyson made the following conjectures which were proved by Atkin and Swinnerton-Dyer in 1953. 
	\begin{theorem}[{\cite{AtkinSDrank}}]
		\label{Dyson's conjectures in ASD}
		For all $n\geq 0$, we have the following identities: 
		\begin{align*}
			N(1,5;5n+1)&=N(2,5;5n+1);			\tag{5-1} 		\\
			N(0,5;5n+2)&=N(2,5;5n+2);			\tag{5-2}		 \\
			N(0,5;5n+4)=N(1,5;5n+4)&=N(2,5;5n+4); 	\tag{5-4}\\
			N(2,7;7n)&=N(3,7;7n); 					\tag{7-0}\\
			N(1,7;7n+1)=N(2,7;7n+1)&=N(3,7;7n+1); 	\tag{7-1}\\
			N(0,7;7n+2)&=N(3,7;7n+2); 				\tag{7-2}\\
			N(0,7;7n+3)=N(2,7;7n+3), &\quad \ N(1,7;7n+3)=N(3,7;7n+3); \tag{7-3}\\
			N(0,7;7n+4)=N(1,7;7n+4)&=N(3,7;7n+4); 				\tag{7-4}\\
			N(0,7;7n+5)=N(1,7;7n+5)&=N(2,7;7n+5)=N(3,7;7n+5); 	\tag{7-5}\\
			N(0,7;7n+6)+N(1,7;7n+6)&=N(2,7;7n+6)+N(3,7;7n+6).  	\tag{7-6}
		\end{align*}
	\end{theorem}
	\begin{remark}
		By $N(a,b;n)=N(-a,b;n)$, the identity (5-4) implies 
		\[N(0,5;5n+4)=N(1,5;5n+4)=\cdots=N(4,5;5n+4)=\tfrac 15 p(5n+4)\]
		hence implies the Ramanujan congruence $p(5n+4)\equiv 0\Mod 5$. The identity (7-5) implies 
		\[N(0,7;7n+5)=N(1,7;7n+5)=\cdots=N(6,7;7n+5)=\tfrac 17 p(7n+5)\]
		hence implies the Ramanujan congruence $p(7n+5)\equiv 0\Mod 7$. 
	\end{remark}
	
	The proof of Theorem~\ref{Dyson's conjectures in ASD} in \cite{AtkinSDrank} involves identities of generating functions
	\begin{equation*}
		\sum_{n=0}^{\infty}  \(N(a,p;pn+k)-N(b,p;pn+k)\) x^{pn} \prod_{r=1}^{\infty}(1-x^r)
	\end{equation*}
	for $p=5,7$ and certain choices of the integer $k$. See \cite[Theorem~4 \& Theorem~5]{AtkinSDrank} for details. Recently, Garvan \cite[\S6]{GarvanTransformationDyson2017} gave a new and simplified proof of Dyson's conjectures. For
	\begin{equation*}
		\mathcal{K}_{p,0}(z)=\prod_{n=1}^{\infty}(1-q^{pn})\sum_{n=\ceil{(p^2-1)/24p}}\(\sum_{k=0}^{p-1}N(k,p;pn-\tfrac{p^2-1}{24})\zeta_p^k\)q^n,
	\end{equation*}
	with $\zeta_p\defeq e(\frac 1p)$ and $\mathcal{K}_{p,m}(z)$ defined in \cite[Definition~6.1]{GarvanTransformationDyson2017}, Garvan showed that $\mathcal{K}_{p,0}(z)$ is a weakly holomorphic modular form of weight $1$ on $\Gamma_1(p)$. By the Valence formula, Garvan proved $\mathcal{K}_{5,0}(z)=\mathcal{K}_{7,0}(z)=0$ and hence proved the Dyson's conjectures in \cite[\S6.3]{GarvanTransformationDyson2017}. 
	
	For integers $b>a>0$, denote $A(\frac ab;n)$ as the Fourier coefficient of $\mR(\zeta_b^a;q)$: 
	\[\mR(\zeta_b^a;q)=:1+\sum_{n=1}^\infty A\(\frac ab;n\)q^n. \]
 There is an important equation which explains the relation between $A(\frac ab;n)$ and $N(a,b;n)$:
	\begin{equation}\label{Relation between rank mod and A()}
		bN(a,b;n)=p(n)+\sum_{j=1}^{b-1}\zeta_b^{-aj}A\(\frac jb;n\). 
	\end{equation}
	It is not hard to show that $A(\frac jb;n)\in \R$ and $A(\frac jb;n)=A(1-\frac jb;n)$ for $1\leq j\leq b-1$, because $N(a,b;n)=N(-a,b;n)$ and $\zeta_b^{-aj}+\zeta_b^{-a(b-j)}=2\cos (\frac{\pi a j}b)$. Specifically, if we know the values of $A(\frac jb;n)$ for $1\leq j\leq b-1$, then we know the value of
	\[N(a_1,b;n)-N(a_2,b;n) \quad \text{for any }0\leq a_1,a_2\leq b-1. \]

	Another way of approaching Dyson's conjectures is therefore via the formulae for $A(\frac jb;n)$ when $b=5,7$. In 2009, Bringmann \cite[Theorem~1.1]{BringmannTAMS} proved the asymptotic formula for $A(\frac jb;n)$ when $b\geq 3$ is odd and $0\leq j\leq b-1$. Bringmann used the asymptotic formula when $b=3$ to prove the Andrews-Lewis conjecture about comparing $N(0,3;n)$ and $N(1,3;n)$. 
	In a previous paper \cite{QihangExactFormula}, for each prime $p\geq 5$ and $1\leq \ell\leq p-1$, the author proved that Bringmann's asymptotic formula, when summing up to infinity, is the exact formula for $A(\frac \ell p;n)$ for all $n\geq 0$. For each integer $A$, denote $[A]$ by
	\[0\leq [A]<7:\quad [A]\equiv A\Mod 7.\]
	When $p=5$ or $7$, that exact formula reduces to the following corollary. 
	\begin{theorem}[{\cite[Corollary~2.2]{QihangExactFormula}}]\label{main theorem corollary}
		For every positive integer $n$, when $p=5$ and $1\leq \ell\leq 4$, we have
		\begin{align}\label{main theorem formula, mod 5}
			\begin{split}
				A\(\frac \ell 5;n\)=\frac{2\pi e(-\frac 18)\sin(\frac{\pi\ell}5)}{(24n-1)^{\frac14}} \sum_{c>0:\,5|c}\frac{S_{\infty\infty}^{(\ell)}(0,n,c,\mu_5)}{c} I_{\frac12}\(\frac{4\pi \sqrt{24n-1}}{24c}\); 
			\end{split}
		\end{align}
		when $p=7$ and $1\leq \ell \leq 6$, we have
		\begin{align}\label{main theorem formula, mod 7}
			\begin{split}
				A\(\frac \ell 7;n\)&=\frac{2\pi e(-\frac 18)\sin(\frac{\pi\ell}7)}{(24n-1)^{\frac14}} \sum_{c>0:\,7|c}\frac{S_{\infty\infty}^{(\ell)}(0,n,c,\mu_7)}{c} I_{\frac12}\(\frac{4\pi \sqrt{24n-1}}{24c}\)\\
				&+\frac{4\pi \sin(\frac{\pi\ell}7)}{(24n-1)^{\frac14}}\sum_{\substack{a>0:\,7\nmid a,\\ [a\ell]=1\text{ or }6}}\!\!\!\!\!\frac{S_{0\infty}^{(\ell)}\(0,n,a,\mu_7;0\)}{\sqrt 7\,a }I_{\frac12}\(\frac{4\pi \sqrt{24n-1}}{24\times 7a}\).
			\end{split}
		\end{align}
		Here $S_{\infty\infty}^{(\ell)}(0,n,c,\mu_p)$ for $p=5,7$ and $S_{0\infty}^{(\ell)}\(0,n,a,\mu_7;0\)$ are given in \eqref{S infty infty for mod p}, \eqref{S 0 infty mod 7, +1} and \eqref{S 0 infty mod 7, -1}.  
	\end{theorem}

	Theorem~\ref{main theorem corollary} gives us a new way to directly compute $A(\frac \ell p;n)$ for $p=5,7$ and all $\ell$. In this paper, we give a new proof of Theorem~\ref{Dyson's conjectures in ASD} by establishing the following vanishing properties of the Kloosterman sums $S_{\infty\infty}^{(\ell)}(0,n,c,\mu_p)$ for $p=5,7$ and $S_{0\infty}^{(\ell)}\(0,n,a,\mu_7;0\)$. This is totally different from the methods in \cite{AtkinSDrank} and \cite{GarvanTransformationDyson2017}. 
	
	\begin{theorem}\label{Kloosterman sums vanish}
		(i) For all integers $n\geq 0$ and $1\leq\ell\leq p-1$ for $p=5,7$ (denoted by $p|c$ below), we have the following vanishing conditions for the Kloosterman sums appeared in Theorem~\ref{main theorem corollary}: 
		\begin{enumerate}
			\item[{\rm(5-4)}] If $5|c$, we have $S_{\infty\infty}^{(\ell)}(0,5n+4,c,\mu_5)=0$.
			\item[{\rm(7-5,1)}] If $7|c$, $\frac c7\cdot \ell\not\equiv 1 \Mod 7$, and $\frac c7\cdot\ell\not \equiv -1\Mod 7$, then $S_{\infty\infty}^{(\ell)}(0,7n+5,c,\mu_7;0)=0$.
			\item[{\rm(7-5,2)}] If $7|c$, $7\nmid a$, $a\ell\equiv \pm1 \Mod 7$, and $c=7a$, we have 
			\[ e(-\tfrac 18)S_{\infty\infty}^{(\ell)}(0,7n+5,c,\mu_7)+2{\sqrt 7}\,S_{0\infty}^{(\ell)}(0,7n+5,a,\mu_7;0)=0.\]
		\end{enumerate}
		
		(ii) Furthermore, we denote $C_p^{a,b}\defeq \cos(\frac{a\pi }p)-\cos(\frac{b\pi}p)$ and
		\begin{align*}
			S_7^{(\ell)}(n,c)\defeq \sin(\tfrac {\pi\ell}7) \(e(-\tfrac 18){S_{\infty\infty}^{(\ell)}(0,n,c,\mu_7)}+\mathbf{1}_{\substack{ a\defeq c/7 \\ {[a\ell]=1,6}}}\cdot
			2\sqrt 7\,S_{0\infty}^{(\ell)}(0,7n+5,a,\mu_7;0)\)
		\end{align*}
		for simplicity, where $\mathbf{1}_{\textrm{condition}}$ equals 1 if the condition meets and equals 0 otherwise. We also have the following vanishing conditions for all $c\in \Z$ divisible by $p$, where $p=5$ or $7$ is marked at the subscript of $C_p^{a,b}$: 
		\begin{align*}
			C_5^{2,4}\sin(\tfrac{\pi}5)S_{\infty\infty}^{(1)}(0,5n+1,c,\mu_5)+C_5^{4,2}\sin(\tfrac{2\pi}5)S_{\infty\infty}^{(2)}(0,5n+1,c,\mu_5)&=0,\tag{5-1}\\
			C_5^{0,4}\sin(\tfrac{\pi}5)S_{\infty\infty}^{(1)}(0,5n+2,c,\mu_5)+C_5^{0,2}\sin(\tfrac{2\pi}5)S_{\infty\infty}^{(2)}(0,5n+2,c,\mu_5)&=0,\tag{5-2}\\
			C_7^{4,6}S_7^{(1)}(7n,c)+C_7^{6,2}S_7^{(2)}(7n,c)+C_7^{2,4}S_7^{(3)}(7n,c)&=0,\tag{7-0}\\
			C_7^{2,4}S_7^{(1)}(7n+1,c)+C_7^{4,6}S_7^{(2)}(7n+1,c)+C_7^{6,2}S_7^{(3)}(7n+1,c)&=0,\tag{7-1,1}\\
			C_7^{4,6}S_7^{(1)}(7n+1,c)+C_7^{6,2}S_7^{(2)}(7n+1,c)+C_7^{2,4}S_7^{(3)}(7n+1,c)&=0,\tag{7-1,2}\\
			C_7^{0,6}S_7^{(1)}(7n+2,c)+C_7^{0,2}S_7^{(2)}(7n+2,c)+C_7^{0,4}S_7^{(3)}(7n+2,c)&=0,\tag{7-2}\\
			C_7^{0,4}S_7^{(1)}(7n+3,c)+C_7^{0,6}S_7^{(2)}(7n+3,c)+C_7^{0,2}S_7^{(3)}(7n+3,c)&=0,\tag{7-3,1}\\
			C_7^{2,6}S_7^{(1)}(7n+3,c)+C_7^{4,2}S_7^{(2)}(7n+3,c)+C_7^{6,4}S_7^{(3)}(7n+3,c)&=0,\tag{7-3,2}\\
			C_7^{0,2}S_7^{(1)}(7n+4,c)+C_7^{0,4}S_7^{(2)}(7n+4,c)+C_7^{0,6}S_7^{(3)}(7n+4,c)&=0,\tag{7-4,1}\\
			C_7^{2,6}S_7^{(1)}(7n+4,c)+C_7^{4,2}S_7^{(2)}(7n+4,c)+C_7^{6,4}S_7^{(3)}(7n+4,c)&=0,\tag{7-4,2}\\
			\(C_7^{0,4}+C_7^{2,6}\)S_7^{(1)}(7n+6,c)+\(C_7^{0,6}+C_7^{4,2}\)S_7^{(2)}(7n+6,c)&\\
			+\(C_7^{0,2}+C_7^{6,4}\)S_7^{(3)}(7n+6,c)&=0,\tag{7-6}
		\end{align*}
		
	\end{theorem}
	
	Using \eqref{Relation between rank mod and A()} and Theorem~\ref{main theorem corollary}, we have the following corollary. 
	\begin{corollary}\label{corollary Ramanujan congruence}
		For any pair {\rm{($p$-$k$)}} (or {\rm{($p$-$k$,$t$)}} for both $t=1,2$) in Theorem~\ref{Kloosterman sums vanish} with
		\[p=5,\ k\in\{1,2,4\}\quad \text{or} \quad p=7,\ k \in\{0,1,2,3,4,5,6\},\]
		we have Dyson's conjecture {\rm{($p$-$k$)}} in Theorem~\ref{Dyson's conjectures in ASD}. 
	\end{corollary}
	
	The paper is organized as follows. In Section~\ref{Section notations} we review about our notations of vector-valued Kloosterman sums as established in \cite{QihangExactFormula}. In Section~\ref{Section (5-4)} we provide a detailed proof of (5-4) of Theorem~\ref{Kloosterman sums vanish}. In Section~\ref{Section 7-5,1} and Section~\ref{Section 7-5,2} we prove (7-5,1) and (7-5,2) of Theorem~\ref{Kloosterman sums vanish}, respectively. Section~\ref{Section part ii} covers the proofs of (5-1), (5-2) and (7-0) of Theorem~\ref{Kloosterman sums vanish} but does not include the repetitive proofs for the remaining parts labeled (7-$k$). 
	
We have omitted detailed proofs for similar cases to avoid redundancy. The complete proof, along with the Mathematica code used for generating illustrative graphs within the proofs, are available in our GitHub repository \cite{QihangKLsumsGitHub}. 
	
	\section{Notations}
	\label{Section notations}
	In this section we define some notation involving Dedekind sums and Kloosterman sums. For the origin of these notations, see  \cite{GarvanTransformationDyson2017,QihangExactFormula}. 
	
	For integers $d$ and $m\geq 1$, let $\overline{ d_{\{m\}}}$ denote the inverse of $d\Mod m$. If there is a subscript, e.g. $d_1$, then we write $\overline{ d_{1\{m\}}}$ as the inverse of $d_1\Mod m$. 
	
	Define
	\[((x))\defeq \left\{
	\begin{array}{ll}
		x-\floor{x}-\tfrac12,& \text{ when }x\in \R\setminus\Z,\\
		0,& \text{ when }x\in \Z. 
	\end{array}
	\right.
	\]
	For integers $c>0$ and $(d,c)=1$, we define the Dedekind sum as
	\begin{equation}\label{Dedekind sum}
		s(d,c)\defeq\sum_{r\Mod c}\(\(\frac rc\)\)\(\(\frac{dr}c\)\). 
	\end{equation}
Here we use the notation $r\Mod c$ to indicate that the summation is taken over all residue classes modulo $c$. Similarly, $r\Mod c^*$ denotes that the summation is restricted to reduced residue classes, where $\gcd(r,c)=1$. For simplicity in more complex subscripts, we will abbreviate these as $r(c)$ and $r(c)^*$, respectively.  
 
	The Dedekind sums have the following properties  \cite[(4.2)-(4.5)]{Lewis1995ModDedekindSum}:  
	\begin{align}
		\label{Congru Dedekind theta is or not 3}&2\theta cs(d,c)\in \Z,\text{\ \ where }\theta=\gcd(c,3),\\
		\label{Congru Dedekind mod theta c}&12 cs(d,c)\equiv d+\overline{ d_{\{\theta c\}}}\Mod {\theta c},\\
		\label{Congru Dedekind c odd}&12 cs(d,c)\equiv c+1-2(\tfrac d c)\Mod8,\text{\ \ if }c\text{ is odd,}\\
		\label{Congru Dedekind c even}&12 cs(d,c)\equiv d+\(c^2+3c+1+2c(\tfrac cd)\)\overline{ d_{\{8\times 2^\lambda \}}}\Mod {8\times 2^\lambda},\text{\ \ if }2^\lambda \|c\text{ for }\lambda\geq 1. 
	\end{align}
	These congruences determine $12cs(d,c)\Mod {24c}$ uniquely in every case ($2|c$ or $2\nmid c$, $3|c$ or $3\nmid c$). 
	
	In the proof we use the following quadratic reciprocity of the Kronecker symbol $(\frac \cdot \cdot)$. For any non-zero integer $n$, write $n=2^\lambda n_o$ where $n_o$ is odd. For integers $m,n$ with $(m,n)=1$, we have
	\begin{equation}\label{quadratic reciprocity}
		\(\frac mn\)\(\frac nm\)=\pm(-1)^{(m_o-1)(n_o-1)/4},
	\end{equation}
	where we take $+$ if $m\geq 0$ or $n\geq 0$, and we take $-$ if $m<0$ and $n<0$.

	Next we define the Kloosterman sums $S_{\infty\infty}^{(\ell)}(0,n,c,\mu_p)$ for $p=5,7$ and $S_{0\infty}^{(\ell)}(0,n,c,\mu_7;0)$ appearing at Corollary~\ref{main theorem corollary}. We follow the notations of vector-valued Kloosterman sums in \cite[\S4.3]{QihangExactFormula}. 
	From \cite[(5.19), (5.29)]{QihangExactFormula}, when $p|c$ we have
	\begin{equation}\label{S infty infty for mod p}
		S_{\infty\infty}^{(\ell)}(0,n,c,\mu_p)=e(-\tfrac 18)\sum_{\substack{d\Mod c^*\\ad\equiv 1\Mod c}} \frac{(-1)^{\ell c}e(-\frac{3ca\ell^2}{2p^2})}{\sin(\frac{\pi a\ell}p)}\,e^{-\pi is(d,c)}e\(\frac{nd}c\).
	\end{equation} 
	
	When $p=7$, recall that $[A\ell]$ is the least non-negative residue of $A\ell\Mod 7$. From \cite[(5.31)]{QihangExactFormula}, when $A\ell=7T+1$ for some integer $T\geq 0$, we have
	\begin{equation}\label{S 0 infty mod 7, +1}
		S_{0\infty}^{(\ell)}(0,n,A,\mu_7;0)=(-1)^{A\ell-[A\ell]}\sum_{\substack{B\Mod A^*\\0<C<7A,\,7|C\\BC\equiv -1(A)}}e\(\frac{(\frac 32 T^2+\frac12 T)C}A\)e^{-\pi i s(B,A)}e\(\frac{nB}A\). 
	\end{equation}
	When $A\ell=7T-1$ for some integer $T\geq 1$, we have
	\begin{align}\label{S 0 infty mod 7, -1}
		\begin{split}
			&S_{0\infty}^{(\ell)}(0,n,A,\mu_7;0)\\
			&=(-1)^{A\ell-[A\ell]}\sum_{\substack{B\Mod A^*\\0<C<7A,\,7|C\\BC\equiv -1(A)}}e\(\frac{(\frac 32 (T-1)^2+\frac52 (T-1)+1)C}A\)e^{-\pi i s(B,A)}e\(\frac{nB}A\).
		\end{split}
	\end{align}
	If $[A\ell]\neq 1$ and $[A\ell]\neq 6$, then $S_{0\infty}^{(\ell)}(0,n,A,\mu_7;0)\defeq 0$.

	\section{Proof of (5-4) of Theorem~\ref{Kloosterman sums vanish}}
	\label{Section (5-4)}
	In this section we prove (5-4) of Theorem~\ref{Kloosterman sums vanish}. 
	We only consider $\ell=1,2$ because $A(\frac {\ell}p;n)=A(1-\frac {\ell}p;n)$. 
	
	Define $c'\defeq c/5$. For any integer $r$ with $(r,c')=1$, we define
	\[V(r,c)\defeq \{d\Mod c^*:\ d\equiv r\Mod {c'}\}. \]
	For example, $V(1,30)=\{d\Mod {30}^*: d\equiv 1,7,13,19\Mod{30}\}$ and $V(4,25)=\{d\Mod{25}^*: d\equiv 4,9,14,19,24\Mod{25}\}$. 
	Clearly, $|V(r,c)|=4$ if $5\| c$ and $|V(r,c)|=5$ if $25|c$. Moreover, $(\Z/c\Z)^*$ is the disjoint union 
	\[(\Z/c\Z)^*=\bigcup_{r\Mod{c'}^*}V(r,c). \] 
	
	By \eqref{S infty infty for mod p} we have
	\begin{equation}\label{S infty infty for mod 5 using}
		e(\tfrac 18)S_{\infty\infty}^{(\ell)}(0,5n+4,c,\mu_5)=\sum_{\substack{d\Mod c^*\\ad\equiv 1\Mod c}} \frac{(-1)^{\ell c}e(-\frac{3c'a\ell^2}{10})}{\sin(\frac{\pi a\ell}{5})}\,e^{-\pi is(d,c)}e\(\frac{(5n+4)d}c\).
	\end{equation}
	We claim the following proposition. 
	
	\begin{proposition}{\label{Vrc sum equals 0}}
		For $\ell=1,2$, the sum on $V(r,c)$ satisfies
		\begin{equation}{\label{Vrc sum equals 0 equation}}
			s_{r,c}\defeq \sum_{\substack{d\in V(r,c)\\ad\equiv 1\Mod c}} \frac{e(-\frac{3c'a\ell^2}{10})}{\sin(\frac{\pi a\ell}{5})}e^{-\pi is(d,c)}e\(\frac{4d}c\)=0.
		\end{equation}
	\end{proposition}
	If Proposition~\ref{Vrc sum equals 0} is true, then 
	\[S_{\infty\infty}^{(\ell)}(0,5n+4,c,\mu_5)=e(-\tfrac 18)(-1)^{\ell c}\sum_{r\Mod {c'}^*}s_{r,c} e\(\frac{nr}{c'}\)=0\]
	for all $n\in \Z$, $\ell=1,2$, and we have proved (5-4) of Theorem~\ref{Kloosterman sums vanish}.

	In the following subsections \S7.1-\S7.4, we prove Proposition~\ref{Vrc sum equals 0} when $5\|c$. In \S7.5, we prove Proposition~\ref{Vrc sum equals 0} when $25|c$. Suppose now that $5\|c$. Since $|V(r,c)|=4$, let $\beta\in \{1,2,3,4\}$ such that $\beta c'\equiv 1\Mod 5$ and we make a special choice of $V(r,c)$ as
	\begin{equation}\label{Vrc dj choice jbetac' p=5}
		V(r,c)=\{d_1,d_2,d_3,d_4\}\quad\text{where }d_j\equiv  j\Mod 5\text{ and }d_{j+1}=d_1+j\beta c'.
	\end{equation}
	We also take $a_j$ for $j\in \{1,2,3,4\}$ such that $a_j\equiv  j\Mod 5$, $a_{j+1}=a_1+j\beta c'$, and
	\begin{equation}\label{ajinverse5 dj =1}
		a_{\overline{j_{\{5\}}}}d_j\equiv 1\Mod c .
	\end{equation}
	These choices do not affect the sum \eqref{Vrc sum equals 0 equation} because $s_{r,c}$ has period $c$ in both $a$ and $d$. In \eqref{Vrc sum equals 0 equation}, we denote each summation term as
	\begin{equation}\label{P1P2P3}
		P(d)\defeq \frac{e\(-\frac{3c'a\ell^2}{10}\)}{\sin(\frac{\pi a\ell}{5})}\cdot e\(-\frac{12 cs(d,c)}{24c}\)\cdot e\(\frac{4d}c\)=:P_1(d)\cdot P_2(d)\cdot P_3(d),
	\end{equation}
	where $P_1(d)\defeq e(-\frac{3c'a \ell^2}{10})/\sin(\frac{\pi a \ell }5)$, $P_2(d)\defeq \exp(-\pi i s(d,c))$, and $P_3(d)\defeq e(\frac{4d}c)$. 
	
	\begin{remark}
		We keep $24c$ in the denominator of $P_2(d)$ because the congruence properties of the Dedekind sum are of the form $12cs(d,c)$. See \eqref{Congru Dedekind theta is or not 3}-\eqref{Congru Dedekind c even} for details. 
	\end{remark}

	We claim that the set of points $P(d)$ for $d\in V(r,c)$ must have the relative position illustrated in one of the following six configurations. Here $0<d_j<c$ for simplicity but we use \eqref{Vrc dj choice jbetac' p=5} in the proof.  
	
	\begin{center}
		
		\includegraphics[scale=0.35]{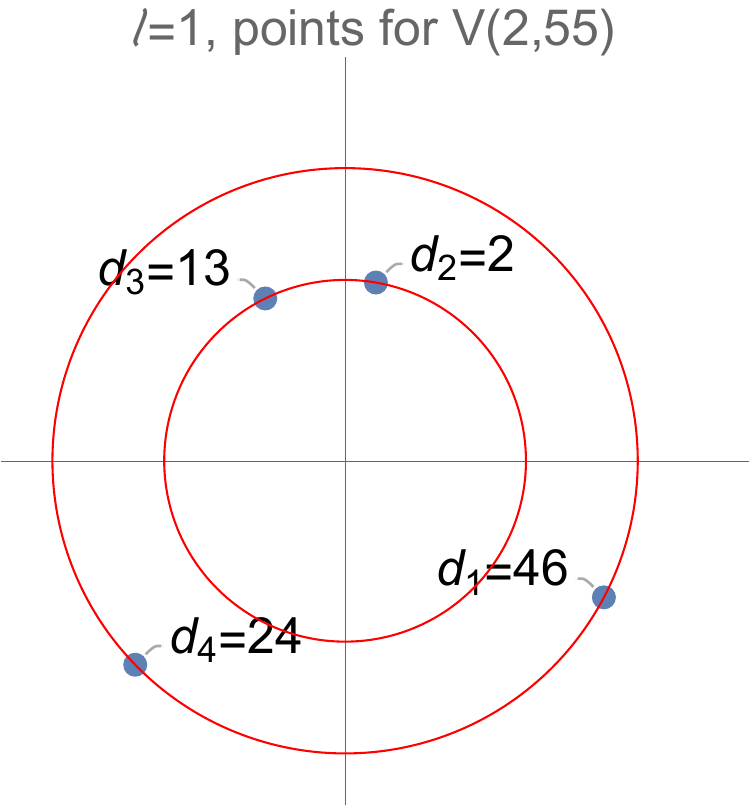}
		\includegraphics[scale=0.35]{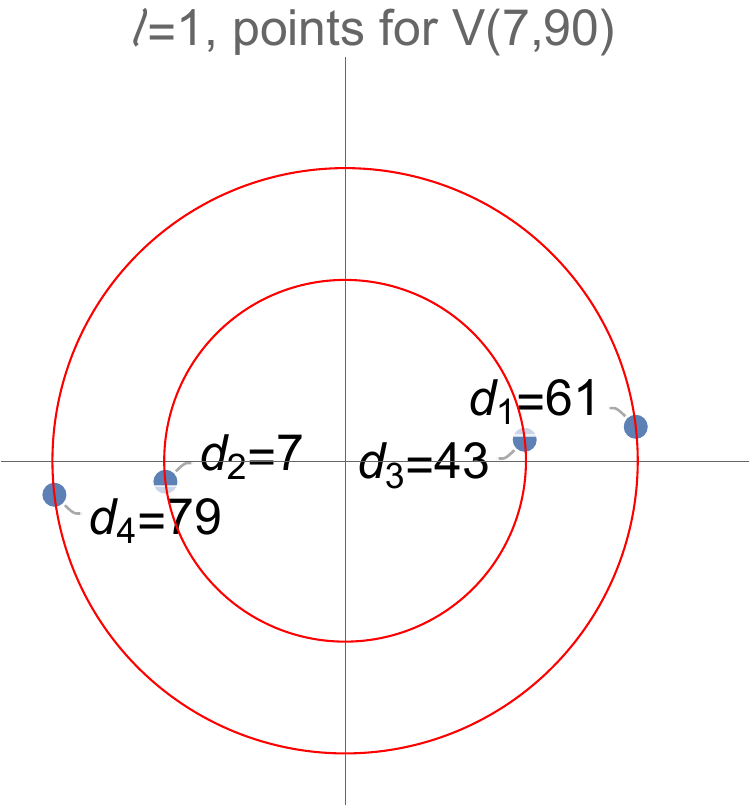}
		\includegraphics[scale=0.35]{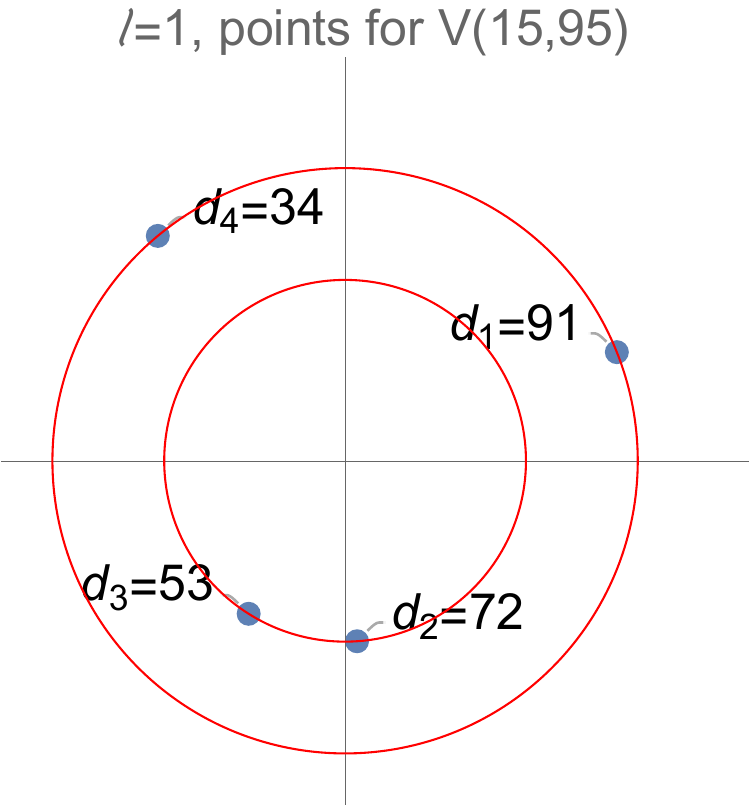}
		\\\vspace{10px}
		\includegraphics[scale=0.35]{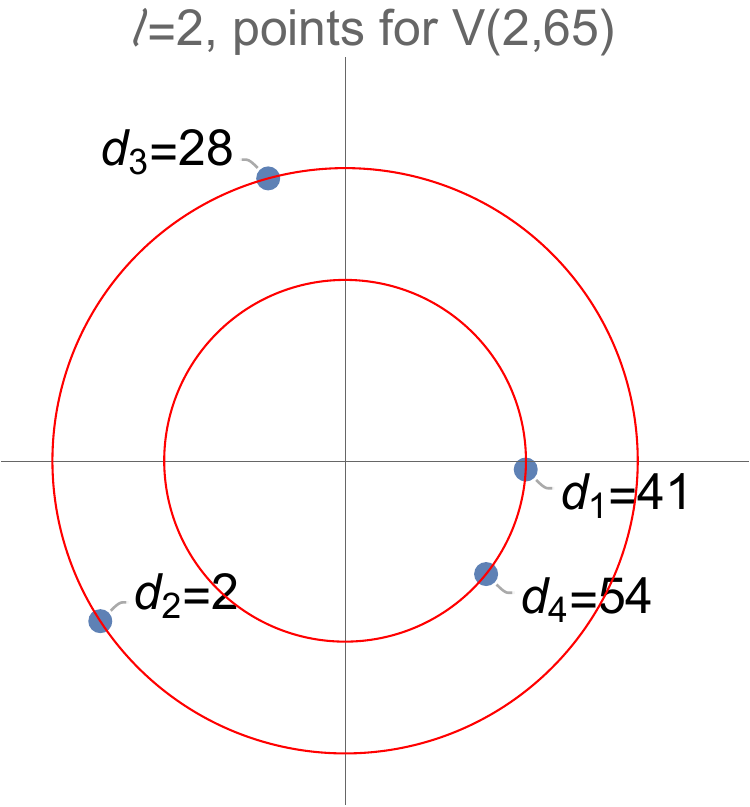}
		\includegraphics[scale=0.35]{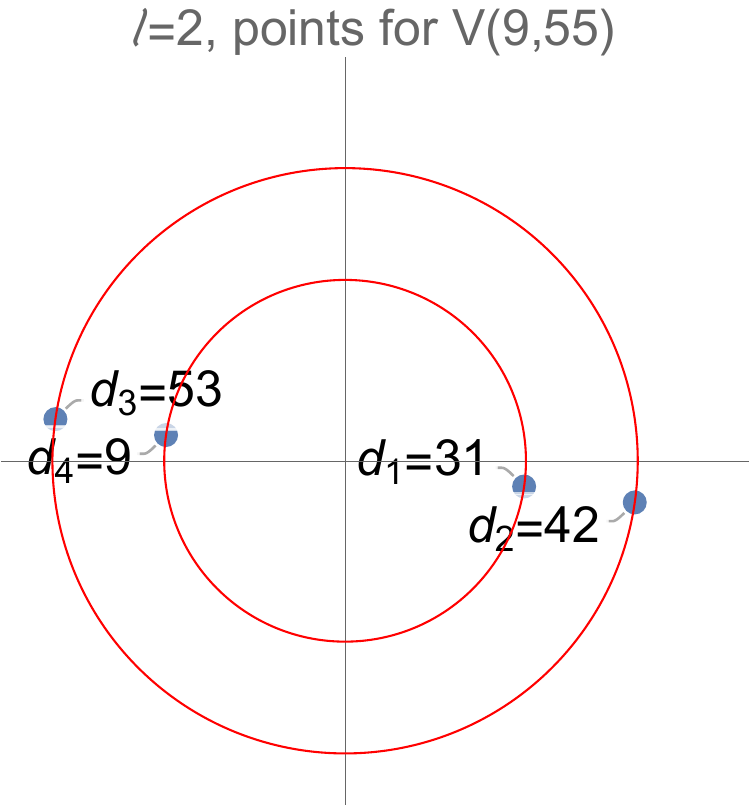}
		\includegraphics[scale=0.35]{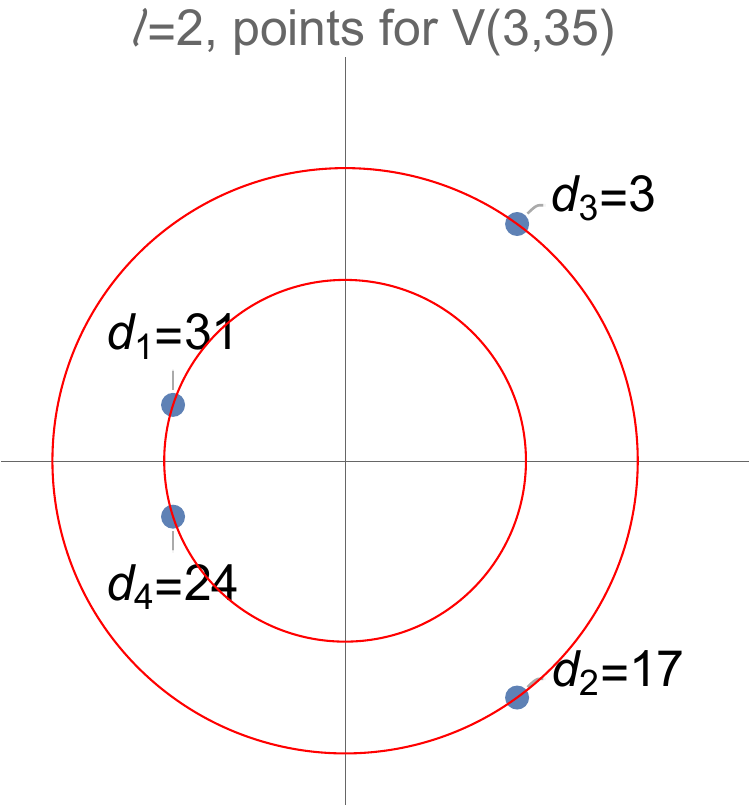}
	\end{center}

	Here we explain the styles. Each graph above has two circles with inner one of radius $\csc(\frac{2\pi}5 )$ and outer one with radius $\csc(\frac\pi5 )$. When $\ell=1$, the value of $P(d_1)$ and $P(d_4)$ will be on the outer circle ($P(d_2)$ and $P(d_3)$ on the inner circle) because the term $P_1(d_j)$ has denominator $\sin(\frac{\pi a_j\ell}5)$. When $\ell=2$, $P(d_1)$ and $P(d_4)$ will be on the inner circle. 
	
	We describe the relative argument differences via the following notation. Let	
	\begin{equation}\label{Arg def}
		\Arg_j(d_u\rightarrow d_v;\ell),\quad \text{ for }j\in\{1,2,3\}, \ u,v\in \{1,2,3,4\}, \text{ and } \ell\in\{1,2\}
	\end{equation}
	be the argument difference (as the proportion of $2\pi$, positive when going counter-clockwise) contributed from $P_j$ going from $d_u$ to $d_v$ when $\ell\in \{1,2\}$. To be precise, if we denote $P_j(d_u)=R_{j,u} \exp(i\Theta_{j,u})$ for $R_{j,u},\Theta_{j,u}\in \R$, then
	\[\Arg_j(d_u\rightarrow d_v;\ell)=\alpha\quad  \Leftrightarrow\quad \Theta_{j,v}-\Theta_{j,u} =\alpha\cdot 2\pi +2k\pi \text{ for some }k\in \Z. \]
	We say that two argument differences equal: $\Arg_j(d_u\rightarrow d_v;\ell)=\Arg_j(d_w\rightarrow d_x;\ell)$ or say $\Arg_j(d_u\rightarrow d_v;\ell)=\alpha$ if their difference is an integer. 
	
	Although the $P_2$ and $P_3$ terms are not affected by the value of $\ell$ in \eqref{P1P2P3}, we still use the notations $\Arg_2(d_u\rightarrow d_v;\ell)$ and $\Arg_e(d_u\rightarrow d_v;\ell)$ to indicate the different cases for $\ell$. Moreover, we define
	\begin{equation}\label{Arg total def}
		\Arg(d_u\rightarrow d_v;\ell)\defeq \sum_{j=1}^3 \Arg_j(d_u\rightarrow d_v;\ell)
	\end{equation}
	as the argument difference in total. 
	
	The following condition ensures Proposition~\ref{Vrc sum equals 0}. 
	\begin{condition}\label{Styles four points mod 5}
		We have the following six styles for the relative position of these four points. 
		\begin{itemize}
			\item $\ell=1$. First graph style: the arguments going $d_1\rightarrow d_2\rightarrow d_3\rightarrow d_4\rightarrow d_1$ are $\frac3{10}$, $\frac1{10}$, $\frac3{10}$, and $\frac3{10}$, respectively. The second graph style is that all the argument differences are $\frac12$, while the third graph style has the reversed order of rotation compared with the first one. 
			\begin{center}
				\begin{tabular}{|c|ccccccccc|}
					\hline
					$\Arg(d_u\rightarrow d_v;1)\searrow$ & $d_1$ & $\rightarrow$ & $d_2$ & $\rightarrow$ & $d_3$ & $\rightarrow$ & $d_4$ & $\rightarrow$ & $d_1$ \\
					\hline
					$c'\equiv 1\Mod 5$  & & $\frac 3{10}$ & & $\frac 1{10}$ & & $\frac 3{10}$ & & $\frac 3{10}$ & \ \\
					[1ex]
					$c'\equiv 2,3\Mod 5$ & & $\frac 12$ & & $\frac 12$ & & $\frac 12$ & & $\frac 12$ & \ \\
					[1ex]
					$c'\equiv 4\Mod 5$ & & $-\frac 3{10}$ & & $-\frac 1{10}$ & & $-\frac 3{10}$ & & $-\frac 3{10}$ & \ \\
					\hline
				\end{tabular}
			\end{center}
			\item $\ell=2$. Here are the styles for the graphs in the second row. 
			\begin{center}
				\begin{tabular}{|c|ccccccccc|}
					\hline
					$\Arg(d_u\rightarrow d_v;2)\searrow$& $d_1$ & $\rightarrow$ & $d_2$ & $\rightarrow$ & $d_3$ & $\rightarrow$ & $d_4$ & $\rightarrow$ & $d_1$ \\
					\hline
					$c'\equiv 3\Mod 5$ & & $-\frac 25$ & & $-\frac 3{10}$ & & $-\frac 25$ & & $\frac 1{10}$ & \ \\
					[1ex]
					$c'\equiv 1,4\Mod 5$ & & $0$ & & $\frac 12$ & & $0$ & & $\frac 12$ & \ \\
					[1ex]
					$c'\equiv 2\Mod 5$ & & $\frac 25$ & & $\frac 3{10}$ & & $\frac 25$ & & $-\frac 1{10}$ & \ \\
					\hline
				\end{tabular}
			\end{center} 
		\end{itemize}
	\end{condition}
	One can check that, whenever the four points on $\C$ satisfy any of the above cases of relative argument differences and corresponding radii, their sum becomes $0$. This can be explained by 
	\[\frac{\cos(\frac{\pi}{10})}{\sin(\frac{2\pi}5)}=\frac{\cos(\frac{3\pi}{10})}{\sin(\frac{\pi}5)}=1,\quad \text{where }\frac 1{\sin(\frac \pi 5)}\text{ and }\frac{1}{\sin(\frac{2\pi}5)}\text{ are the radii.} \]

	Before we divide into the cases, we first claim the following lemma: 
	\begin{lemma}\label{p=5 Argreduction}
		For $\ell\in \{1,2\}$, we have
		\begin{equation}
			\Arg(d_1\rightarrow d_2;\ell)+\Arg(d_4\rightarrow d_3;\ell)=0\quad\text{and}\quad  \Arg(d_1\rightarrow d_3;\ell)+\Arg(d_4\rightarrow d_2;\ell)=0. 
		\end{equation}	
	\end{lemma}
	Granted the above reduction, to prove that each case of the argument differences are one of the cases in Condition~\ref{Styles four points mod 5}, we only need to verify that
	\[\Arg(d_1\rightarrow d_4;\ell)\quad \text{and}\quad \Arg(d_1\rightarrow d_2;\ell)\quad {\text{for }}\ell=1,2\]
	satisfy Condition~\ref{Styles four points mod 5}. We prove this by enumerating all of the cases. We can list the argument differences for $\Arg_1$ and $\Arg_3$, but for $\Arg_2$, we require the congruence properties of Dedekind sums in \eqref{Congru Dedekind theta is or not 3}-\eqref{Congru Dedekind c even}.

	\begin{proof}[Proof of Lemma~\ref{p=5 Argreduction}]
		Note that \[\Arg(d_u\rightarrow d_v;\ell)=\Arg(d_u\rightarrow d_w;\ell)+\Arg(d_w\rightarrow d_v;\ell)\]
		for all $u,v,w\in \{1,2,3,4\}$. Then it suffices to prove 
		\[\Arg(d_1\rightarrow d_2;\ell)=\Arg(d_3\rightarrow d_4;\ell).\]
		Recall our notation for $d_j$ and $a_j$ in \eqref{Vrc dj choice jbetac' p=5}. Since $a_3-a_1=a_4-a_2=2\beta c'$, one can show $\Arg_1(d_1\rightarrow d_2;\ell)=\Arg_1(d_3\rightarrow d_4;\ell)$ by
		\[\sgn\(\sin(\tfrac{\pi a_3\ell}5)/\sin(\tfrac{\pi a_1 \ell }5 )\)=\sgn\(\sin(\tfrac{\pi a_4\ell}5)/\sin(\tfrac{\pi a_2 \ell}5 )\)=1. \]
		It is also easy to show $\Arg_3(d_1\rightarrow d_2;\ell)=\Arg_3(d_3\rightarrow d_4;\ell)$.

		For $\Arg_2$, we apply \eqref{Congru Dedekind mod theta c}, \eqref{Congru Dedekind c odd} and \eqref{Congru Dedekind c even} with the Chinese Remainder Theorem to show
		\[12cs(d_2,c)-12cs(d_1,c)\equiv 12cs(d_4,c)-12cs(d_3,c) \Mod{24c} \]
        in the following cases. 
		
		When $\gcd(c,3)=1$, we recall $d_2-d_1=d_4-d_3=\beta c'$ and have
		\begin{align}
			12cs(d_2,c)-12cs(d_1,c)&\equiv d_2+a_3-d_1-a_1\equiv 3\beta c'\Mod c,\label{Preproof 12cs d2-d1, 3 nmid c, mod c}\\
			12cs(d_4,c)-12cs(d_3,c)&\equiv d_4+a_4-d_3-a_2\equiv 3\beta c'\Mod c,\label{Preproof 12cs d4-d3, 3 nmid c, mod c}\\
			12cs(d_2,c)-12cs(d_1,c)&\equiv 12cs(d_4,c)-12cs(d_3,c)\equiv 0\Mod 6. \label{Preproof 12cs d2-d1=d4-d3, 3 nmid c, mod 6}
		\end{align}
		
		When $3|c$, we apply the congruence 
		\begin{equation}\label{Congruece inverse difference}
			\overline{(x+y)_{\{m\}} }-\overline{x_{\{m\}} }\equiv -y\overline{(x+y)_{\{m\}} }\cdot \overline{ x_{\{m\}}}\Mod m 
		\end{equation}
		to compute
		\begin{align*}
			d_2+\overline{d_{2\{3c\}}}-	d_1-\overline{d_{1\{3c\}}}\equiv \beta c'(1-\overline{d_{2\{3c\}}}\cdot \overline{d_{1\{3c\}}})\Mod {3c},\\
			d_4+\overline{d_{4\{3c\}}}-	d_3-\overline{d_{3\{3c\}}}\equiv \beta c'(1-\overline{d_{4\{3c\}}}\cdot\overline{d_{3\{3c\}}})\Mod {3c},
		\end{align*}
		which imply
		\begin{align*}
			12cs(d_2,c)-12cs(d_1,c)\equiv 12cs(d_4,c)-12cs(d_3,c)\equiv 0\Mod {c'}
		\end{align*}
		by \eqref{Congru Dedekind mod theta c}. 
		After dividing by $c'$ (recall that the denominator of $P_2(d)$ is $24c$), we have
		\begin{align}
			60s(d_2,c)-60s(d_1,c)\equiv \beta (1-\overline{d_{2\{3c\}}}\cdot \overline{d_{1\{3c\}}})\equiv \beta(1-a_3a_1) \Mod {15},\label{Preproof 12cs d2-d1, 3|c, mod 15}\\
			60s(d_4,c)-60s(d_3,c)\equiv \beta (1-\overline{d_{4\{3c\}}}\cdot \overline{d_{3\{3c\}}})\equiv \beta(1-a_4a_2) \Mod {15},\label{Preproof 12cs d4-d3, 3|c, mod 15}
		\end{align}
		because of \eqref{ajinverse5 dj =1} and $\overline{x_{\{un\}}}\equiv \overline{x_{\{vn\}}}\Mod n$. Since $a_3\equiv a_1\Mod 3$ and $a_4\equiv a_2\Mod 3$, we have $a_3a_1\equiv a_4a_2\equiv 1\Mod 3$. Moreover, $a_3a_1\equiv a_4a_2\equiv 3\Mod 5$. Hence $a_3a_1\equiv a_4a_2\equiv 13\Mod {15}$ and we get
		\begin{align}
			60s(d_2,c)-60s(d_1,c)\equiv 60s(d_4,c)-60s(d_3,c)\equiv 3\beta\Mod {15}. \label{Preproof 12cs d2-d1=d4-d3, 3|c, mod 15}
		\end{align}
		
		When $c$ is odd, by \eqref{Congru Dedekind c odd} and $d_{j_1}\equiv d_{j_2}\Mod {c'}$, we have
		\begin{align}
			12cs(d_2,c)-12cs(d_1,c)&\equiv 2(\tfrac {d_1}c)-2(\tfrac {d_2}c) \equiv 2(\tfrac 15)(\tfrac {d_1}{c'})-2(\tfrac 25)(\tfrac {d_1}{c'})\equiv 4\Mod 8, \label{Preproof 12cs d2-d1, c odd, mod 8}\\
			12cs(d_4,c)-12cs(d_3,c)&\equiv 2(\tfrac {d_3}c)-2(\tfrac {d_4}c) \equiv 2(\tfrac 35)(\tfrac {d_3}{c'})-2(\tfrac 45)(\tfrac {d_3}{c'})\equiv 4\Mod 8. \label{Preproof 12cs d4-d3, c odd, mod 8}
		\end{align}
		
		When $c$ is even and $2^\lambda\| c$ for $\lambda\geq 1$, by \eqref{Congru Dedekind c even} and \eqref{Congruece inverse difference} we have
		\begin{align*}
			12cs(d_2,c)-12cs(d_1,c)&\equiv d_2+(c^2+3c+1)\overline{d_{2\{8\times 2^\lambda\}}}+2c(\tfrac c{d_2})\overline{d_{2\{8\times 2^\lambda\}}}\\
			&-d_1-(c^2+3c+1)\overline{d_{1\{8\times 2^\lambda\}}}-2c(\tfrac c{d_1})\overline{d_{1\{8\times 2^\lambda\}}}\\
			&\equiv \beta c'(1-(c^2+3c+1)\overline{d_{2\{8\times 2^\lambda\}}}\cdot\overline{d_{1\{8\times 2^\lambda\}}})\\
			&+2c(\tfrac c{d_2})\overline{d_{2\{8\times 2^\lambda\}}}-2c(\tfrac c{d_1})\overline{d_{1\{8\times 2^\lambda\}}}\Mod {8\times 2^\lambda}. 
		\end{align*}
		Since $12cs(d_2,c)-12cs(d_1,c)$ is a multiple of $c'$ by the discussion of $\gcd(c,3)=1$ or $3|c$ above, by dividing $c'$ and by $x^2\equiv 1\Mod 8$ for odd $x$ we have
		\begin{align*}
			60s(d_2,c)-60s(d_1,c)\equiv \beta(1-(c^2+3c+1)d_2d_1)+2(\tfrac {c}{d_2})d_2-2(\tfrac{c}{d_1})d_1\Mod 8. 
		\end{align*}
		Similarly, $12cs(d_4,c)-12cs(d_3,c)$ is a multiple of $c'$ and
		\begin{align*}
			60s(d_4,c)-60s(d_3,c)\equiv \beta(1-(c^2+3c+1)d_4d_3)+2(\tfrac {c}{d_4})d_4-2(\tfrac{c}{d_3})d_3\Mod 8. 
		\end{align*}
		Dividing into cases for $4|c$ or $2\|c$ with $c'\equiv 2$ or $6\Mod 8$, one can conclude \[d_2d_1\equiv d_4d_3\Mod 8. \]
		For the remaining part, we only need to determine $(\frac c{d_j})d_j\equiv \pm 1\Mod 4$ for $j\in \{1,2,3,4\}$. Since $d_3\equiv d_1\Mod 4$ and $d_2\equiv d_4\Mod 4$, it is not hard to show that 
		\[(\tfrac c{d_2})d_2-(\tfrac c{d_1})d_1\equiv (\tfrac c{d_4})d_4-(\tfrac c{d_3})d_3\Mod 4. \]
        Now we have proved that when $c$ is even,
        \begin{align}
            60s(d_2,c)-60s(d_1,c)\equiv 60s(d_4,c)-60s(d_3,c)\Mod 8. \label{Preproof 12cs d2-d1=d4-d3, c even, mod 8}
        \end{align}
		
		Combining \eqref{Preproof 12cs d2-d1, 3 nmid c, mod c}, \eqref{Preproof 12cs d4-d3, 3 nmid c, mod c}, \eqref{Preproof 12cs d2-d1=d4-d3, 3 nmid c, mod 6}, \eqref{Preproof 12cs d2-d1=d4-d3, 3|c, mod 15}, \eqref{Preproof 12cs d2-d1, c odd, mod 8}, \eqref{Preproof 12cs d4-d3, c odd, mod 8}, \eqref{Preproof 12cs d2-d1=d4-d3, c even, mod 8}, we have shown
		\[\Arg_2(d_1\rightarrow d_2;\ell)=\Arg_2(d_3\rightarrow d_4;\ell)\quad \text{for }\ell\in \{1,2\}\]
		by proving 
		\[\frac{12cs(d_2,c)-12cs(d_1,c)}{24c}-\frac{12cs(d_4,c)-12cs(d_3,c)}{24c}\in \Z\]
		in all the cases for $c$ ($2|c$ or $2\nmid c$, $3|c$ or $3\nmid c$). The lemma follows. 
		
	\end{proof}

	Now we begin to prove that $\Arg(d_1\rightarrow d_4;\ell)$ and $\Arg(d_1\rightarrow d_2;\ell)$ both satisfy Condition~\ref{Styles four points mod 5} in all the cases of $5\|c$. 
	
	\subsection{Case \texorpdfstring{$2\nmid c'$, $3\nmid c'$, and $5\nmid c'$}{c' coprime to 2, 3 and 5}}
	
	We first treat the case when $c'\equiv 1,7,11,13,17,19,23,29\Mod {30}$. Recall our notations in \eqref{Vrc dj choice jbetac' p=5} and \eqref{ajinverse5 dj =1}: 
	\[d_4=d_1+3\beta c',\ \ d_2=d_1+\beta c', \ \ a_4=a_1+3\beta c', \ \ a_3=a_1+2\beta c',\ \ \beta c'\equiv 1\Mod 5. \]
	
	The argument differences $\Arg_j(d_1\rightarrow d_4;\ell)$ for $j=1,2,3$ are given by the arguments of 
	\[\frac{e\(-\frac 9{10} \beta c'^2\ell^2\)}{\sgn\(\sin(\frac{\pi a_4\ell}5)/\sin(\frac{\pi a_1 \ell}5 )\)}, \quad e\(-\frac{12cs(d_4,c)-12cs(d_1,c)}{24c}\),\quad \text{and }e\(\frac{2\beta}5\), \]
	respectively. First we have
	\begin{equation}\label{sign change for sin term, d4-d1}
		\sgn\(\sin(\tfrac{\pi a_4\ell}5)/\sin(\tfrac{\pi a_1 \ell}5 )\)=-1\quad \text{whenever }
		\left\{\begin{array}{l}
			\ell=1 \quad \text{and}\\ 
			3\beta c'\equiv 8\Mod {10} \end{array}\right.
		\text{ or }\ell=2. 
	\end{equation}
	This is easy to prove because $3\beta c'\times 2\equiv 6\Mod {10}$. 
	
	By \eqref{Congru Dedekind mod theta c}, we have $\theta=1$ and 
	\begin{equation}\label{12cs d4-d1 in case 1 7 11 13, mod c}
		-12cs(d_4,c)+12cs(d_1,c)\equiv -d_4-a_4+d_1+a_1\equiv -6\beta c'\equiv -\beta c'\Mod c. 
	\end{equation}
	Moreover, we have $-12cs(d_4,c)+12cs(d_1,c)\equiv 0\Mod 6$ and
	\[-12cs(d_4,c)+12cs(d_1,c)\equiv 2\((\tfrac {d_4}c)-(\tfrac {d_1}c)\)\equiv 2\((\tfrac {d_4}{5})(\tfrac {d_4}{c'})-(\tfrac {d_1}{5})(\tfrac {d_1}{c'})\)\equiv 0\Mod 8. \]
	Here we have used $(\frac{d_j}{5})=1$ for $j=1,4$ and $d_j\equiv d_1\Mod{c'}$ for all $j$. Then, 
	\begin{equation}\label{12cs d4-d1 in case 1 7 11 13, mod 24}
		-12cs(d_4,c)+12cs(d_1,c)\equiv 0\Mod{24}. 
	\end{equation}
	Combining \eqref{12cs d4-d1 in case 1 7 11 13, mod c} and \eqref{12cs d4-d1 in case 1 7 11 13, mod 24}, since $c'$ is odd, we can divide both the denominator and numerator in $P_2$ by $24c'$. We obtain
	\[\Arg_2(d_1\rightarrow d_4;\ell)=\frac \beta 5. \]
	
	Now we have Table~\ref{table: d4-d1 2n 3n 5n}. In the row of $\Arg_1(d_1\rightarrow d_4;1)$, we see $+\frac12$ because the sign difference $\sgn\(\sin(\frac{\pi a_4\ell}5)/\sin(\frac{\pi a_1 \ell}5 )\)=-1$ when $3\beta c'\equiv 8\Mod{10}$. The $\Arg_1(d_1\rightarrow d_4;2)$ contains the term $+\frac12$ because $3\beta c'\times 2\equiv 6\Mod {10}$. The upper half of the table is for the case $\ell=1$ and the lower half is for $\ell=2$. 
	\begin{table}[!htbp]
		\centering 
		\begin{tabular}{|c|c|c|c|c|c|c|c|c|}
			\hline
			$c'\Mod{30}$ & 1 & 7 & 11 & 13 & 17 & 19 & 23 & 29 \\
			$\beta$ & 1 & 3 & 1 & 2 & 3 & 4 & 2 & 4 \\
			$3\beta c'\Mod {10 }$ & 3 & 3 & 3 & \boxed{8} & 3 & \boxed{8} & \boxed{8} & \boxed{8}\\
			$-9\beta c'^2\Mod {10}$ & 1 & 7 & 1 & 8 & 7 & 4 & 7 & 4\\
			\hline 	
			$\Arg_1(d_1\rightarrow d_4;1)$ & $\frac 1{10}$ & $-\frac 3{10}$  & $\frac 1{10}$ & $-\frac 2{10}+\frac12$ & $-\frac 3{10}$ & $-\frac 6{10}+\frac12$ & $-\frac 3{10}+\frac12$ & $-\frac 6{10}+\frac12$ \\
			[1ex]
			$\Arg_2(d_1\rightarrow d_4;1)$ & $\frac 15$ & $\frac 35$ & $\frac 15$ & $\frac 25$ & $\frac 35$ & $\frac 45$ & $\frac 25$ & $\frac 45$\\
			[1ex] 
			$\Arg_3(d_1\rightarrow d_4;1)$ & $\frac 25$ & $\frac 15$ & $\frac 25$ & $\frac 45$ & $\frac 15$ & $\frac 35$ & $\frac 45$ & $\frac 35$\\
			[1ex]
			Total $\Arg(d_1\rightarrow d_4;1)$ & $-\frac 3{10}$ & $\frac 12$ & $-\frac 3{10}$ & $\frac 12$ & $\frac 12$ & $\frac 3{10}$ & $\frac 12$ & $\frac 3{10}$\\
			\hhline{|=|=|=|=|=|=|=|=|=|}
			$3\beta c'\Mod {10 }$ & 3 & 3 & 3 & 8 & 3 & 8 & 8 & 8\\
			$-18\beta c'^2\equiv 2c'\Mod {5}$ & 2 & 4 & 2 & 1 & 4 & 3 & 1 & 3\\
			\hline 	
			$\Arg_1(d_1\rightarrow d_4;2):\frac12 +\frac{2c'}5$ & $-\frac1{10}$ & $\frac 3{10}$  & $-\frac 1{10}$ & $-\frac 3{10}$ & $\frac 3{10}$ & $\frac 1{10}$ & $-\frac 3{10}$ & $\frac 1{10}$ \\
			[1ex]
			$\Arg_2(d_1\rightarrow d_4;2)$ & $\frac 15$ & $\frac 35$ & $\frac 15$ & $\frac 25$ & $\frac 35$ & $\frac 45$ & $\frac 25$ & $\frac 45$\\
			[1ex] 
			$\Arg_3(d_1\rightarrow d_4;2)$ & $\frac 25$ & $\frac 15$ & $\frac 25$ & $\frac 45$ & $\frac 15$ & $\frac 35$ & $\frac 45$ & $\frac 35$\\
			[1ex]
			Total $\Arg(d_1\rightarrow d_4;2)$ & $\frac 12$ & $\frac 1{10}$ & $\frac 12$ & $-\frac 1{10}$ & $\frac 1{10}$ & $\frac 12$ & $-\frac 1{10}$ & $\frac 12$\\
			\hline
		\end{tabular}
		\vspace{0.5ex}
		\caption{Table for $\Arg(d_1\rightarrow d_4;\ell)$; $2\nmid c$, $3\nmid c$, $5\nmid c$.  }
		\label{table: d4-d1 2n 3n 5n}
	\end{table}
	
	For $\Arg_j(d_1\rightarrow d_2;\ell)$, recall $a_3d_2\equiv 1\Mod c$. The argument differences $\Arg_j(d_1\rightarrow d_2;\ell)$ for $j=1,2,3$ are given by
	\[\frac{e\(-\frac 3{5} \beta c'^2\ell^2\)}{\sgn\(\sin(\frac{\pi a_3\ell}5)/\sin(\frac{\pi a_1 \ell}5 )\)}, \quad e\(-\frac{12cs(d_2,c)-12cs(d_1,c)}{24c}\),\quad e\(\frac{4\beta}5\),\]
	respectively. Since $2\beta c'\ell\equiv 2\ell\Mod{10}$, we always have
	\begin{equation}\label{sign change for sin term, d2-d1}
		\sgn\(\sin(\tfrac{\pi a_3}5)/\sin(\tfrac{\pi a_1 }5 )\)=1 \quad\text{and}\quad  \sgn\(\sin(\tfrac{2\pi a_3}5)/\sin(\tfrac{2\pi a_1 }5 )\)=-1. 
	\end{equation} 
	
	Moreover, from \eqref{Congru Dedekind mod theta c} we have
\begin{align}
    \label{12cs d2-d1 in case 1 7 11 13, mod c}
    12cs(d_2,c)-12cs(d_1,c)&\equiv d_2+a_3-d_1-a_1\equiv 3\beta c'\Mod c,  \\
    12cs(d_2,c)-12cs(d_1,c)&\equiv -2\(-(\tfrac{d_2}{c'})-(\tfrac{d_1}{c'})\)\equiv 4\Mod 8,  \\
    12cs(d_2,c)-12cs(d_1,c)&\equiv 0\Mod 6, \quad \text{and}\\
    \label{12cs d2-d1 in case 1 7 11 13, mod 24}
    12cs(d_2,c)-12cs(d_1,c)&\equiv 12\Mod{24}.  
\end{align}
	Combining \eqref{12cs d2-d1 in case 1 7 11 13, mod c} and \eqref{12cs d2-d1 in case 1 7 11 13, mod 24}, we divide by $c'$ and determine the unique value modulo $120$: 
	\[-(60s(d_2,c)-60s(d_1,c))\text{ \ \ congruent to \ \ }-3\beta\Mod 5\text{\ \ and \ \ }12\Mod{24}.  \]
	This gives the contribution of the argument difference from $P_2$. Now we can make Table~\ref{table: d2-d1 2n 3n 5n}. 
	
	\begin{table}[!htbp]
		\centering 
		\begin{tabular}{|c|c|c|c|c|c|c|c|c|}
			\hline
			$c'\Mod{30}$ & 1 & 7 & 11 & 13 & 17 & 19 & 23 & 29 \\
			$\beta$ & 1 & 3 & 1 & 2 & 3 & 4 & 2 & 4 \\
			$-3c'\Mod 5$ & 2 & 4 & 2 & 1 & 4 & 3 & 1 & 3 \\
			\hline 	
			$\Arg_1(d_1\rightarrow d_2;1)$ & $\frac 25$ & $\frac 45$ & $\frac 25$ & $\frac 15$ & $\frac 45$ & $\frac 35$ & $\frac 15$ & $\frac 35$\\
			[1ex]
			$\Arg_2(d_1\rightarrow d_2;1)$ & $\frac 1{10}$ & $\frac 3{10}$ & $\frac 1{10}$ & $-\frac 3{10}$ & $\frac 3{10}$ & $-\frac 1{10}$ & $-\frac 3{10}$ & $-\frac 1{10}$\\
			[1ex] 
			$\Arg_3(d_1\rightarrow d_2;1)$ & $\frac 45$ & $\frac 25$ & $\frac 45$ & $\frac 35$ & $\frac 25$ & $\frac 15$ & $\frac 35$ & $\frac 15$\\
			[1ex]
			Total $\Arg(d_1\rightarrow d_2;1)$ & $\frac 3{10}$ & $\frac 12$ & $\frac 3{10}$ & $\frac 12$ & $\frac 12$ & $-\frac 3{10}$ & $\frac 12$ & $-\frac 3{10}$\\
			\hhline{|=|=|=|=|=|=|=|=|=|}
			$-3c'\times 4\Mod 5$ & 3 & 1 & 3 & 4 & 1 & 2 & 4 & 2 \\
			\hline 	
			$\Arg_1(d_1\rightarrow d_2;2): \frac12-\frac{12c'}5 $ & $ \frac 1{10}$ & $ -\frac 3{10}$ & $\frac 1{10}$ & $\frac 3{10}$ & -$\frac 3{10}$ & $-\frac 1{10}$ & $\frac 3{10}$ & $-\frac 1{10}$\\
			[1ex]
			$\Arg_2(d_1\rightarrow d_2;2)$ & $\frac 1{10}$ & $\frac 3{10}$ & $\frac 1{10}$ & $-\frac 3{10}$ & $\frac 3{10}$ & $-\frac 1{10}$ & $-\frac 3{10}$ & $-\frac 1{10}$\\
			[1ex] 
			$\Arg_3(d_1\rightarrow d_2;2)$ & $\frac 45$ & $\frac 25$ & $\frac 45$ & $\frac 35$ & $\frac 25$ & $\frac 15$ & $\frac 35$ & $\frac 15$\\
			[1ex]
			Total $\Arg(d_1\rightarrow d_2;2)$ & $0$ & $\frac 25$ & $0$ & $-\frac 25$ & $\frac 25$ & $0$ & $-\frac 25$ & $0$\\
			\hline
		\end{tabular}
		\vspace{0.5ex}
		\caption{Table for $\Arg(d_1\rightarrow d_2;\ell)$; $2\nmid c$, $3\nmid c$, $5\nmid c$.  }
		\label{table: d2-d1 2n 3n 5n}
	\end{table}
	
	Combining Table \ref{table: d4-d1 2n 3n 5n} and Table~\ref{table: d2-d1 2n 3n 5n}, we see that $\Arg(d_1\rightarrow d_4;\ell)$ and $\Arg(d_1\rightarrow d_2;\ell)$ for $\ell=1,2$ satisfy the styles in Condition~\ref{Styles four points mod 5}. This finishes the proof when $2\nmid c'$, $3\nmid c'$ and $5\nmid c'$.

	\subsection{Case \texorpdfstring{$2\nmid c'$, $3| c'$, and $5\nmid c'$}{c' is coprime to 10 but divisible by 3}}
    
	These are the cases when $c'\equiv 3,9,21,27\Mod {30}$. The process is similar as the former subsection, so we just give Table~\ref{table: d4-d1 2n 3y 5n} and Table~\ref{table: d2-d1 2n 3y 5n} below for the values of argument differences in each case. The details in calculation can be found in the author's GitHub repository \cite{QihangKLsumsGitHub}.

%    For $\Arg_1(d_1\rightarrow d_4;\ell)$ we use \eqref{sign change for sin term, d4-d1}.
%	For $\Arg_2(d_1\rightarrow d_4;\ell)$, we need the congruence \eqref{Congruece inverse difference}. By \eqref{Congru Dedekind mod theta c}, we have
%	\begin{equation}\label{12cs d4-d1 in case 3 9 21 29, mod 3c}
%		12cs(d_4,c)-12cs(d_1,c)\equiv d_4+\overline{ d_{4\{3c\}}}-d_1-\overline{ d_{1\{3c\}}}\equiv 3\beta c'(1-\overline{ d_{4\{3c\}}}\cdot\overline{ d_{1\{3c\}}})\Mod{3c}.   
%	\end{equation}
%	By \eqref{Congru Dedekind c odd} we also have
%	\begin{equation}\label{12cs d4-d1 in case 3 9 21 29, mod 8}
%		12cs(d_4,c)-12cs(d_1,c)\equiv 0\Mod 8.   
%	\end{equation}
%	Dividing the numerator and denominator of $P_2$ by $24c'$, we observe that
%	\begin{equation}\label{12cs d4-d1 in case 3 9 21 29, mod c then mod 5}
%		-\tfrac 52(s(d_4,c)-s(d_1,c))\equiv -\overline{8_{\{5\}}}\beta(1-\overline{ d_{4\{3c\}}}\cdot\overline{ d_{1\{3c\}}})\equiv \beta \Mod 5
%	\end{equation} 
%	because $\overline{ d_{j\{3c\}}}\equiv \overline{d_{j\{5\}}}\equiv j\Mod 5$ for $j=1,4$. Now we get $\Arg_2(d_1\rightarrow d_4;\ell)=\frac \beta 5$. Since $\Arg_3(d_1\rightarrow d_4;\ell)=\frac{2\beta}5$, we have Table~\ref{table: d4-d1 2n 3y 5n}.

	\begin{table}[!htbp]
		\centering 
		\begin{tabular}{|c|c|c|c|c|}
			\hline
			$c'\Mod{30}$ & 3 & 9 & 21 & 27 \\
			$\beta$ & 2 & 4 & 1 & 3\\
			$3\beta c'\Mod {10 }$ & 8 & 8 & 3 & 3 \\
			$-9\beta c'^2\Mod {10}$ & 8 & 4 & 1 & 7 \\
			\hline 	
			$\Arg_1(d_1\rightarrow d_4;1)$ & $-\frac 2{10}+\frac12$ & $\frac 4{10}-\frac 12$  & $\frac 1{10}$ & $-\frac 3{10}$ \\
			[1ex]
			$(\Arg_2+\Arg_3)(d_1\rightarrow d_4;1)$: $\frac{3\beta}5$ & $\frac 15$ & $\frac 25$ & $\frac 35$ & $\frac 45$ \\
			[1ex]
			Total $\Arg(d_1\rightarrow d_4;1)$ & $\frac12$ & $\frac 3{10}$ & $-\frac 3{10}$ & $\frac 12$ \\
			\hhline{|=|=|=|=|=|}
			$3\beta c'\Mod {10 }$ & 8 & 8 & 3 & 3 \\
			$-18\beta c'^2\Mod {5}$ & 1 & 3 & 2 & 4 \\
			\hline 	
			$\Arg_1(d_1\rightarrow d_4;2)$ & $-\frac 3{10}$ & $\frac 1{10}$  & $-\frac 1{10}$ & $\frac 3{10}$ \\
			[1ex]
			$(\Arg_2+\Arg_3)(d_1\rightarrow d_4;2)$: $\frac{3\beta}5$ & $\frac 15$ & $\frac 25$ & $\frac 35$ & $\frac 45$ \\
			[1ex]
			Total $\Arg(d_1\rightarrow d_4;2)$ & $-\frac1{10}$ & $\frac 12$ & $\frac 12$ & $\frac 1{10}$ \\
			\hline
		\end{tabular}
		\vspace{0.5ex}
		\caption{Table for $\Arg(d_1\rightarrow d_4;\ell)$; $2\nmid c$, $3| c$, $5\nmid c$.  }
		\label{table: d4-d1 2n 3y 5n}
	\end{table}

%	Next we investigate $\Arg(d_1\rightarrow d_2;\ell)$. For $\Arg_1(d_1\rightarrow d_2;\ell)$, we use \eqref{sign change for sin term, d2-d1}.
%	For $\Arg_2(d_1\rightarrow d_2;\ell)$, by \eqref{Congru Dedekind mod theta c} we have
%	\begin{equation}\label{12cs d2-d1 in case 3 9 21 29, mod 3c}
%		12cs(d_2,c)-12cs(d_1,c)\equiv d_2+\overline{ d_{2\{3c\}}}-d_1-\overline{ d_{1\{3c\}}}\equiv \beta c'(1-\overline{ d_{2\{3c\}}}\cdot\overline{ d_{1\{3c\}}})\Mod{3c}.   
%	\end{equation}
%	As $15|3c$, after dividing by $c'$ we have
%	\begin{equation}\label{12cs d2-d1 in case 3 9 21 29, mod 3c then mod 15}
%		60s(d_2,c)-60s(d_1,c)\equiv \beta(1-\overline{ d_{2\{15\}}}\cdot \overline{ d_{1\{15\}}})\equiv \beta(1-a_3a_1)\Mod{15}.   
%	\end{equation}
%	Since $a_3\equiv a_1\Mod 3$, we have $a_3a_1\equiv 1\Mod 3$. We also have $a_3a_1\equiv 3\Mod 5$ by \eqref{ajinverse5 dj =1}, then $a_3a_1\equiv 13\Mod {15}$ and 
%	\begin{equation}\label{12cs d2-d1 in case 3 9 21 29, mod 3c then mod 15, true}
%		-(60s(d_2,c)-60s(d_1,c))\equiv -3\beta\Mod{15}.   
%	\end{equation}
%	By \eqref{Congru Dedekind c odd} we have
%	\begin{equation}\label{12cs d2-d1 in case 3 9 21 29, mod 8}
%		12cs(d_4,c)-12cs(d_1,c)\equiv 4\Mod 8.   
%	\end{equation}
%	The congruences \eqref{12cs d2-d1 in case 3 9 21 29, mod 3c then mod 15, true} and \eqref{12cs d2-d1 in case 3 9 21 29, mod 8} determine a unique value modulo 120. 

	\begin{table}[!htbp]
		\centering 
		\begin{tabular}{|c|c|c|c|c|}
			\hline
			$c'\Mod{30}$ & 3 & 9 & 21 & 27 \\
			$\beta$ & 2 & 4 & 1 & 3\\
			$-3c'\Mod {5 }$ & 1 & 3 & 2 & 4 \\
			\hline 	
			$\Arg_1(d_1\rightarrow d_2;1)$ & $\frac 15$ & $\frac 35$  & $\frac 25$ & $\frac 45$ \\
			[1ex]
			$\Arg_2(d_1\rightarrow d_2;1)$ & $-\frac 3{10}$ & $-\frac 1{10}$ & $\frac 1{10}$ & $\frac 3{10}$ \\
			[1ex]
			$\Arg_3(d_1\rightarrow d_2;1)$ & $\frac 35$ & $\frac 15$ & $\frac 45$ & $\frac 25$ \\
			[1ex]
			Total $\Arg(d_1\rightarrow d_2;1)$ & $\frac12$ & $-\frac 3{10}$ & $\frac 3{10}$ & $\frac 12$ \\
			\hhline{|=|=|=|=|=|}
			$-12c'\Mod {5 }$ & 4 & 2 & 3 & 1 \\
			\hline 	
			$\Arg_1(d_1\rightarrow d_2;2)$ & $\frac 3{10}$ & $-\frac 1{10}$  & $\frac 1{10}$ & $-\frac 3{10}$ \\
			[1ex]
			$\Arg_2(d_1\rightarrow d_2;2)$ & $-\frac 3{10}$ & $-\frac 1{10}$ & $\frac 1{10}$ & $\frac 3{10}$ \\
			[1ex]
			$\Arg_3(d_1\rightarrow d_2;2)$ & $\frac 35$ & $\frac 15$ & $\frac 45$ & $\frac 25$ \\
			[1ex]
			Total $\Arg(d_1\rightarrow d_2;2)$ & $-\frac 25$ & $0$ & $0$ & $\frac 25$ \\
			\hline
		\end{tabular}
		\vspace{0.5ex}
		\caption{Table for $\Arg(d_1\rightarrow d_2;\ell)$; $2\nmid c$, $3| c$, $5\nmid c$.  }
		\label{table: d2-d1 2n 3y 5n}
	\end{table}
	
	%Combining Table~\ref{table: d4-d1 2n 3y 5n} and Table~\ref{table: d2-d1 2n 3y 5n} we finish the proof in the case $2\nmid c'$, $3| c'$ and $5\nmid c'$. 
	
	\subsection{Case \texorpdfstring{$2|c'$, $3\nmid c'$, and $5\nmid c'$}{c' is coprime to 15 but even}} 
	These are the cases $c'\equiv 2,4,8,14,16,22,26,28\Mod {30}$. For $\Arg_1(d_1\rightarrow d_4;\ell)$ we still use \eqref{sign change for sin term, d4-d1}.
	By \eqref{Congru Dedekind mod theta c}, $\theta=1$ and we still have
	\begin{equation}\label{12cs d4-d1 in case 2 4 8 14, mod c}
		-(12cs(d_4,c)-12cs(d_1,c))\equiv -(d_4+a_4-d_1-a_1)\equiv -6\beta c'\equiv -\beta c'\Mod c,  
	\end{equation}
	and $12cs(d,c)\equiv 0\Mod 6$. 
	Define the integer $\lambda\geq 1$ by $2^\lambda\|c$. To determine the value modulo $24c$, we need to determine it modulo $8\times 2^\lambda$. By \eqref{Congru Dedekind c even} we have
	\begin{align}
		\label{12cs d4-d1 in case 2 4 8 14, mod 8 larger}
		\begin{split}
			12cs(d_4,c)-12cs(d_1,c)&\equiv d_4-d_1+(c^2+3c+1)(\overline{d_{4\{8\times 2^\lambda\}}}-\overline{d_{1\{8\times 2^\lambda\}}})\\
			&+2c\(\overline{d_{4\{8\times 2^\lambda\}}}(\tfrac c{d_4})-\overline{d_{1\{8\times 2^\lambda\}}}(\tfrac c{d_1})\)\Mod{8\times 2^\lambda}\\
			&\equiv 3\beta c'\(1-(c^2+3c+1)\overline{d_{4\{8\times 2^\lambda\}} }\cdot \overline{d_{1\{8\times 2^\lambda\}}}\)\\
			&+2c\(\overline{d_{4\{8\times 2^\lambda\}}}(\tfrac c{d_4})-\overline{d_{1\{8\times 2^\lambda\}}}(\tfrac c{d_1})\)\Mod{8\times 2^\lambda}.  
		\end{split}
	\end{align}
	We claim that 
	\begin{equation}\label{12cs d4-d1 in case 2 4 8 14, mod 8 2 lambda}
		12cs(d_4,c)-12cs(d_1,c)\equiv 0\Mod{8\times 2^\lambda}.
	\end{equation}
	To see this, since $2^\lambda\|c'$ and $c'|(12cs(d_4,c)-12cs(d_1,c))$ by \eqref{12cs d4-d1 in case 2 4 8 14, mod c}, we divide \eqref{12cs d4-d1 in case 2 4 8 14, mod 8 larger} by $c'$ and obtain
	\begin{align*}
		60\(s(d_4,c)-s(d_1,c)\)&\equiv 3\beta\(1-(c^2+3c+1)d_4d_1\)+2\(d_4(\tfrac c{d_4})-d_1(\tfrac c{d_1})\)\\
		&\equiv 3\beta c'\(3\beta d_1-1\)(c'-1)+2\(d_4(\tfrac c{d_4})-d_1(\tfrac c{d_1})\)\Mod 8. 
	\end{align*}
	Define val$\defeq 3\beta c'\(3\beta d_1-1\)(c'-1)\Mod 8$ as the first part of the congruence above. Note that both $d_1$ and $c'-1$ are odd. We have Table~\ref{table: d4-d1 2y 3n 5n, inner case first part 4 mod 8} for val.
	\begin{table}[!htbp]
		\centering 
		\begin{tabular}{|c|c|c|c|c|}
			\hline
			$c'\Mod{5}$ & 1 & 2 & 3 & 4 \\
			$\beta$ & 1 & 3 & 2 & 4\\
			$3\beta c'$ & $3c'$ & $6c'$ & $9c'$ & $12c'$ \\
			$3\beta d_1-1\Mod 2$ & $3d_1-1$ & $6d_1-1$ & $9d_1-1$ & $12d_1-1$ \\
			\hline 	
			$2\|c$, $d_1\equiv 1\Mod 4$ & 4 & 4  & 0 & 0 \\
			$2\|c$, $d_1\equiv 3\Mod 4$  & 0 & 4  & 4 & 0 \\
			\hline
			$4|c$;  & 0 & 0 & 0 &0 \\
			\hline
		\end{tabular}
		\vspace{0.5ex}
		\caption{Table of val$\defeq 3\beta c'\(3\beta d_1-1\)(c'-1)\Mod 8$; $2| c$, no requirement for $(c,3)$, $5\nmid c$.  }
		\label{table: d4-d1 2y 3n 5n, inner case first part 4 mod 8}
	\end{table}
	
	For the second part we only need to determine $d_4(\tfrac c{d_4})-d_1(\tfrac c{d_1})\Mod 4$. When $4|c$, we get $(\frac{2^\lambda}{d_4})=(\frac{2^{\lambda}}{d_1})=1$. By quadratic reciprocity, 
	\begin{align*}
		d_4(\tfrac c{d_4})-d_1(\tfrac c{d_1})&\equiv d_1\((\tfrac 5{d_4}) \(\tfrac{c'/2^\lambda}{d_4}\)-(\tfrac 5{d_4})\(\tfrac{c'/2^\lambda}{d_1}\)\)\\
		&\equiv d_1\(\tfrac{d_1}{c'/2^\lambda}\)\((-1)^{(d_4-1)(\frac {c'}{2^{\lambda}} -1)/4}-(-1)^{(d_1-1)(\frac {c'}{2^{\lambda}} -1)/4}\)\equiv 0\Mod 4
	\end{align*}
	where the last equality follows since $\frac{d_4-1}2$ and $\frac{d_1-1}2$ have the same parity. This gives the last row in Table~\ref{table: d4-d1 2y 3n 5n, inner case first part 4 mod 8}. 
	
	When $2\|c$, recall that $d_4=d_1+3\beta c'$, from which
	\begin{equation}\label{12cs d4-d1 in case 2 4 8 14, mod 8, second part reciprocity}
		d_4(\tfrac c{d_4})-d_1(\tfrac c{d_1})\equiv (\tfrac{d_1}{c'/2})\((\tfrac{2}{d_4})(-1)^{(d_4-1)(\frac {c'}{2} -1)/4}d_4-(\tfrac{2}{d_1})(-1)^{(d_1-1)(\frac {c'}{2} -1)/4}d_1\) \Mod 4. 
	\end{equation}
	When $c'\equiv 2\Mod 8$, $\frac{c'/2-1}2$ is even and \eqref{12cs d4-d1 in case 2 4 8 14, mod 8, second part reciprocity} becomes $(\frac {2}{d_4})d_4-(\frac{2}{d_1})d_1\Mod 4$; when $c'\equiv 6\Mod 8$, $\frac{c'/2-1}2$ is odd and \eqref{12cs d4-d1 in case 2 4 8 14, mod 8, second part reciprocity} becomes $(\frac {2}{d_4})(-1)^{\frac{d_4-1}2}d_4-(\frac{2}{d_1})(-1)^{\frac{d_1-1}2}d_1\Mod 4$. Since $c=5c'\equiv c'\Mod 8$, we can use $d_4=d_1+3\beta c'$ to determine $d_4\Mod 8$ and get Table~\ref{table: d4-d1 2y 3n 5n, inner case second part 4 mod 8}. 
	\begin{table}[!htbp]
		\centering 
		\begin{tabular}{|c|cccc|cccc|}
			\hline
			\quad \eqref{12cs d4-d1 in case 2 4 8 14, mod 8, second part reciprocity}  $\searrow$ & \multicolumn{4}{c|}{$c'\equiv 2\Mod 8$} & \multicolumn{4}{c|}{$c'\equiv 6\Mod 8$} \\
			\hline
			$d_1\Mod 8$ & 1 & 3 & 5 & 7 & 1 & 3 & 5 & 7 \\
			\hline
			$\beta =1$ & 2 & 0 & 2 & 0 & 2 & 0 & 2 & 0 \\
			$\beta =2$ & 2 & 2 & 2 & 2 & 2 & 2 & 2 & 2 \\
			$\beta =3$ & 0 & 2 & 0 & 2 & 0 & 2 & 0 & 2 \\
			$\beta =4$ & 0 & 0 & 0 & 0 & 0 & 0 & 0 & 0 \\
			\hline 	
		\end{tabular}
		\vspace{0.5ex}
		\caption{Table for \eqref{12cs d4-d1 in case 2 4 8 14, mod 8, second part reciprocity}; $2| c$, no requirement for $(c,3)$, $5\nmid c$.  }
		\label{table: d4-d1 2y 3n 5n, inner case second part 4 mod 8}
	\end{table}
	
	Combining Table~\ref{table: d4-d1 2y 3n 5n, inner case first part 4 mod 8} and Table~\ref{table: d4-d1 2y 3n 5n, inner case second part 4 mod 8}, we prove \eqref{12cs d4-d1 in case 2 4 8 14, mod 8 2 lambda}. 
	Recall \eqref{Congru Dedekind theta is or not 3} and \eqref{12cs d4-d1 in case 2 4 8 14, mod c}, we divide both the denominator and numerator in $P_2$ by $24c'$ and get $\Arg_2(d_1\rightarrow d_4;\ell)=\frac{\beta }5$. 
	Since $\Arg_3(d_1\rightarrow d_4;\ell)=\frac{2\beta }5$, we have Table~\ref{table: d4-d1 2y 3n 5n}. 
	\begin{table}[!htbp]
		\centering 
		\begin{tabular}{|c|c|c|c|c|c|c|c|c|}
			\hline
			$c'\Mod{30}$ & 2 & 4 & 8 & 14 & 16 & 22 & 26 & 28 \\
			$\beta$ & 3 & 4 & 2 & 4 & 1 & 3 & 1 & 2 \\
			$3\beta c'\Mod {10 }$ & 8 & 8 & 8 & 8 & 8 & 8 & 8 & 8\\
			$-9\beta c'^2\Mod {10}$ & 2 & 4 & 8 & 4 & 6 & 2 & 6 & 8\\
			\hline 	
			$\Arg_1(d_1\rightarrow d_4;1)$ & $-\frac 3{10}$ & $-\frac 1{10}$  & $\frac 3{10}$ & $-\frac 1{10}$ & $\frac 1{10}$ & $-\frac 3{10}$ & $\frac 1{10}$ & $\frac 3{10}$ \\
			[1ex]
			$(\Arg_2+\Arg_3)(d_1\rightarrow d_4;1)$: $\frac {3\beta}5$ & $\frac 45$ & $\frac 25$ & $\frac 15$ & $\frac 25$ & $\frac 35$ & $\frac 45$ & $\frac 35$ & $\frac 15$\\
			[1ex]
			Total $\Arg(d_1\rightarrow d_4;1)$ & $\frac 12$ & $\frac 3{10}$ & $\frac12$ & $\frac 3{10}$ & $-\frac 3{10}$ & $\frac 12$ & $-\frac 3{10}$ & $\frac 12$ \\
			\hhline{|=|=|=|=|=|=|=|=|=|}
			$-18\beta c'^2\equiv 2c'\Mod {5}$ & 4 & 3 & 1 & 3 & 2 & 4 & 2 & 1\\
			\hline 	
			$\Arg_1(d_1\rightarrow d_4;2)$ & $\frac 3{10}$ & $\frac 1{10}$  & $-\frac 3{10}$ & $\frac 1{10}$ & $-\frac 1{10}$ & $\frac 3{10}$ & $-\frac 1{10}$ & $-\frac 3{10}$ \\
			[1ex]
			$(\Arg_2+\Arg_3)(d_1\rightarrow d_4;2)$: $\frac {3\beta}5$ & $\frac 45$ & $\frac 25$ & $\frac 15$ & $\frac 25$ & $\frac 35$ & $\frac 45$ & $\frac 35$ & $\frac 15$\\
			[1ex]
			Total $\Arg(d_1\rightarrow d_4;2)$ & $\frac 1{10}$ & $\frac 12$ & $-\frac1{10}$ & $\frac 12$ & $\frac 12$ & $\frac 1{10}$ & $\frac 12$ & $-\frac 1{10}$ \\
			\hline 
		\end{tabular}
		\vspace{0.5ex}
		\caption{Table for $\Arg(d_1\rightarrow d_4;2)$; $2|c$, $3\nmid c$, $5\nmid c$.  }
		\label{table: d4-d1 2y 3n 5n}
	\end{table}
	
	Next we deal with $\Arg(d_1\rightarrow d_2;\ell)$.  
	For $\Arg_1(d_1\rightarrow d_2;\ell)$, we still use \eqref{sign change for sin term, d2-d1}.
	By \eqref{Congru Dedekind mod theta c}, 
	\begin{equation}\label{12cs d2-d1 in case 2 4 8 14, mod c}
		-(12cs(d_2,c)-12cs(d_1,c))\equiv -(d_2+a_3-d_1-a_1)\equiv -3\beta c'\equiv 2\beta c'\Mod c. 
	\end{equation}
	This congruence shows that $12cs(d_2,c)-12cs(d_1,c)$ is divisible by $c'$. 
	Denote $\lambda$ by $2^\lambda\|c$. We claim that 
	\begin{equation}\label{12cs d2-d1 in case 2 4 8 14, mod 8 exact}
		-(12cs(d_2,c)-12cs(d_1,c))\equiv 4\times 2^\lambda  \Mod{8\times 2^\lambda}. 
	\end{equation}
	To prove \eqref{12cs d2-d1 in case 2 4 8 14, mod 8 exact}, we apply \eqref{Congru Dedekind c even} to get 
	\begin{align}
		\label{12cs d2-d1 in case 2 4 8 14, mod 8 larger}
		\begin{split}
			12cs(d_2,c)-12cs(d_1,c)	&\equiv \beta c'\(1-(c^2+3c+1)\overline{d_{4\{8\times 2^\lambda\}}}\cdot \overline{d_{1\{8\times 2^\lambda\}}}\)\\
			&+2c\(\overline{d_{2\{8\times 2^\lambda\}}}(\tfrac c{d_2})-\overline{d_{1\{8\times 2^\lambda\}}}(\tfrac c{d_1})\)\Mod{8\times 2^\lambda}.  
		\end{split}
	\end{align}
	Then as in \eqref{12cs d4-d1 in case 2 4 8 14, mod 8 larger}, we have
	\begin{equation}
		60s(d_2,c)-60s(d_1,c)\equiv  \beta c'(\beta d_1-1)(c'-1)+2\(d_2(\tfrac c{d_2})-d_1(\tfrac c{d_1})\)\Mod 8. 
	\end{equation}
	See Table~\ref{table: d2-d1 2y 3n 5n, inner case first part 4 mod 8} for the first term val$\defeq \beta c'\(\beta d_1-1\)(c'-1)\Mod 8$ and note that $d_1$ and $c'-1$ are both odd. 
	\begin{table}[!htbp]
		\centering 
		\begin{tabular}{|c|c|c|c|c|}
			\hline
			$c'\Mod{5}$ & 1 & 2 & 3 & 4 \\
			$\beta$ & 1 & 3 & 2 & 4\\
			$\beta c'$ & $c'$ & $3c'$ & $2c'$ & $4c'$ \\
			$\beta d_1-1$  & $d_1-1$ & $3d_1-1$ & $2d_1-1$ & $4d_1-1$ \\
			\hline 	
			$2\|c$, $d_1\equiv 1\Mod 4$ & 0 & 4  & 4 & 0 \\
			$2\|c$, $d_1\equiv 3\Mod 4$  & 4 & 0  & 4 & 0 \\
			\hline
			$4|c$ & 0 & 0 & 0 &0 \\
			\hline
		\end{tabular}
		\vspace{0.5ex}
		\caption{Table for val.$\defeq \beta c'\(\beta d_1-1\)(c'-1)\Mod 8$; $2| c$, no requirement for $(c,3)$, $5\nmid c$.  }
		\label{table: d2-d1 2y 3n 5n, inner case first part 4 mod 8}
	\end{table}
	
	For the second term $2\(d_2(\tfrac c{d_2})-d_1(\tfrac c{d_1})\)\Mod 8$, we argue as above using the quadratic reciprocity \eqref{quadratic reciprocity} and omit the details. Combining \eqref{Congru Dedekind theta is or not 3}, \eqref{12cs d2-d1 in case 2 4 8 14, mod c} and \eqref{12cs d2-d1 in case 2 4 8 14, mod 8 exact}, we have
	\[-(12cs(d_2,c)-12cs(d_1,c))\equiv 12\times 2^\lambda  \Mod{24\times 2^\lambda}.  \] 
	After dividing $c'$, $-60s(d_2,c)+60s(d_1,c)\Mod {120}$ is uniquely determined by $2\beta \Mod 5$ and $12\Mod{24}$. Hence
	\[\Arg_2(d_1\rightarrow d_2;\ell)=\frac{1,7,3,9}{10},\quad \text{for }\beta=1,2,3,4, \text{ respectively},\]
	and we get Table~\ref{table: d2-d1 2y 3n 5n}. 
	\begin{table}[!htbp]
		\centering 
		\begin{tabular}{|c|c|c|c|c|c|c|c|c|}
			\hline
			$c'\Mod{30}$ & 2 & 4 & 8 & 14 & 16 & 22 & 26 & 28 \\
			$\beta$ & 3 & 4 & 2 & 4 & 1 & 3 & 1 & 2 \\
			$2\beta c'\Mod {10 }$ & 2 & 2 & 2 & 2 & 2 & 2 & 2 & 2\\
			$-3\beta c'^2\equiv 2c'\Mod {5}$ & 4 & 3 & 1 & 3 & 2 & 4 & 2 & 1\\
			\hline 	
			$\Arg_1(d_1\rightarrow d_2;1)$ & $\frac 45$ & $\frac 35$ & $\frac 15$ & $\frac 35$ & $\frac 25$ & $\frac 45$ & $\frac 25$ & $\frac 15$\\
			[1ex]
			$\Arg_2(d_1\rightarrow d_2;1)$ & $\frac 3{10}$ & $-\frac 1{10}$ & $-\frac 3{10}$ & $-\frac 1{10}$ & $\frac 1{10}$ & $\frac 3{10}$ & $\frac 1{10}$ & $-\frac 3{10}$\\
			[1ex]
			$\Arg_3(d_1\rightarrow d_2;1)$ & $\frac 25$ & $\frac 15$ & $\frac 35$ & $\frac 15$ & $\frac 45$ & $\frac 25$ & $\frac 45$ & $\frac 35$\\
			[1ex]
			Total $\Arg(d_1\rightarrow d_2;1)$& $\frac 12$ & $-\frac 3{10}$ & $\frac12$ & $-\frac 3{10}$ & $\frac 3{10}$ & $\frac 12$ & $\frac 3{10}$ & $\frac 12$ \\
			\hhline{|=|=|=|=|=|=|=|=|=|}
			$-12\beta c'^2\equiv 3c'\Mod {5}$ & 1 & 2 & 4 & 2 & 3 & 1 & 3 & 4\\
			\hline 	
			$\Arg_1(d_1\rightarrow d_2;2)$ & $-\frac 3{10}$ & $-\frac 1{10}$ & $\frac 3{10}$ & $-\frac 1{10}$ & $\frac 1{10}$ & $-\frac 3{10}$ & $\frac 1{10}$ & $\frac 3{10}$\\
			[1ex]
			$\Arg_2(d_1\rightarrow d_2;2)$ & $\frac 3{10}$ & $-\frac 1{10}$ & $-\frac 3{10}$ & $-\frac 1{10}$ & $\frac 1{10}$ & $\frac 3{10}$ & $\frac 1{10}$ & $-\frac 3{10}$\\
			[1ex]
			$\Arg_3(d_1\rightarrow d_2;2)$ & $\frac 25$ & $\frac 15$ & $\frac 35$ & $\frac 15$ & $\frac 45$ & $\frac 25$ & $\frac 45$ & $\frac 35$\\
			[1ex]
			Total $\Arg(d_1\rightarrow d_2;2)$ & $\frac 25$ & $0$ & $-\frac 25$ & $0$ & $0$ & $\frac 25$ & $0$ & $-\frac 25$ \\
			\hline
		\end{tabular}
		\vspace{0.5ex}
		\caption{Table for $\Arg(d_1\rightarrow d_2;\ell)$; $2|c$, $3\nmid c$, $5\nmid c$.  }
		\label{table: d2-d1 2y 3n 5n}
	\end{table}
	
	Combining Table~\ref{table: d4-d1 2y 3n 5n} and Table~\ref{table: d2-d1 2y 3n 5n}, we confirm that Condition~\ref{Styles four points mod 5} is satisfied in these cases. 
	
	\subsection{Case \texorpdfstring{$2|c'$, $3|c'$, and $5\nmid c'$}{c' is coprime to 5 but divisible by 6}}
	These are the cases $c'\equiv 6,12,18,24\Mod {30}$. The process is quite similar as the former subsection. We list the values of argument differences in Table~\ref{table: d4-d1 2n 3y 5n} and Table~\ref{table: d2-d1 2n 3y 5n} below, where the details can also be found in \cite{QihangKLsumsGitHub}. 
    
%	For $\Arg_1(d_1\rightarrow d_4;\ell)$ we use \eqref{sign change for sin term, d4-d1}.
%	For $\Arg_2(d_1\rightarrow d_4;\ell)$, by \eqref{Congru Dedekind mod theta c} we have
%	\begin{equation}\label{12cs d4-d1 in case 6 12 18 24, mod 3c}
%		-(12cs(d_4,c)-12cs(d_1,c))\equiv -3\beta c'(1-\overline{d_{4\{3c\}}}\cdot \overline{d_{1\{3c\}}})\Mod{3c}. 
%	\end{equation}
%	The proof of \eqref{12cs d4-d1 in case 2 4 8 14, mod 8 2 lambda} in the former subsection still works for $3|c$. Then $-(12cs(d_4,c)-12cs(d_1,c))$ is a multiple of $24c'$. After dividing  both the denominator and numerator in $P_2$ and recalling $\overline{d_{j\{3c\}}}\equiv a_j\equiv j\Mod 5$ for $j=1,4$, we get $\Arg_2(d_1\rightarrow d_4;\ell)=e(\frac{\beta}5)$. This gives Table~\ref{table: d4-d1 2y 3y 5n}. 
	\begin{table}[!htbp]
		\centering 
		\begin{tabular}{|c|c|c|c|c|}
			\hline
			$c'\Mod{30}$ & 6 & 12 & 18 & 24 \\
			$\beta$ & 1 & 3 & 2 & 4\\
			$3\beta c'\Mod {10 }$ & 8 & 8 & 8 & 8 \\
			$-9\beta c'^2\Mod {10}$ & 6 & 2 & 8 & 4 \\
			\hline 	
			$\Arg_1(d_1\rightarrow d_4;1)$ & $\frac 1{10}$ & $-\frac 3{10}$  & $\frac 3{10}$ & $-\frac 1{10}$ \\
			[1ex]
			$(\Arg_2+\Arg_3)(d_1\rightarrow d_4;1): \frac{3\beta }5 $ & $\frac 35$ & $\frac 45$ & $\frac 15$ & $\frac 25$ \\
			[1ex]
			Total $\Arg(d_1\rightarrow d_4;1)$ & $-\frac 3{10}$ & $\frac 12$ & $\frac 12$ & $\frac 3{10}$ \\
			\hhline{|=|=|=|=|=|}
			$-18\beta c'^2\equiv 2c'\Mod {5}$ & 2 & 4 & 1 & 3 \\
			\hline 	
			$\Arg_1(d_1\rightarrow d_4;2)$ & $-\frac 1{10}$ & $\frac 3{10}$  & $-\frac 3{10}$ & $\frac 1{10}$ \\
			[1ex]
			$(\Arg_2+\Arg_3)(d_1\rightarrow d_4;2): \frac{3\beta }5 $ & $\frac 35$ & $\frac 45$ & $\frac 15$ & $\frac 25$ \\
			[1ex]
			Total $\Arg(d_1\rightarrow d_4;2)$ & $\frac 12$ & $\frac 1{10}$ & $-\frac 1{10}$ & $\frac 12$ \\
			\hline 
		\end{tabular}
		\vspace{0.5ex}
		\caption{Table for $\Arg(d_1\rightarrow d_4;\ell)$; $2|c$, $3| c$, $5\nmid c$.  }
		\label{table: d4-d1 2y 3y 5n}
	\end{table}
	
%	Then we check $\Arg(d_1\rightarrow d_2;\ell)$. 
%	For $\Arg_1(d_1\rightarrow d_2;\ell)$, we use \eqref{sign change for sin term, d2-d1}.
%	For $\Arg_2(d_1\rightarrow d_2;\ell)$, by \eqref{Congru Dedekind mod theta c} we have
%	\begin{equation}\label{12cs d2-d1 in case 6 12 18 24, mod 3c}
%		-(12cs(d_2,c)-12cs(d_1,c))\equiv -\beta c'(1-\overline{d_{2\{3c\}}}\overline{d_{1\{3c\}}})\Mod{3c}. 
%	\end{equation}
%	Since $3|c$, $\overline{d_{2\{3c\}}}\equiv a_3\Mod {15}$ and $\overline{d_{1\{3c\}}}\equiv a_1\Mod{15}$. After dividing by $c'$ we have
%	\[
%	-(60s(d_2,c)-60s(d_1,c))\equiv -\beta (1-a_3a_1)\Mod{15}. 
%	\]
%	We have $a_3=a_1+2\beta c'$ and $a_3a_1\equiv 13\Mod{15}$, so
%	\begin{equation}\label{12cs d2-d1 in case 6 12 18 24, mod 15}
%		-(60s(d_2,c)-60s(d_1,c))\equiv -3\beta\Mod{15}. 
%	\end{equation}
%	Denote $\lambda$ by $2^\lambda\|c$, then \eqref{12cs d2-d1 in case 2 4 8 14, mod 8 exact} still holds since
%	\begin{equation}\label{12cs d2-d1 in case 6 12 18 24, mod 8 exact}
%		-(60s(d_2,c)-60s(d_1,c))\equiv 4  \Mod{8}. 
%	\end{equation}
%	By \eqref{12cs d2-d1 in case 6 12 18 24, mod 15} and \eqref{12cs d2-d1 in case 6 12 18 24, mod 8 exact}, we obtain
%	\[\Arg_2(d_1\rightarrow d_2;\ell)=\frac{1,7,3,9}{10}\quad \text{for }\beta=1,2,3,4,\text{ respectively.}\]
%	This gives Table~\ref{table: d2-d1 2y 3y 5n}. 
	\begin{table}[!htbp]
		\centering 
		\begin{tabular}{|c|c|c|c|c|}
			\hline
			$c'\Mod{30}$ & 6 & 12 & 18 & 24 \\
			$\beta$ & 1 & 3 & 2 & 4\\
			$-3\beta c'^2\equiv 2c'\Mod {5 }$ & 2 & 4 & 1 & 3 \\
			\hline 	
			$\Arg_1(d_1\rightarrow d_2;1)$ & $\frac 25$ & $\frac 45$  & $\frac 15$ & $\frac 35$ \\
			[1ex]
			$\Arg_2(d_1\rightarrow d_2;1)$ & $\frac 1{10}$ & $\frac 3{10}$ & $-\frac 3{10}$ & $-\frac 1{10}$ \\
			[1ex]
			$\Arg_3(d_1\rightarrow d_2;1)$ & $\frac 45$ & $\frac 25$ & $\frac 35$ & $\frac 15$ \\
			[1ex]
			Total $\Arg(d_1\rightarrow d_2;1)$ & $\frac 3{10}$  & $\frac 12$  & $\frac 12$  &  $-\frac 3{10}$ \\
			\hhline{|=|=|=|=|=|}
			$-12\beta c'^2\equiv 3c'\Mod {5}$ & 3 & 1 & 4 & 2 \\
			\hline
			$\Arg_1(d_1\rightarrow d_2;2)$ & $\frac1{10}$ & $-\frac 3{10}$  & $\frac 3{10}$ & $-\frac 1{10}$ \\
			[1ex]
			$\Arg_2(d_1\rightarrow d_2;2)$ & $\frac 1{10}$ & $\frac 3{10}$ & $-\frac 3{10}$ & $-\frac 1{10}$ \\
			[1ex]
			$\Arg_3(d_1\rightarrow d_2;2)$ & $\frac 45$ & $\frac 25$ & $\frac 35$ & $\frac 15$ \\
			[1ex]
			Total $\Arg(d_1\rightarrow d_2;2)$ & $0$  & $\frac 25$  & $-\frac 25$  &  $0$ \\
			\hline
		\end{tabular}
		\vspace{0.5ex}
		\caption{Table for $\Arg(d_1\rightarrow d_2;\ell)$; $2| c$, $3| c$, $5\nmid c$.  }
		\label{table: d2-d1 2y 3y 5n}
	\end{table}
	
%	Comparing Table~\ref{table: d4-d1 2y 3y 5n} and Table~\ref{table: d2-d1 2y 3y 5n}, we have proved that Condition~\ref{Styles four points mod 5} is satisfied in these cases. 
	
	We have finished the proof of Proposition~\ref{Vrc sum equals 0} by proving that the four points $P(d)$ satisfy Condition~\ref{Styles four points mod 5} when $5\|c$. The next subsection is to prove Proposition~\ref{Vrc sum equals 0} in the case $25|c$, which is different from the former ones.
	
	\subsection{Case \texorpdfstring{$5|c'$}{c' is a multiple of 5}}
	
	We still denote $c'=c/5$ and $V(r,c)\defeq \{d\Mod c^*:\ d\equiv r\Mod {c'}\}$ for $r\Mod{c'}^*$. Now $|V(r,c)|=5$ and since $(d+c',c)=1$ when $(d,c)=1$, we can write $V(r,c)=\{d,d+c',d+2c',d+3c',d+4c'\}$ for $1\leq d<c'$ and $d\equiv r\Mod{c'}$. 
	
	We claim that Proposition~\ref{Vrc sum equals 0} is still true: 
	\begin{equation}{\label{Vrc sum equals 0, 5 divide c case}}
		\sum_{d\in V(r,c)} \frac{e\(-\frac{3c'a\ell^2}{10}\)}{\sin(\frac{\pi a\ell}{p})}e\(-\frac{12 cs(d,c)}{24c}\)e\(\frac{4d}c\)=0,
	\end{equation}
	but this time we have five summands. We prove \eqref{Vrc sum equals 0, 5 divide c case} by showing that there are only two possible configurations for the summands: 
	\begin{center}
		\includegraphics[scale=0.4]{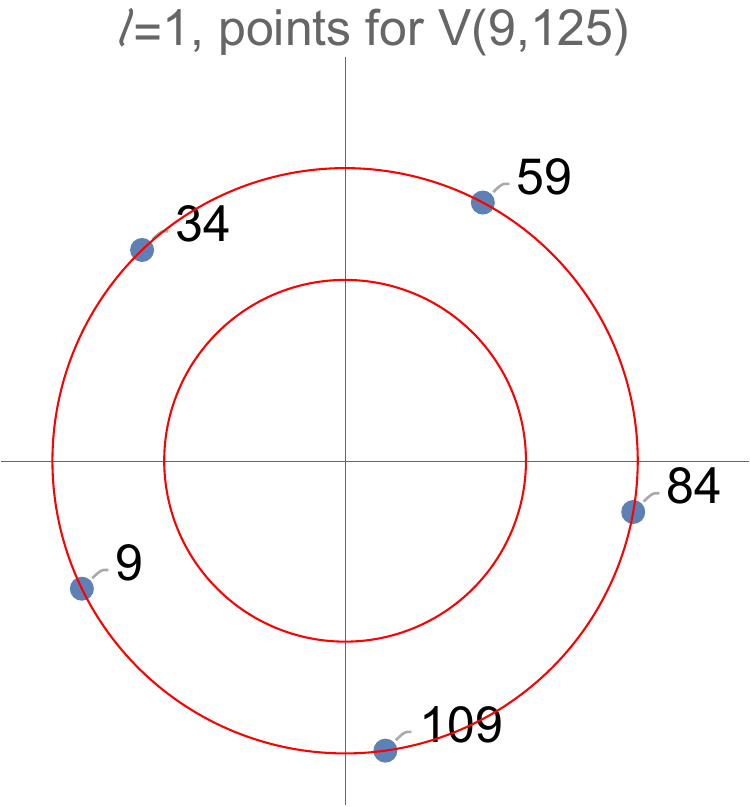}
		\includegraphics[scale=0.4]{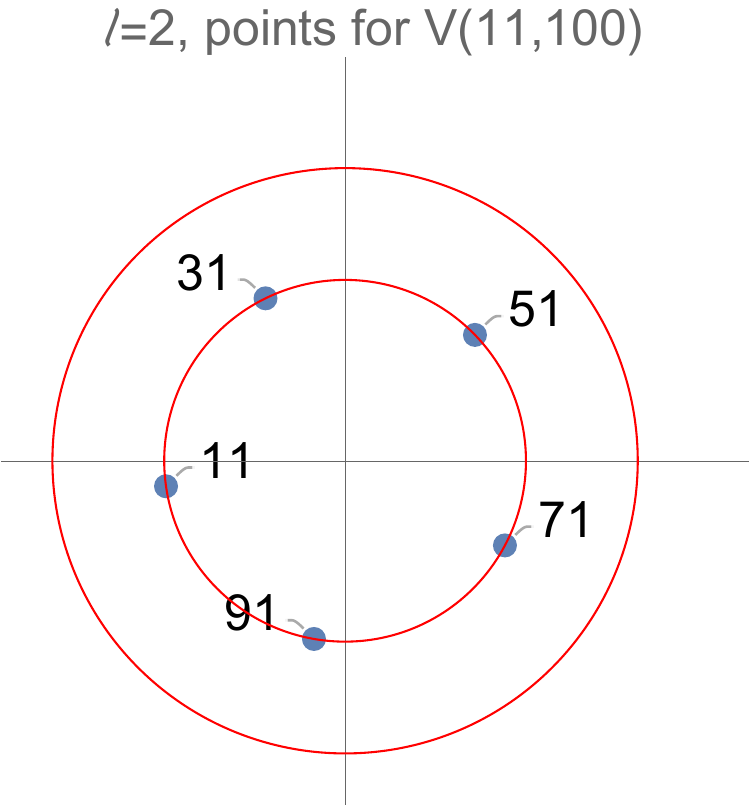}
	\end{center}
	i.e. all at the outer circle (radius $\csc(\frac{\pi }5)$) or all at the inner circle (radius $\csc(\frac{2\pi }5)$) and equally distributed. As in \eqref{P1P2P3}, we still denote the factors in \eqref{Vrc sum equals 0, 5 divide c case} by $P_1$, $P_2$ and $P_3$ and investigate the argument differences contributed from each term. 
	
	For any $d\in V(r,c)$, we take $a\Mod c$ such that $ad\equiv 1\Mod c$. We denote $d_*=d+c'$ and denote $a_*$ by $a_*d_*\equiv 1\Mod c$. Then we can pick $a_*=a-c'$ when $d\equiv 1,4\Mod 5$ and pick $a_*=a+c'$ when $d\equiv 2,3\Mod 5$. 
	
	Note that $P_1(d)=(-1)^{ca\ell}/{\sin(\frac{\pi a\ell}5)}$ has period $c'$, hence $\Arg_1(d\rightarrow d_*;\ell)=0$ always. In the following two cases, we prove 
	\begin{equation}\label{Arg diff d to d star, 5n+4}
		\Arg(d\rightarrow d_*;\ell)=-\tfrac 15\quad \text{for every }d\in V(r,c)
	\end{equation}
	when $\ell=1$. The other case $\ell=2$ only affects $P_1$ (radii for those five points) and results in the same conclusion. This proves \eqref{Vrc sum equals 0, 5 divide c case} when $25|c$. 
	
	\subsubsection{\texorpdfstring{$c$ is odd}{c is odd}}
	
	When $d\equiv 1,4\Mod 5$ and $3\nmid c$, \eqref{Congru Dedekind theta is or not 3}, \eqref{Congru Dedekind mod theta c} and \eqref{Congru Dedekind c odd} imply that
	\begin{equation}\label{12cs in case d1-d 3n 5y}
		12cs(d_*,c)-12cs(d,c)\equiv 0\Mod {24c},
	\end{equation}
	hence $\Arg_2(d\rightarrow d_*;\ell)=0$ always. As $\Arg_3(d\rightarrow d_*;\ell)=\frac 45$ for any $d\in V(r,c)$, we have proved \eqref{Vrc sum equals 0, 5 divide c case} in this case. 
	
	When $3|c$ and $d\equiv 1,4\Mod 5$, \eqref{Congru Dedekind mod theta c} implies
	\begin{equation}\label{12cs in case d1-d 3y 5y}
		-(12cs(d_*,c)-12cs(d,c))\equiv -c'(1-\overline{d_{*\{3c\}}}\cdot \overline{d_{\{3c\}}})\Mod{3c}.
	\end{equation} 
	Since $15|c$, after dividing by $c'$ we have
	\begin{equation}\label{12cs d1-d in case 15 30 45, mod 15, c odd}
		-(60s(d_*,c)-60s(d,c))\equiv a^2-1\Mod{15}. 
	\end{equation}
	Note that $a\equiv 1,4\Mod 5$ and $a^2\equiv 1\Mod{15}$, hence we have $-(12cs(d_*,c)-12cs(d,c))\equiv 0\Mod{24 c} $ and conclude \eqref{Arg diff d to d star, 5n+4}. 
	
	When $d\equiv 2,3\Mod 5$ and $3\nmid c$, recall $d_*=d+c'$ and $a_1=a+c'$ with $a+d\equiv 0\Mod 5$. By \eqref{Congru Dedekind theta is or not 3}, \eqref{Congru Dedekind mod theta c} and \eqref{Congru Dedekind c odd}, we have
	\[-(12cs(d_*,c)-12cs(d,c))\equiv -2c'\Mod {c}\quad \text{and}\quad \equiv 0\Mod{24}. \]
	Then $\Arg_2(d\rightarrow d_*;\ell)=\frac 25$. Since $\Arg_3(d\rightarrow d_*;\ell)=\frac 45$, we have proved \eqref{Arg diff d to d star, 5n+4} in this case. 
	
	When $d\equiv 2,3\Mod 5$ and $3|c$, we still get \eqref{12cs d1-d in case 15 30 45, mod 15, c odd}, while this time $a\equiv 3,2\Mod 5$,  $a^2-1\equiv 3\Mod{15}$, and hence $a^2-1\equiv 48\Mod {120}$. We have $\Arg_2(d\rightarrow d_*;\ell)=\frac 25$. Since $\Arg_3(d\rightarrow d_*;\ell)=\frac 45$, we have proved \eqref{Arg diff d to d star, 5n+4} in this case. 
	
	\subsubsection{\texorpdfstring{$c$ is even}{c is even}}
	In this case, denote $\lambda$ by $2^\lambda\|c$. Then by \eqref{Congru Dedekind c even} we have
	\begin{align*}
		12cs(d_*,c)-12cs(d,c)&\equiv c'\(1-(c^2+3c+1)\overline{d_{1\{8\times 2^\lambda\}}}\cdot \overline{d_{\{8\times 2^\lambda\}}}\)\\
		&+2c\((\tfrac{c}{d_*})\overline{d_{1\{8\times 2^\lambda\}}}-(\tfrac{c}{d})\overline{d_{\{8\times 2^\lambda\}}}\)\Mod{8\times 2^\lambda}. 
	\end{align*}
	Since $c'|(12cs(d_*,c)-12cs(d,c))$ by \eqref{12cs in case d1-d 3n 5y} and \eqref{12cs in case d1-d 3y 5y}, dividing the above congruence by $c'$ we have
	\begin{equation}\label{12cs in case d1-d 2y 3n 5y, mod 8 exact}
		-60(s(d_*,c)-s(d,c))\equiv -c'(d-1)(c'-1)-2\((\tfrac{c}{d_*})d_*-(\tfrac cd)d\)\Mod 8. 
	\end{equation}
	For the first term,
	\begin{equation}	\label{12cs in case d1-d 2y 3n 5y, mod 8, first part}
		-c'(d-1)(c'-1)\equiv \left\{
		\begin{array}{ll}
			0\Mod 8 &\ \text{if }2\|c,\ d\equiv 1\Mod 4;\\
			4\Mod 8 &\ \text{if }2\|c,\ d\equiv 3\Mod 4;\\
			0\Mod 8 &\ \text{if }4|c. \\
		\end{array}
		\right. 
	\end{equation}
	When $\lambda$ is even, $(\frac {2^\lambda }{d_*})=(\frac {2^\lambda }d)=1$; when $\lambda\geq 3$ is odd, $(\frac{2}{d_*})=(\frac{2}{d})$. In either case $\frac{d_*-1}2$ and ${\frac{d-1}2}$ have the same parity. Hence when $4|c$, we have
	\[(\tfrac{c}{d_*})d_*-(\tfrac cd)d\equiv 0\Mod 4. \]
	When $2\|c$, we have Table~\ref{table: d1-d 2y 3n 5y} for val$\defeq (\tfrac{c}{d_*})d_*-(\tfrac cd)d \Mod 4$ using quadratic reciprocity. 
	\begin{table}[!htbp]
		\centering
		\begin{tabular}{|c|c|c|c|c|}
			\hline
			$d\Mod 8$ & 1 & 3 & 5 & 7\\
			\hline
			$d_*\Mod 8$ when $c'\equiv 2\Mod 8$ & 3 & 5 & 7 & 1 \\
			val. & 0 & 2 & 0 & 2 \\
			\hline
			$d_*\Mod 8$ when $c'\equiv 6\Mod 8$ & 7 & 1 & 3 & 5 \\
			val. & 0 & 2 & 0 & 2 \\
			\hline
		\end{tabular}
		\vspace{0.5ex}
		\caption{Table for val$\defeq (\tfrac{c}{d_*})d_*-(\tfrac cd)d\Mod 4$; $2| c$, no requirement for $(3,c)$, $5|c$.  }
		\label{table: d1-d 2y 3n 5y}
	\end{table}
	
	Combining \eqref{12cs in case d1-d 2y 3n 5y, mod 8, first part} and Table~{\ref{table: d1-d 2y 3n 5y}}, for $2^\lambda\| c$ we get
	\begin{equation}
		12cs(d_*,c)-12cs(d,c)\equiv 0\Mod {8\times 2^\lambda}. 
	\end{equation}
	The argument for the cases $d\equiv 1,4\Mod 5$ or $d\equiv 2,3\Mod 5$, or the cases $3\nmid c$ or $3|c$, still works as the former case. 
	
	\begin{proof}[Proof of Proposition~\ref{Vrc sum equals 0}]
		This is proved by Condition~\ref{Styles four points mod 5} and \eqref{Arg diff d to d star, 5n+4}.
	\end{proof}
	
	This finishes the proof of (5-4) in Theorem~\ref{Kloosterman sums vanish}.

	\section{Proof of (7-5,1) of Theorem~\ref{Kloosterman sums vanish}}
	\label{Section 7-5,1}
	
	Recall \eqref{S infty infty for mod p} in the case $p=7$: 
	\begin{equation}\label{S infty infty for mod 7}
		e(\tfrac 18)S_{\infty\infty}^{(\ell)}(0,7n+5,c,\mu_7)=\sum_{\substack{d\Mod c^*\\ad\equiv 1\Mod c}} \frac{(-1)^{\ell c}e(-\frac{3c'a\ell^2}{14})}{\sin(\frac{\pi a\ell}7)}\,e^{-\pi is(d,c)}e\(\frac{(7n+5)d}c\).
	\end{equation} 
	We only need to consider $\ell=1,2,3$ because $A(\frac {\ell}p;n)=A(1-\frac {\ell}p;n)$. 
	
	As in the previous section, we define $c'\defeq c/7$. For an integer $r$ with $(r,c')=1$, we define
	\[V(r,c)=\{d\Mod c^*:\ d\equiv r\Mod {c'}\}. \]
	For example, $V(1,42)=\{1,13,19,25,31,37\}$ and $V(4,35)=\{4,9,19,24,29,34\}$. 
	Then $|V(r,c)|=6$ if $7\nmid c'$, $|V(r,c)|=7$ if $49|c$, and $(\Z/c\Z)^*$ is the disjoint union 
	\[(\Z/c\Z)^*=\bigcup_{r\Mod{c'}^*}V(r,c). \]

	We claim the following proposition. 
	\begin{proposition}{\label{Vrc sum equals 0, mod 7}}
		For $\ell=1,2,3$, when $7|c$, $\frac c7\cdot \ell\not\equiv 1\Mod 7$ and $\frac c7\cdot \ell\not\equiv -1\Mod 7$, the sum on $d\in V(r,c)$ for all $r\Mod {c'}^*$ is zero: 
		\begin{equation}\label{Vrc sum equals 0, mod 7 equation}
			s_{r,c}\defeq \sum_{d\in V(r,c)} \frac{e\(-\frac{3c'a\ell^2}{14}\)}{\sin(\frac{\pi a\ell}{7})}e\(-\frac{12 cs(d,c)}{24c}\)e\(\frac{5d}c\)=0
		\end{equation}
	\end{proposition}
	
	If Proposition~\ref{Vrc sum equals 0, mod 7} is true, then 
	\[S_{\infty\infty}^{(\ell)}(0,7n+5,c,\mu_7)=e(-\tfrac 18)(-1)^{\ell c}\sum_{r\Mod c'}s_{r,c}e\(\frac{nr}{c'}\)=0\]
	for all $n\in \Z$, $\ell=1,2,3$ and we have proved (7-5,1) of Theorem~\ref{Kloosterman sums vanish}. 
	
	As in \eqref{P1P2P3}, we label the terms in \eqref{Vrc sum equals 0, mod 7 equation} as
	\begin{equation}\label{P1P2P3 7n+5}
		P(d)\defeq \frac{e\(-\frac{3c'a\ell^2}{14}\)}{\sin(\frac{\pi a\ell}{7})}\cdot  e\(-\frac{12 cs(d,c)}{24c}\)\cdot e\(\frac{5d}c\) =:P_1(d)\cdot P_2(d)\cdot P_3(d).
	\end{equation}
	We first deal with the case $7\nmid c'$. We denote the argument differences as in \eqref{Arg def}, but in this case $u,v\in\{1,2,\cdots,6\}$ and $\ell\in\{1,2,3\}$, where 
	\begin{equation}\label{ajinverse7 dj =1}
		d_u\equiv a_u\equiv u\Mod 7,\ a_{\overline{u_{\{7\}}}}d_u\equiv 1\Mod c, \  d_{u+1}=d_u+\beta c'\text{ and }a_{u+1}=a_u+\beta c'.
	\end{equation}
	Note that $a_ud_u$ may not be $1\Mod c$. Let $1\leq \beta \leq 6$ such that $\beta c'\equiv 1 \Mod 7 $. 
	
	As in Condition~\ref{Styles four points mod 5}, we have the following styles for the six summands followed by the explanation in Condition~\ref{Styles six points mod 7}: 
	
	\textbullet\ $\ell=1,2,3$, first style. 
	\begin{center}
		
		\includegraphics[scale=0.35]{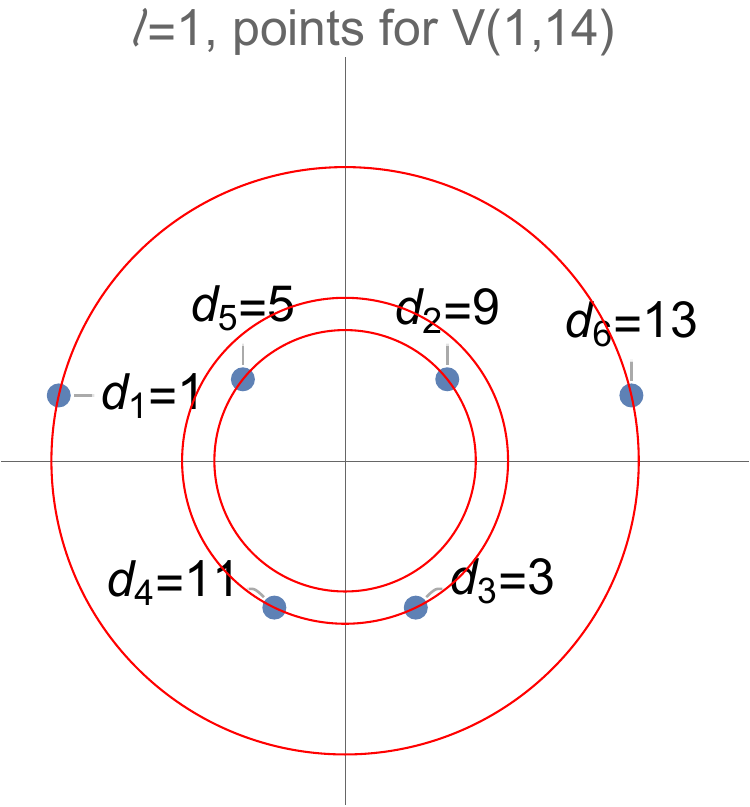}
		\includegraphics[scale=0.35]{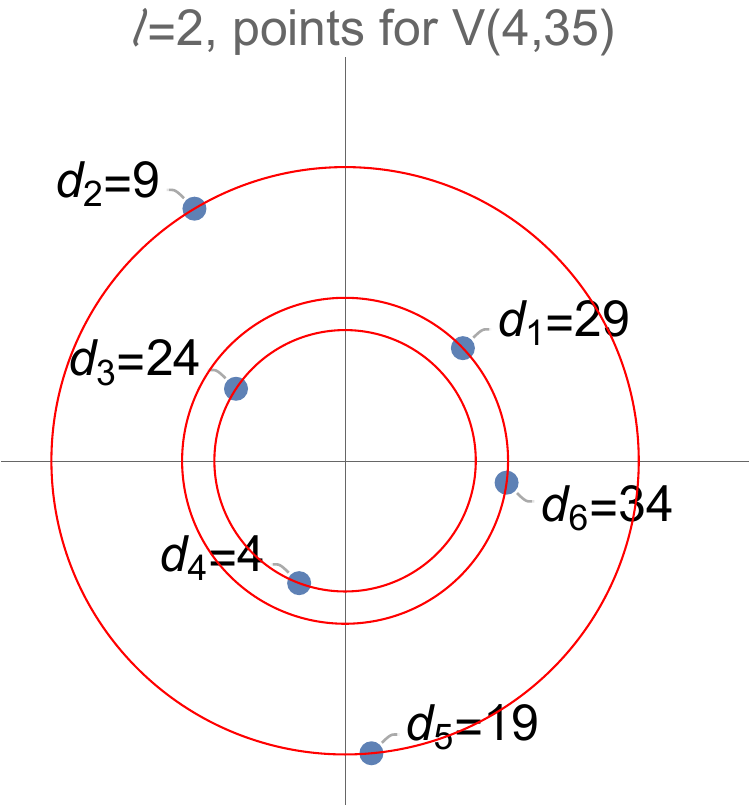}
		\includegraphics[scale=0.35]{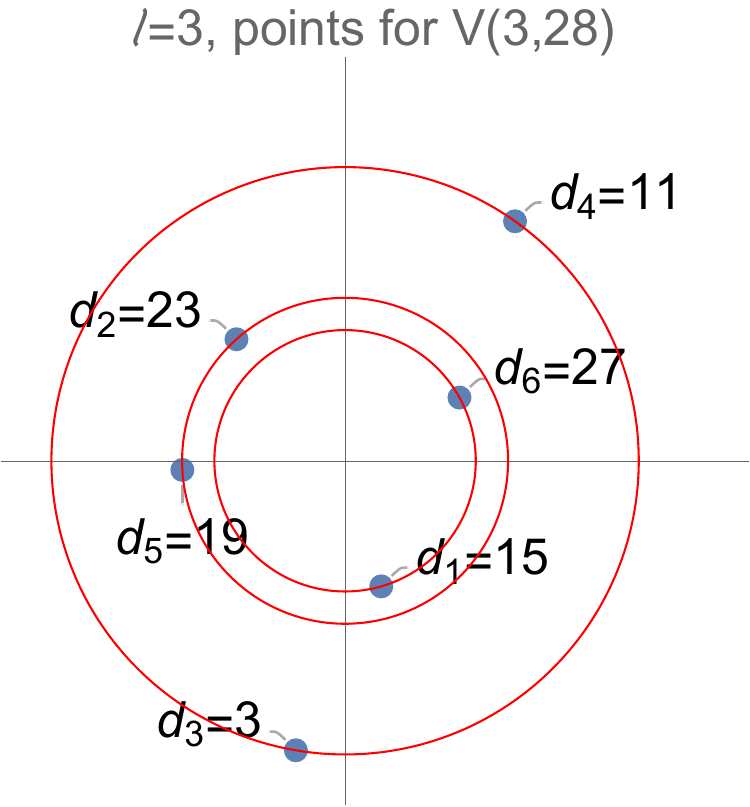}
		
	\end{center}
	
	\textbullet\ $\ell=1,2,3$, reversed style from the above.  
	\begin{center}
		\includegraphics[scale=0.35]{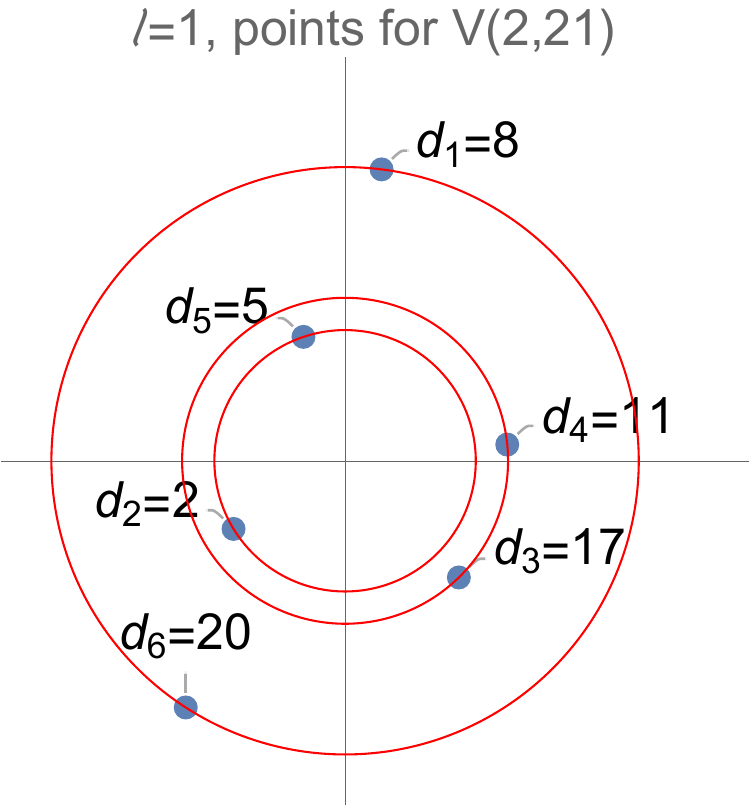}
		\includegraphics[scale=0.35]{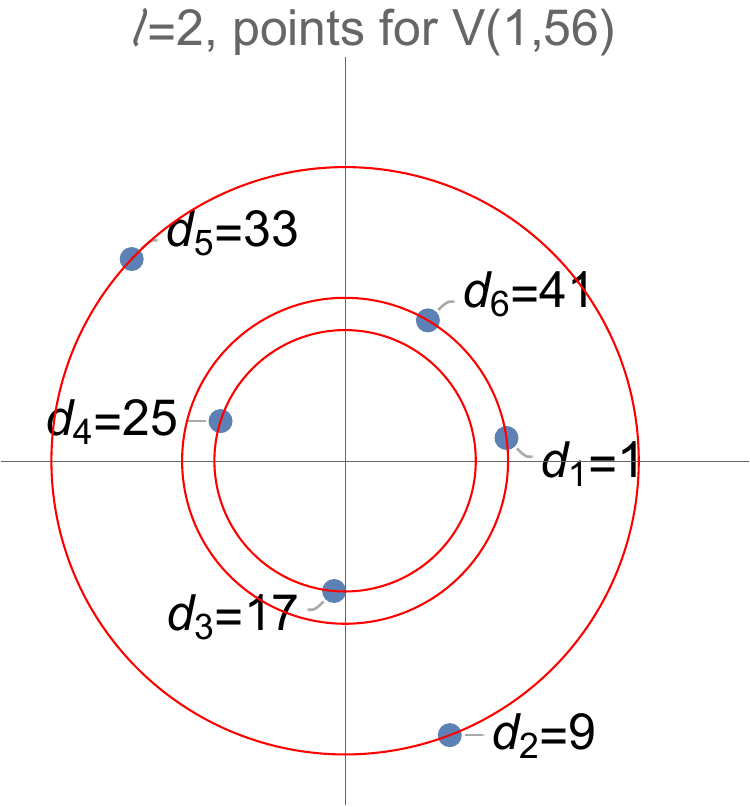}
		\includegraphics[scale=0.35]{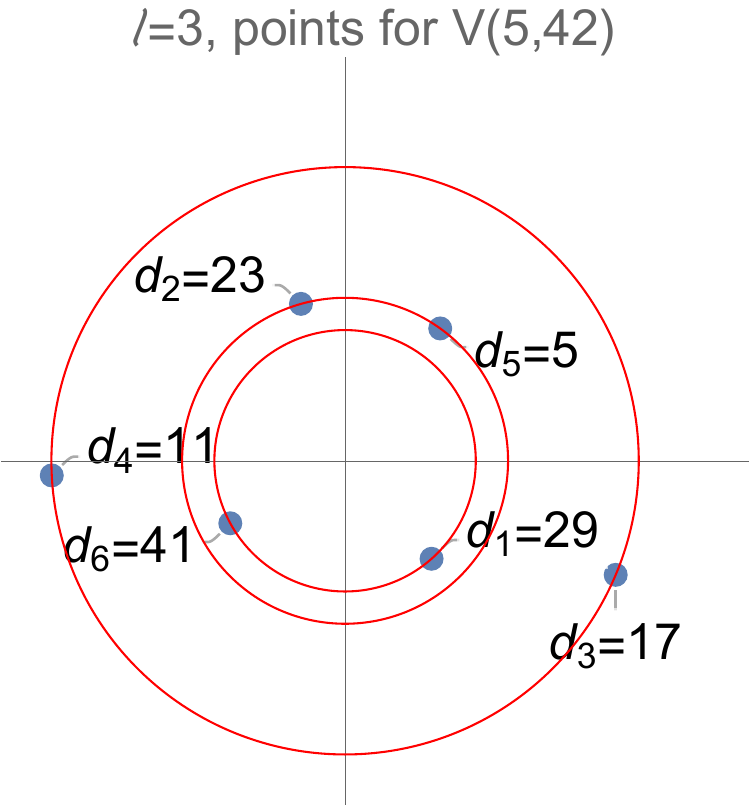}
	\end{center}
	
	Here we explain these styles. Each graph above includes three circles centered at the origin with radii $\csc(\frac \pi 7)$, $\csc(\frac{2\pi }7)$ and $\csc(\frac{3\pi}7)$, respectively. 
	The six points in each graph above mark $P(d)$ for $d\in V(r,c)$ on these three circles. It is not hard to prove that whenever the six points satisfy the following condition on their argument differences, they sum to zero. This proves Proposition~\ref{Vrc sum equals 0, mod 7} by using the equation
	\[\frac{\cos(\frac{3\pi}7)}{\sin(\frac{\pi}7)}-\frac{\cos(\frac{\pi}7)}{\sin(\frac{2\pi}7)}+\frac{\cos(\frac{2\pi}7)}{\sin(\frac{3\pi}7)}=0,\quad \text{where }\frac 1{\sin(\frac{\pi}7)},\ \frac 1{\sin(\frac{2\pi}7)},\ \frac 1{\sin(\frac{3\pi}7)}\text{ are the radii.}\]
	\begin{condition}\label{Styles six points mod 7}
		We have the following six styles for these six points when $7\|c$, $\frac c7\cdot \ell\not \equiv 1\Mod 7$ and $\frac c7\cdot\ell\not \equiv -1\Mod 7$. 
		\begin{itemize}
			\item $\ell=1$: the arguments (as a proportion of $2\pi$) going $d_1\rightarrow d_2\rightarrow d_3\rightarrow d_4\rightarrow d_5\rightarrow d_6\rightarrow d_1$ are $-\frac5{14}$, $-\frac2{7}$, $-\frac1{7}$, $-\frac2{7}$, $-\frac5{14}$, and $\frac3{7}$, respectively, or the reversed style. 
			\begin{center}
				\begin{tabular}{|c|ccccccccccccc|}
					\hline
					& $d_1$ & $\rightarrow$ & $d_2$ & $\rightarrow$ & $d_3$ & $\rightarrow$ & $d_4$ & $\rightarrow$ & $d_5$ & $\rightarrow$ & $d_6$ & $\rightarrow$ & $d_1$ \\
					\hline
					$c'\equiv 2,4\Mod 7$  & & $-\frac 5{14}$ & & $-\frac 2{7}$ & & $-\frac 1{7}$ & & $-\frac 2{7}$ & & $-\frac 5{14}$ & & $\frac 3{7}$& \ \\
					\hline
					$c'\equiv 3,5\Mod 7$  & & $\frac 5{14}$ & & $\frac 2{7}$ & & $\frac 1{7}$ & & $\frac 2{7}$ & & $\frac 5{14}$ & & $-\frac 3{7}$& \ \\ 
					\hline
				\end{tabular}
			\end{center}
			
			\item $\ell=2$, second graph style:  \begin{center}
				\begin{tabular}{|c|ccccccccccccc|}
					\hline
					& $d_1$ & $\rightarrow$ & $d_2$ & $\rightarrow$ & $d_3$ & $\rightarrow$ & $d_4$ & $\rightarrow$ & $d_5$ & $\rightarrow$ & $d_6$ & $\rightarrow$ & $d_1$ \\
					\hline
					$c'\equiv 5,6\Mod 7$ & & $\frac 3{14}$ & & $\frac 1{14}$ & & $\frac 2{7}$ & & $\frac 1{14}$ & & $\frac 3{14}$ & & $\frac 1{7}$& \ \\
					\hline
					$c'\equiv 1,2\Mod 7$ & & $-\frac 3{14}$ & & $-\frac 1{14}$ & & $-\frac 2{7}$ & & $-\frac 1{14}$ & & $-\frac 3{14}$ & & $-\frac 1{7}$& \ 
					\\
					\hline
				\end{tabular}
			\end{center}
			
			\item $\ell=3$, third graph style:
			\begin{center}
				\begin{tabular}{|c|ccccccccccccc|}
					\hline
					& $d_1$ & $\rightarrow$ & $d_2$ & $\rightarrow$ & $d_3$ & $\rightarrow$ & $d_4$ & $\rightarrow$ & $d_5$ & $\rightarrow$ & $d_6$ & $\rightarrow$ & $d_1$ \\
					\hline
					$c'\equiv 1,4\Mod 7$  & & $-\frac 3{7}$ & & $\frac 5{14}$ & & $\frac 3{7}$ & & $\frac 5{14}$ & & $-\frac 3{7}$ & & $-\frac 2{7}$& \ \\
					\hline
					$c'\equiv 3,6\Mod 7$ & & $\frac 3{7}$ & & $-\frac 5{14}$ & & $-\frac 3{7}$ & & $-\frac 5{14}$ & & $\frac 3{7}$ & & $\frac 2{7}$& \ \\
					\hline
				\end{tabular}
			\end{center}
		\end{itemize}

	\end{condition}

	\begin{proof}[Proof of Proposition~\ref{Vrc sum equals 0, mod 7} when $7\nmid c'$]
		This is proved by Condition~\ref{Styles six points mod 7}. 	
	\end{proof}
%	\begin{remark}
%		Note that (7-5,1) of Theorem~\ref{Kloosterman sums vanish} is for the case $c'\ell\not \equiv \pm 1\Mod{7}$, so Condition~\ref{Styles six points mod 7} does not include all the cases of $c'\Mod 7$. We will highlight these exceptional cases among the tables in this section by a row ``$c'\ell\equiv \pm 1\Mod 7$?". The corresponding entry is: 
%		\[\left\{\begin{array}{ll}
%			\text{blank, } &\text{if }c'\ell\not\equiv \pm 1\Mod 7; \\
%			``+", &\text{if }c'\ell\equiv 1\Mod 7;\\
%			``-", &\text{if }c'\ell \equiv -1\Mod 7. 
%		\end{array}
%		\right.\]
%		We will explain these exceptional styles $c'\ell\equiv \pm 1\Mod 7$ in the next section for (7-5,2). 
%	\end{remark}
	
%	In the following subsections, we show $\Arg(d_1\rightarrow d_2;\ell)$, $\Arg(d_2\rightarrow d_3;\ell)$, and $\Arg(d_3\rightarrow d_4;\ell)$ in all the cases $c'\Mod {42}$. These argument differences are sufficient to check Condition~\ref{Styles six points mod 7} because
%	\[\Arg(d_1\rightarrow d_2;\ell)=\Arg(d_5\rightarrow d_6;\ell)\text{ and } \Arg(d_2\rightarrow d_3;\ell)=\Arg(d_4\rightarrow d_5;\ell),\]
%	where the proof is the same as the proof of Lemma~\ref{p=5 Argreduction}. 
%	When $7\nmid c'$, we prove that $\Arg(d_1\rightarrow d_2;\ell)$, $\Arg(d_2\rightarrow d_3;\ell)$, and $\Arg(d_3\rightarrow d_4;\ell)$ satisfy Condition~\ref{Styles six points mod 7} in \S\ref{Subsection 7n+5: coprime to 42}. When $49|c$, we prove Proposition~\ref{Vrc sum equals 0, mod 7} in \S\ref{Subsection 7n+5: 49|c}. 
%	
	
    To verify Condition~\ref{Styles six points mod 7}, it suffices to show $\Arg(d_1\rightarrow d_2;\ell)$, $\Arg(d_2\rightarrow d_3;\ell)$, and $\Arg(d_3\rightarrow d_4;\ell)$ for all the cases $c'\Mod {42}$. This is because
	\[\Arg(d_1\rightarrow d_2;\ell)=\Arg(d_5\rightarrow d_6;\ell)\text{ and } \Arg(d_2\rightarrow d_3;\ell)=\Arg(d_4\rightarrow d_5;\ell),\]
    where the proof is similar as the proof of Lemma~\ref{p=5 Argreduction}. To prove that $\Arg(d_1\rightarrow d_2;\ell)$, $\Arg(d_2\rightarrow d_3;\ell)$, and $\Arg(d_3\rightarrow d_4;\ell)$ all satisfy Condition~\ref{Styles six points mod 7}, the approach is similar as in Section~\ref{Section (5-4)}: we divide the cases based on whether $c'$ is divisible by $2,3,7$, respectively. The detailed arguments are omitted here, but interested readers can refer to \cite{QihangKLsumsGitHub} for in-depth steps. Specifically, we provide the result for the case $7|c'$, as this will be useful in the proof of Section~\ref{Section part ii}.

	\subsection{Case \texorpdfstring{$7|c'$}{c' divisible by 7}}
	\label{Subsection 7n+5: 49|c}
	We still denote $c'=c/7$ and $V(r,c)=\{d\Mod c^*:\ d\equiv r\Mod {c'}\}$ for $r\Mod{c'}^*$. Now $|V(r,c)|=7$. Since $(d+c',c)=1$ when $(d,c)=1$, we can write $V(r,c)=\{d,d+c',d+2c',\cdots, d+6c'\}$ for $1\leq d<c'$ and $d\equiv r\Mod{c'}$. 
	
	We claim that Proposition~\ref{Vrc sum equals 0, mod 7} is still true: 
	\begin{equation}{\label{Vrc sum equals 0, 7 divide c case}}
		s_{r,c}\defeq \sum_{d\in V(r,c)} \frac{e\(-\frac{3c'a\ell^2}{14}\)}{\sin(\frac{\pi a\ell}{7})}e\(-\frac{12 cs(d,c)}{24c}\)e\(\frac{5d}c\)=0,
	\end{equation}
	while this time we have seven summands. We prove \eqref{Vrc sum equals 0, 7 divide c case} by showing that there are only three cases for the sum: all at the outer circle (radius $1/\sin(\frac{\pi }7)$), all at the middle circle (radius $1/\sin(\frac{2\pi }7)$), and all at the inner circle (radius $1/\sin(\frac{3\pi }7)$). Moreover, the seven points are equally distributed. Similar as before, we still denote $P_1$, $P_2$ and $P_3$ for each term in \eqref{Vrc sum equals 0, 7 divide c case} and investigate the argument differences contributed from each term.

	For any $d\in V(r,c)$ and $a\Mod c$ such that $ad\equiv 1\Mod c$, we define $d_*=d+c'$ and $a_*$ by $a_*d_*\equiv 1\Mod c$. Specifically, we take $a_*=a-c'$, $a-2c'$, $a+3c'$, $a+3c'$, $a-2c'$, $a-c'$, when $d\equiv 1,2,3,4,5,6\Mod 7$, respectively. 
	Note that $P_1(d)=(-1)^{ca\ell}/{\sin(\frac{\pi a\ell}7)}$ has period $c'$.  Hence we always have
	\[\Arg_1(d\rightarrow d_*;\ell)=0 \quad \text{and}\quad \Arg_3(d\rightarrow d_*;\ell)=\tfrac 57. \]

	For the result, we have
	\begin{equation}\label{Arg diff 7n+5 case 49|c}
		\Arg(d\rightarrow d_*;\ell)=\left\{\begin{array}{ll}
			-\frac 27 &\ d\equiv 1,6\Mod 7;\vspace{1ex}\\
			\frac 37 &\ d\equiv 2,5\Mod 7;\vspace{1ex}\\ 
			-\frac 17 &\ d\equiv 3,4\Mod 7. 
		\end{array}
		\right.
	\end{equation} 
	when $\ell=1$. 
	In the other cases $\ell=2,3$, only $P_1$ is affected and we still get \eqref{Arg diff 7n+5 case 49|c}. 
	\begin{proof}[Proof of Proposition~\ref{Vrc sum equals 0, mod 7} when $49|c$]
		It is clear that \eqref{Arg diff 7n+5 case 49|c} implies \eqref{Vrc sum equals 0, 7 divide c case}.
	\end{proof} 
	One may visualize \eqref{Arg diff 7n+5 case 49|c} in the following graphs: 
	\begin{center}
		\includegraphics[scale=0.35]{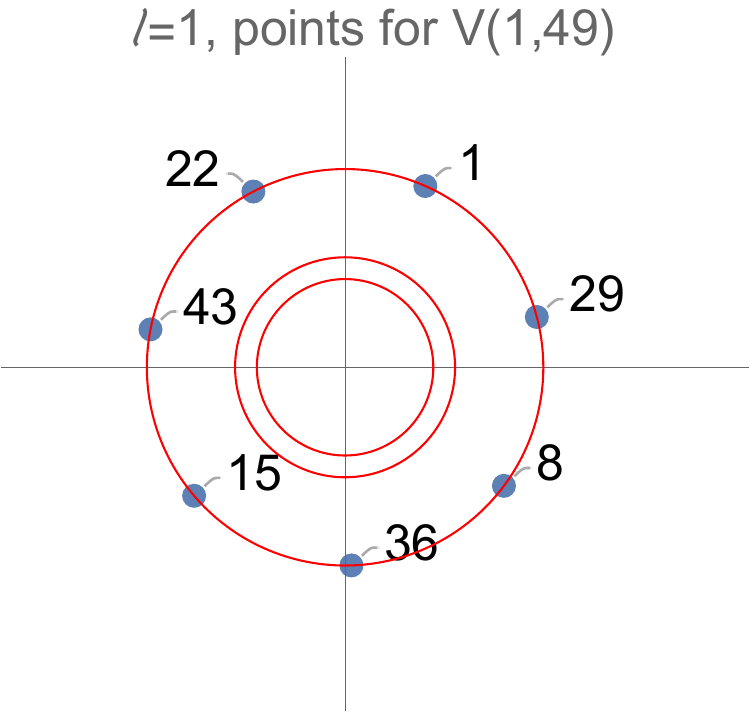}
		\includegraphics[scale=0.35]{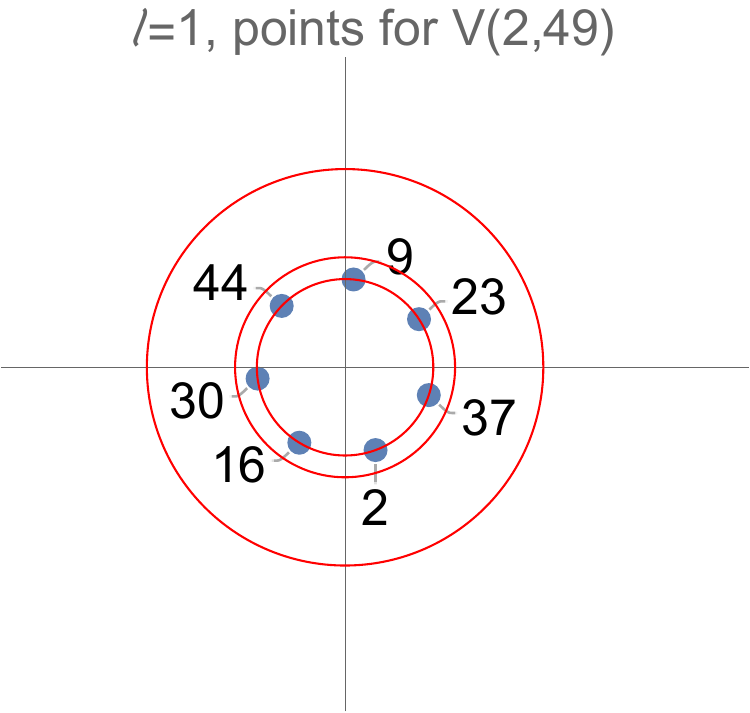}
		\includegraphics[scale=0.35]{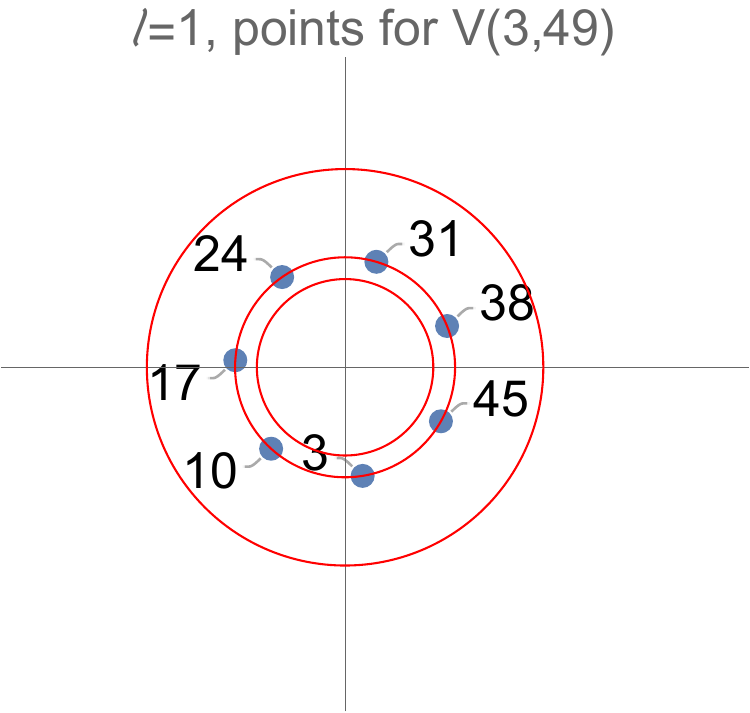}
	\end{center}

	We have proved Proposition~\ref{Vrc sum equals 0, mod 7}, which implies (7-5,1) of Theorem~\ref{Kloosterman sums vanish}.

	\section{Proof of (7-5,2) of Theorem~\ref{Kloosterman sums vanish}}
	\label{Section 7-5,2}
	
	For all $1\leq \ell\leq 6$, $n\geq 0$, $7|c$ and $7\nmid A$, if $A\ell\equiv \pm1 \Mod 7$ and $c=7A$, (7-5,2) becomes
	\begin{equation}\label{Appendix two KL sums cancel eq to prove}
		e(\tfrac 18)S_{\infty\infty}^{(\ell)}(0,7n+5,c,\mu_7)+2i{\sqrt 7}S_{0\infty}^{(\ell)}(0,7n+5,A,\mu_7;0)=0. 
	\end{equation} 
	We still denote $c'=c/7=A$ and $V(r,c)\defeq \{d\Mod c^*:\ d\equiv r\Mod c\}$ for $(r,c')=1$. 
	
	Recall \eqref{S infty infty for mod p} for $p=7$. By $\ell c\equiv \ell A\Mod 2$ we have
	\begin{equation}\label{explain S infty infty mod 7}
		e(\tfrac 18)S_{\infty\infty}^{(\ell)}(0,7n+5,c,\mu_7)=\sum_{d\Mod c^*}\frac{(-1)^{\ell A}\exp\(-\frac{3\pi ic'a\ell^2}{7}\)}{\sin(\frac{\pi a\ell}{7})}e^{-\pi i s(d,c)}e\(\frac{(7n+5)d}c\). 
	\end{equation}
	Recall that $[A\ell]$ is defined as the least non-negative residue of $A\ell\Mod 7$. By \eqref{S 0 infty mod 7, +1}, when $[A\ell]=1$, we denote $T$ by $A\ell=7T+1$ and
	\begin{align}\label{explain S 0 infty mod 7 +1}
		\begin{split}
			&2i{\sqrt 7}S_{0\infty}^{(\ell)}(0,7n+5,A,\mu_7;0)\\
			&=2i\sqrt 7(-1)^{A\ell-[A\ell]}\sum_{\substack{B\Mod A^*\\0<C<7A,\,7|C\\BC\equiv -1(A)}}e\(\frac{(\frac 32 T^2+\frac12 T)C}A\)e^{-\pi i s(B,A)}e\(\frac{(7n+5)B}A\). 
		\end{split}
	\end{align}
	By \eqref{S 0 infty mod 7, -1}, when $[A\ell]=6$, we denote $T$ by $A\ell=7T-1$ and have
	\begin{align}\label{explain S 0 infty mod 7 -1}
		\begin{split}
			&2i{\sqrt 7}S_{0\infty}^{(\ell)}(0,7n+5,A,\mu_7;0)\\
			&=2i\sqrt 7(-1)^{A\ell-[A\ell]}\!\!\!\!\sum_{\substack{B\Mod A^*\\0<C<7A,\,7|C\\BC\equiv -1(A)}}\!\!\!\! e\(\frac{(\frac 32 (T-1)^2+\frac52 (T-1)+1)C}A\)e^{-\pi i s(B,A)}e\(\frac{(7n+5)B}A\).    
		\end{split}
	\end{align}

	For $(r,c')=1$ and any $d\in V(r,c)$, we define $P(d)$ as 
	\begin{equation}\label{P d1}
		P(d)\defeq  \frac{(-1)^{[A\ell]}e\(-\frac{3c'a\ell^2}{14}\)}{\sin(\frac{\pi a\ell}{7})}e^{-\pi i s(d,c)}e\(\frac{(7n+5)d}c\)=: P_1(d)\cdot P_2(d)\cdot P_3(d). 
	\end{equation}
	When $A\ell=7T+1$, we denote $Q_1(B)= i$, $Q_3(B)=e(\frac{(7n+5)B}A)$, 
	\begin{equation}\label{Q B}
		Q_2(B)\defeq e\(\frac{(\frac 32 T^2+\frac 12 T)C}A\)e^{-\pi i s(B,A)}\quad  \text{and}\quad Q(B)=:2\sqrt 7\cdot Q_1(B)Q_2(B)Q_3(B); 
	\end{equation}
	when $A\ell=7T-1$, we only change the definition of $Q_2(B)$ to
	\begin{equation}\label{Q B -1}
		Q_2(B)\defeq e\(\frac{(\frac32 (T-1)^2+\frac52 (T-1)+1)C}A\)e^{-\pi i s(B,A)} 
	\end{equation}
	and still denote $Q(B)=2\sqrt 7 \cdot Q_1(B)Q_2(B)Q_3(B)$.

	We divide the cases according to $c'\ell\equiv \pm 1\Mod 7$, $\ell$, and the divisibility of $A$ by $2,3$.  For each $r\Mod A^*$,  recall that $d_1\in V(r,c)$ is the unique $d_1\Mod c^*$ such that $d_1\equiv 1\Mod 7$. 
	We compare the argument difference from $Q(B)$ to $P(d_1)$, where we choose 
	\begin{equation}\label{choose B and C, 7n+5}
		B=\left\{\begin{array}{ll}
			-d_1T, &\ A\ell=7T+1,\\
			d_1T, &\ A\ell=7T-1,	
		\end{array}
		\right. \quad \text{and}\quad 
		C=-7\overline{d_{1\{A\}}}. 
	\end{equation}
	We define $\Arg(Q_j\rightarrow P_j;\ell)$ in the following way: suppose $P_j(d_1)=Re^{i\Theta}$ and $Q_j(B)=R_Be^{i\Theta_B}$, then 
	\[\Arg(Q_j\rightarrow P_j;\ell)=\alpha \quad \text{if and only if}\quad \Theta-\Theta_B=\alpha\cdot 2\pi +2k\pi \text{ for }k\in \Z. \]
	We also denote $\Arg(Q\rightarrow P;\ell)=\sum_{j=1}^3 \Arg(Q_j\rightarrow P_j;\ell)$. Note that if $\Arg(Q_j\rightarrow P_j;\ell)=\alpha$, then $\Arg(Q_j\rightarrow P_j;\ell)=\alpha+k$ for all $k\in \Z$. 
	
	With the notation above, we claim that the argument differences satisfy the following proposition. 
	\begin{proposition}\label{claim 7T+1 Arg diff proposition}
		For $c=7c'=7A$, any $r\Mod{c'}^*$, $d_1\in V(r,c)$ and $B$ chosen by \eqref{choose B and C, 7n+5}, we have 
		\begin{align}
			\label{claim 7T+1 Arg diff}
			A\ell=7T+1: \quad &\Arg(Q\rightarrow P;\ell)=-\frac 37, -\frac 5{14}, \frac 3{14} \quad \text{for }\ell=1,2,3,\ \text{respectively;}\\
			\label{claim 7T-1 Arg diff}
			A\ell=7T-1:\quad & 
			\Arg(Q\rightarrow P;\ell)=\frac 37, \frac 5{14}, -\frac 3{14} \quad \text{for }\ell=1,2,3,\ \text{respectively}. 
		\end{align}
	\end{proposition}
	
	To visualize the argument differences, here are a few examples: 
	\begin{center}
		\includegraphics[scale=0.4]{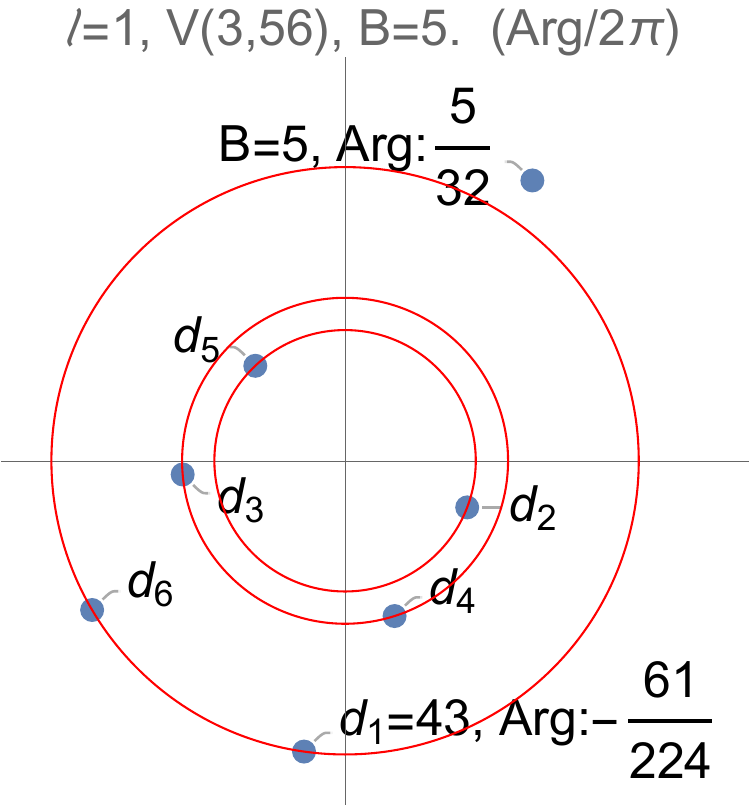}
		\includegraphics[scale=0.4]{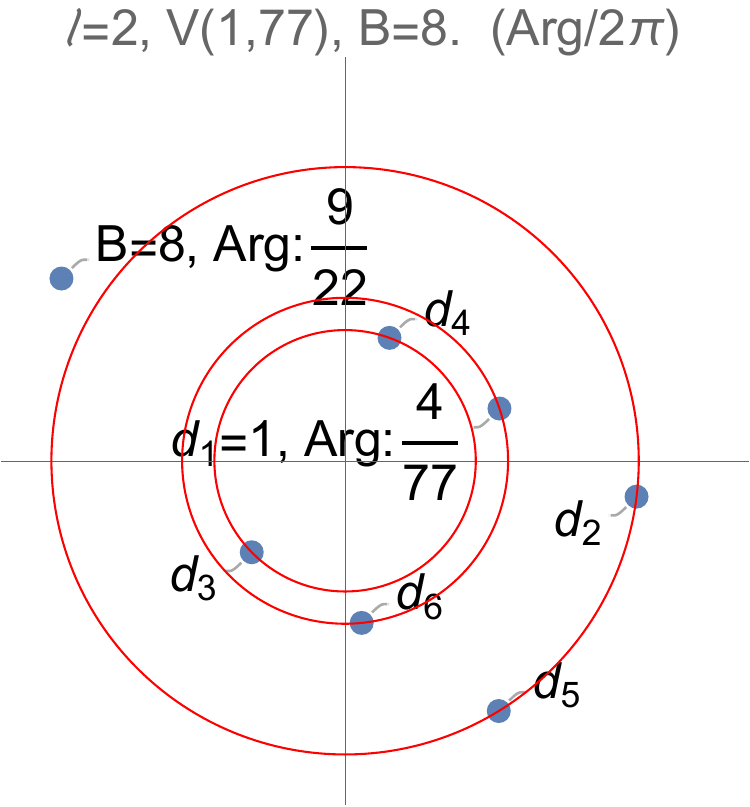}
		\includegraphics[scale=0.4]{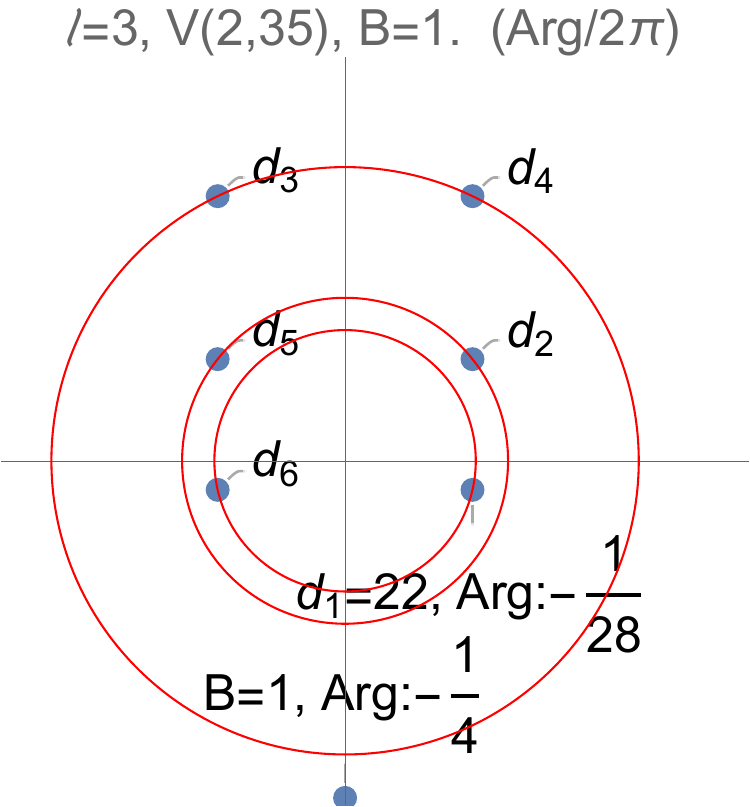}
		\vspace{15pt}
		
		\includegraphics[scale=0.4]{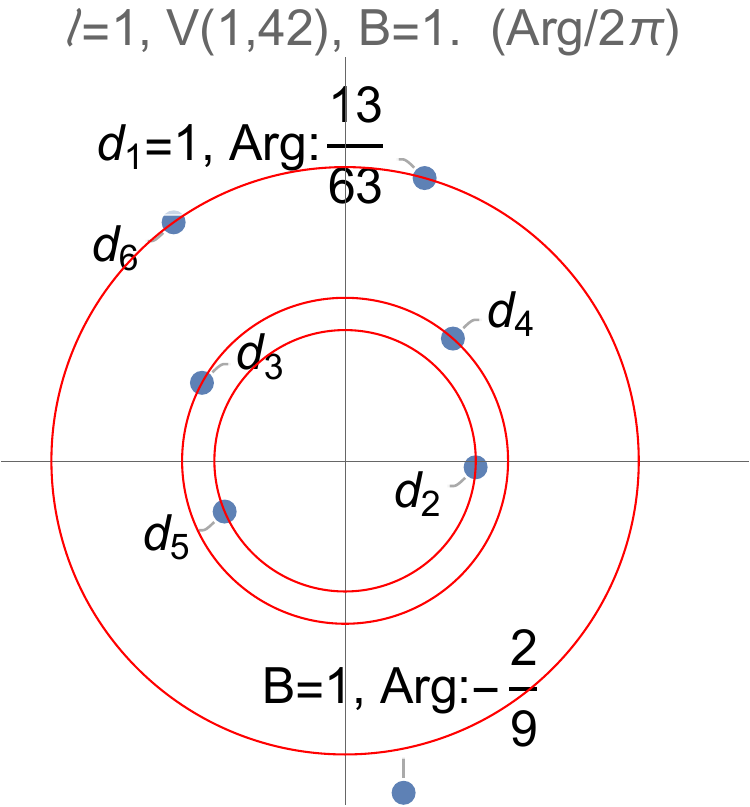}
		\includegraphics[scale=0.4]{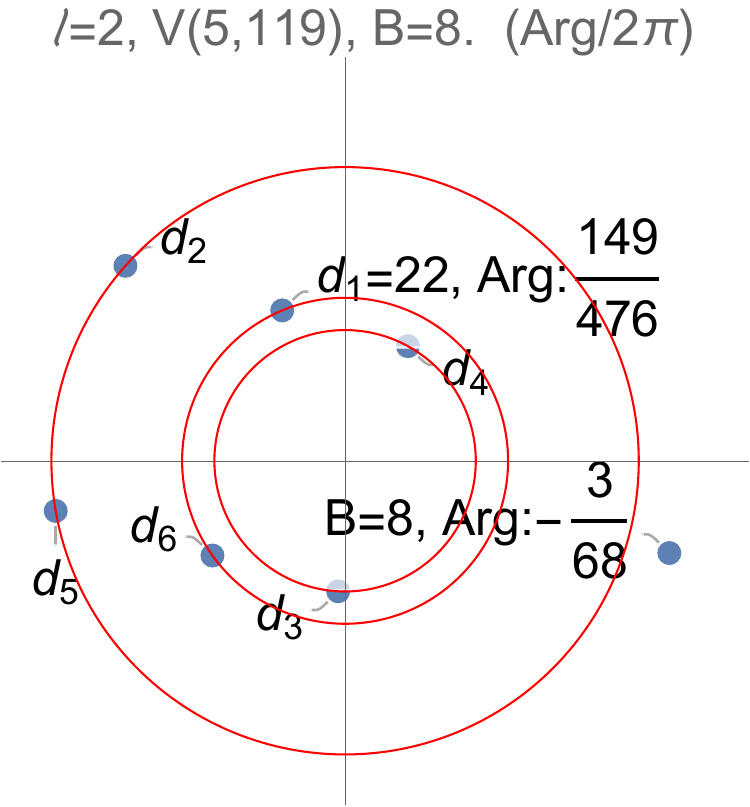}
		\includegraphics[scale=0.4]{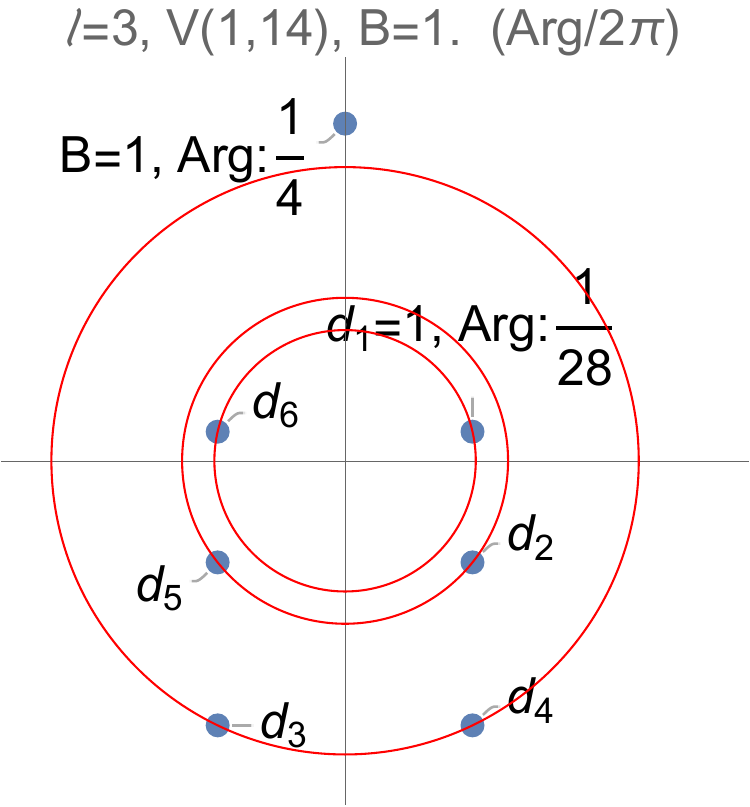}
	\end{center} 
	The red circles in the figures are centered at the origin with radii $\csc(\frac{\pi}7)$, $\csc(\frac{2\pi}7)$, and $\csc(\frac{3\pi}7)$, respectively, from the outside to the inside. The point labeled by $B$ represents $\frac{Q(B)}2$.

	For the styles of the six points $P(d_j)$ for $d_j\in V(r,c)$, we have the following condition. % This {\it has already been proved} by the tables in the former section, corresponding to the rows marked with ``$c'\ell\equiv \pm 1\Mod 7$?" whose entries are $+$ or $-$. 
	
	\begin{condition}\label{Styles six points mod 7, 7Tpm1 case}
		When $c'\ell\equiv \pm 1\Mod 7$, we have the following six styles for these six points $P(d)$ for $d\in V(r,c)$. 
		\begin{itemize}
			\item $\ell=1$. When $c'\ell\equiv 1\Mod 7$, the arguments going $d_1\rightarrow d_2\rightarrow d_3\rightarrow d_4\rightarrow d_5\rightarrow d_6\rightarrow d_1$ are  $\frac 3{14}$, $-\frac 3{7}$, $\frac 2{7}$, $-\frac 3{7}$, $\frac 3{14}$, $\frac 1{7}$, respectively. When $c'\ell\equiv -1\Mod 7$, the direction is reversed, as shown in the second line. 
			\begin{center}
				\begin{tabular}{|c|ccccccccccccc|}
					\hline
					& $d_1$ & $\rightarrow$ & $d_2$ & $\rightarrow$ & $d_3$ & $\rightarrow$ & $d_4$ & $\rightarrow$ & $d_5$ & $\rightarrow$ & $d_6$ & $\rightarrow$ & $d_1$ \\
					\hline
					$c'\equiv 1\Mod 7$  & & $\frac 3{14}$ & & $-\frac 3{7}$ & & $\frac 2{7}$ & & $-\frac 3{7}$ & & $\frac 3{14}$ & & $\frac 1{7}$& \ \\
					\hline
					$c'\equiv 6\Mod 7$  & & $-\frac 3{14}$ & & $\frac 3{7}$ & & $-\frac 2{7}$ & & $\frac 3{7}$ & & $-\frac 3{14}$ & & $-\frac 1{7}$& \ \\ 
					\hline
				\end{tabular}
			\end{center}
			
			\item $\ell=2$. The first line is for $c'\ell\equiv 1\Mod 7$ and the second line is for $c'\ell\equiv -1\Mod 7$. 
			\begin{center}
				\begin{tabular}{|c|ccccccccccccc|}
					\hline
					& $d_1$ & $\rightarrow$ & $d_2$ & $\rightarrow$ & $d_3$ & $\rightarrow$ & $d_4$ & $\rightarrow$ & $d_5$ & $\rightarrow$ & $d_6$ & $\rightarrow$ & $d_1$ \\
					\hline
					$c'\equiv 4\Mod 7$ & & $-\frac 1{14}$ & & $-\frac 5{14}$ & & $-\frac 3{7}$ & & $-\frac 5{14}$ & & $-\frac 1{14}$ & & $\frac 2{7}$& \ \\
					\hline
					$c'\equiv 3\Mod 7$ &  & $\frac 1{14}$ & & $\frac 5{14}$ & & $\frac 3{7}$ & & $\frac 5{14}$ & & $\frac 1{14}$ & & $-\frac 2{7}$& \ 
					\\
					\hline
				\end{tabular}
			\end{center}
			
			\item $\ell=3$. The first line is for $c'\ell\equiv 1\Mod 7$ and the second line is for $c'\ell\equiv -1\Mod 7$. 
			\begin{center}
				\begin{tabular}{|c|ccccccccccccc|}
					\hline
					& $d_1$ & $\rightarrow$ & $d_2$ & $\rightarrow$ & $d_3$ & $\rightarrow$ & $d_4$ & $\rightarrow$ & $d_5$ & $\rightarrow$ & $d_6$ & $\rightarrow$ & $d_1$ \\
					\hline
					$c'\equiv 5\Mod 7$  & & $\frac 1{7}$ & & $\frac 3{14}$ & & $-\frac 1{7}$ & & $\frac 3{14}$ & & $\frac 1{7}$ & & $\frac 3{7}$& \ \\
					\hline
					$c'\equiv 2\Mod 7$ & & $-\frac 1{7}$ & & $-\frac 3{14}$ & & $\frac 1{7}$ & & $-\frac 3{14}$ & & $-\frac 1{7}$ & & $-\frac 3{7}$& \ \\
					\hline
				\end{tabular}
			\end{center}
		\end{itemize}
		
	\end{condition}
	
	If the six points $P(d)$ for $d\in V(r,c)$ satisfy Condition~\ref{Styles six points mod 7, 7Tpm1 case}, and $\Arg(Q\rightarrow P;\ell)$ satisfies \eqref{claim 7T+1 Arg diff} and \eqref{claim 7T-1 Arg diff} in the corresponding cases, then we have
	\begin{equation}\label{sum Pd and Qb to 0, mod 7}
		s_{r,c}\defeq \sum_{d\in V(r,c)}P(d)+Q(B)=0. 
	\end{equation}
	Note that $B$ is chosen from $d_1\in V(r,c)$ and $A$, hence from $r$ and $c$. 
	One way to prove \eqref{sum Pd and Qb to 0, mod 7} is by using the identity
	\[\frac{\cos(\frac {\pi} 7)}{\sin(\frac{\pi}7)}+\frac{\cos(\frac {2\pi} 7)}{\sin(\frac{2\pi}7)}-\frac{\cos(\frac {3\pi} 7)}{\sin(\frac{3\pi}7)}=\sqrt 7,\ \  \text{ where }\frac 1{\sin(\frac{j\pi}7)}\text{ for }j=1,2,3\text{ are the radii.}\]
	
	\begin{proof}[Proof of (7-5,2) of Theorem~\ref{Kloosterman sums vanish}]
		This is implied by \eqref{Appendix two KL sums cancel eq to prove}, which is proved by \eqref{sum Pd and Qb to 0, mod 7}, \eqref{explain S infty infty mod 7}, \eqref{explain S 0 infty mod 7 +1}, \eqref{explain S 0 infty mod 7 -1}, $7B\equiv r\Mod A$, and
		\begin{align*}
			e(\tfrac 18)&S_{\infty\infty}^{(\ell)}(0,7n+5,c,\mu_7)+2i{\sqrt 7}S_{0\infty}^{(\ell)}(0,7n+5,A,\mu_7;0)\\
			&=(-1)^{A\ell-[A\ell]}\sum_{r\Mod A^*} s_{r,c} \ e\(\frac{nr}A\)=0. 
		\end{align*}
	\end{proof}

	Subsections \S\ref{Case 1 Mod 7, 2 nmid A, 3 nmid A}-\S\ref{Case 1 Mod 7, A even, 3 divides A} are devoted to prove \eqref{claim 7T+1 Arg diff}, i.e. the cases $A\ell=c'\ell \equiv 1\Mod 7$. We will not repeat the proof for \eqref{claim 7T-1 Arg diff} but just list a few key intermediate steps at the end of this section. 
	
	\subsection{Case \texorpdfstring{$c'\ell\equiv 1\Mod 7, \ 2\nmid A,$ and $3\nmid A$}{1 Mod 7, 2 nmid A, 3 nmid A}}\label{Case 1 Mod 7, 2 nmid A, 3 nmid A}
	Recall $d_1\equiv 1\Mod7$ and $d_1\equiv r\Mod {c'}$. 
	Recall that we define $1\leq \beta\leq 6$ as $\beta c'\equiv 1\Mod 7$ and here $\beta =\ell$. Note that $d_1-\beta A\equiv 7B\Mod{7A}$: 
	\[7B=d_1(1-A\ell)\equiv d_1+(7-d_1)\ell A\Mod{7A},\text{ so }7B\equiv \left\{\begin{array}{l}
		0\Mod 7,\\
		r\Mod A. 
	\end{array}\right. \]
	On the other hand, $d_1-\beta c'\equiv r\Mod A$ and $d_1-\beta c' \equiv 0\Mod 7$. 
	The argument difference between $P_3$ and $Q_3$ is easy to compute:
	\begin{equation}\label{Al1mod7 Q3 to P3}
		7\Arg(Q_3\rightarrow P_3;\ell)\equiv 5d_1\ell\equiv 5\ell\Mod 7
	\end{equation}
	which does not depend on $n$. 
	
	Recall $\overline{d_{1\{7A\}}}\equiv a_1\Mod{7A}$ and $\overline{B_{\{A\}}}\equiv 7\overline{d_{1\{A\}}}\Mod A$. We have
	\begin{align*}
		-84A(s(d_1,7A)-s(B,A))&\equiv -d_1-a_1+d_1(1-\beta A)+49\overline{d_{1\{A\}}}\\
		&\equiv -d_1\beta A-a_1+49\overline{d_{1\{A\}}}\Mod{7A}. 
	\end{align*}
	Hence 
	\begin{equation}\label{Al1mod7 84A mod 7 mod A}
		-84A(s(d_1,7A)-s(B,A))\equiv \left\{\begin{array}{l}
			-2\Mod{7}\\
			48\overline{ d_{1\{A\}}}\Mod A
		\end{array}
		\right.
	\end{equation}
	We also have
	\begin{align*}
		-84A(s(d_1,7A)-s(B,A))&\equiv -7A-1+2(\tfrac{d_1}{7A})+7(A+1)-14(\tfrac{B}{A})\\
		&\equiv 6+2(\tfrac{d_1}A)+2(\tfrac{d_1}A)(\tfrac{7}{A})\Mod 8,
	\end{align*}
	where the last step is because $(A,7)=1$, $(\frac{d_1}7)=(\frac 17)=1$ and $7B\equiv d_1\Mod A$. By $A$ is odd and $A\ell\equiv 1\Mod T$, we have $(\frac 7A)=(\frac{\ell}7)(-1)^{\frac{A-1}2}$. Combining $6|12cs(d_1,c)$ and $6|12As(B,A)$ we conclude
	\begin{equation}\label{Al1mod7 84A mod 24}
		-84A(s(d_1,7A)-s(B,A))\equiv \left\{\begin{array}{lll}
			18\Mod {24},& \text{ if } &A\equiv 1\Mod 4\  (\ell=1,2)\\
			& \text{ or if } &A\equiv 3\Mod 4 \ (\ell=3)
			\\
			6\Mod{24},& \text{ if } &A\equiv 3\Mod 4 \ (\ell=1,2)\\
			&\text{ or if } &A\equiv 1\Mod 4 \ (\ell=3). 
		\end{array}\right. 
	\end{equation}
	
	Next we check the part of $Q_2$ other than $e^{-\pi i s(B,A)}$. Since $A$ is odd and $T$ is even, we have
	\begin{align*}
		\(\tfrac 32 T^2+\tfrac 12 T \)C&\equiv \tfrac T2(3T+1)(-7\overline{d_{1\{A\}}})\\
		&\equiv  \tfrac T2(3-3A\ell -7)\overline{d_{1\{A\}}}\\
		&\equiv   -2T\overline{d_{1\{A\}}}\Mod A.
	\end{align*}
	Then the part of $Q_2$ other than $e^{-\pi i s(d,c)}$ is
	\begin{equation}\label{Al1mod7 Q2 first part}
		e\(\frac{24\cdot 2\overline{d_{1\{A\}}}(-7T)}{24\cdot 7A}\)=e\(\frac{48\overline{d_{1\{A\}}}(1-A\ell)}{168A}\),\text{ with numerator }\equiv \left\{\begin{array}{l}
			0\Mod 7,\\
			48\overline{d_{1\{A\}}}\Mod A,\\
			0\Mod{24}. 
		\end{array}
		\right.  
	\end{equation}

	We conclude that
	\begin{equation}
		24\cdot 7A\Arg(Q_2\rightarrow P_2;\ell)\equiv R_2\Mod{168A}
	\end{equation}
	where $R_2$ is determined by \eqref{Al1mod7 84A mod 7 mod A}, \eqref{Al1mod7 84A mod 24} and \eqref{Al1mod7 Q2 first part}: $R_2\equiv 0\Mod A$, $R_2\equiv -2\Mod 7$, and $R_2\equiv 18,6\Mod{24}$ depending on the cases in \eqref{Al1mod7 84A mod 24}. Therefore, by $A\ell \equiv 1\Mod 7$ and $A\Mod 4$ in \eqref{Al1mod7 84A mod 24} we conclude
	\begin{equation}\label{Al1mod7 Q2 to P2}
		\Arg(Q_2\rightarrow P_2;\ell)=\frac{23,\ 11,\ 13}{28}\quad \text{for }\ell=1,2,3. 
	\end{equation}

	Then we compute $\Arg(Q_1\rightarrow P_1;\ell)$. When $\ell=1$, since $A$ is odd, $A\equiv 1\Mod {14}$. Note that both $a_1\equiv 1,8\Mod{14}$ give the same result due to the sign of $\sin(\frac{\pi a}7)$. It is direct to get
	\begin{equation}\label{Al1mod7 Q1 to P1 l=1}
		\Arg(Q_1\rightarrow P_1;1)=\frac12-\frac 3{14}-\frac 14=\frac{1}{28}. 
	\end{equation}
	When $\ell=2$, we get $A\equiv 4\Mod 7$ and 
	\begin{equation}\label{Al1mod7 Q1 to P1 l=2}
		\Arg(Q_1\rightarrow P_1;2)=\frac12-\frac 3{7}-\frac 14=-\frac{5}{28}. 
	\end{equation}
	When $\ell=3$, we have $A\equiv 5\Mod {14}$ and
	\begin{equation}\label{Al1mod7 Q1 to P1 l=3}
		\Arg(Q_1\rightarrow P_1;3)=\frac12-\frac {9}{14}-\frac 14=-\frac{11}{28}. 
	\end{equation}
	
	Combining \eqref{Al1mod7 Q1 to P1 l=1}, \eqref{Al1mod7 Q1 to P1 l=2}, \eqref{Al1mod7 Q1 to P1 l=3}, \eqref{Al1mod7 Q2 to P2}, and \eqref{Al1mod7 Q3 to P3} proves \eqref{claim 7T+1 Arg diff}. 
	
	\subsection{Case \texorpdfstring{$c'\ell\equiv 1\Mod 7, \ 2\nmid A,$ and $3|A$}{1 Mod 7, 2 nmid A, 3 divides A}}\label{Case 1 Mod 7, 2 nmid A, 3 divides A}
	In this case \eqref{Al1mod7 Q3 to P3} still holds. For $\Arg(Q_2\rightarrow P_2;\ell)$, by \eqref{Congru Dedekind mod theta c} we have
	\begin{align*}
		-84A(s(d_1,7A)-s(B,A))\equiv -d_1A\ell-\overline{d_{1\{21A\}}}+7\overline{(-d_1T)_{\{3A\}}}\Mod{21A}. 
	\end{align*}
	We have 
	\begin{align}\label{Al1mod7 84A mod 3A, A odd divisible by 3}
		\begin{split}
			-84A(s(d_1,7A)-s(B,A))&\equiv -d_1A\ell-\overline{d_{1\{3A\}}}+49\overline{(d_1-d_1A\ell)_{\{3A\}}}\\
			&\equiv -d_1A\ell +(48d_1+d_1A\ell) \overline{d_{1\{3A\}}}\overline{(d_1-d_1A\ell)_{\{3A\}}}\\
			&\equiv d_1A\ell\(\overline{d_{1\{3A\}}}\overline{(d_1-d_1A\ell)_{\{3A\}}}-1\)+48\overline{d_{1\{A\}}}\\
			&\equiv 48\overline{d_{1\{A\}}}\Mod{3A}
		\end{split}
	\end{align}
	where in the second congruence we use
	\[\overline{(x+y)_{m}}-49\overline{x_{\{m\}}}\equiv \overline{x_{\{m\}}}\overline{(x+y)_{\{m\}}}(-48x-49y)\Mod m \]
	for $(x+y,m)=(x,m)=1$ and in the last two congruences we use
	\begin{equation}\label{Congru m1 x inverse m1m2 is m1 x inverse m2 mod m1m2}
		m_1\overline{x_{\{m_1m_2\}}}\equiv m_1\overline{x_{\{m_2\}}}\Mod{m_1m_2}
	\end{equation}
	for $(x,m_1m_2)=1$. We still have
	\begin{equation}\label{Al1mod7 84A mod 7, A odd divisible by 3}
		-84A(s(d_1,7A)-s(B,A))\equiv -2\Mod 7. 
	\end{equation}
	Moreover, \eqref{Al1mod7 84A mod 24} and \eqref{Al1mod7 Q2 first part} still hold except the second congruence in \eqref{Al1mod7 Q2 first part} should be changed to $48\overline{d_{1\{A\}}}\Mod{3A} $. 
	
	We conclude
	\begin{equation}
		24\cdot 7A\Arg(Q_2\rightarrow P_2;\ell)\equiv R_2\Mod{168A}
	\end{equation}
	where $R_2$ is determined by \eqref{Al1mod7 84A mod 3A, A odd divisible by 3}, \eqref{Al1mod7 84A mod 7, A odd divisible by 3}, \eqref{Al1mod7 84A mod 24} and \eqref{Al1mod7 Q2 first part}: $R_2\equiv 0\Mod {3A}$, $R_2\equiv -2\Mod 7$, and $R_2\equiv 18,6\Mod{24}$ depending on the cases in \eqref{Al1mod7 84A mod 24}. Therefore, by $A\ell \equiv 1\Mod 7$ and \eqref{Al1mod7 84A mod 24} we conclude
	\begin{equation}\label{Al1mod7 Q2 to P2, A odd divisible by 3}
		\Arg(Q_2\rightarrow P_2;\ell)=\frac{23,\ 11,\ 13}{28}\quad \text{for }\ell=1,2,3. 
	\end{equation}
	
	The condition $3|A$ does not affect $\Arg(Q_1\rightarrow P_1;\ell)$ and $\Arg(Q_3\rightarrow P_3;\ell)$. Combining \eqref{Al1mod7 Q2 to P2, A odd divisible by 3} with \eqref{Al1mod7 Q1 to P1 l=1}, \eqref{Al1mod7 Q1 to P1 l=2}, \eqref{Al1mod7 Q1 to P1 l=3}, and \eqref{Al1mod7 Q3 to P3}, we have proved \eqref{claim 7T+1 Arg diff} in this case. 
	
	\subsection{Case \texorpdfstring{$c'\ell\equiv 1\Mod 7, \ 2| A,$ and $3\nmid A$}{1 Mod 7, A even, 3 nmid A}}\label{Case 1 Mod 7, A even, 3 nmid A}
	Recall \eqref{Al1mod7 Q3 to P3}. For $\Arg(Q_2\rightarrow P_2;\ell)$ we have \eqref{Al1mod7 84A mod 7 mod A} and need to use \eqref{Congru Dedekind c even}. Let $\lambda\geq 1$ be defined as $2^\lambda \| A$. Recall $B=-d_1T$ and $7T+1=A\ell$. We have
	\begin{align*}
		-84A&(s(d_1,7A)-s(B,A))\\
		\equiv & -d_1-\overline{ d_{1\{8\times 2^\lambda \}}}(49A^2+21A+1)-14\overline{ d_{1\{8\times 2^\lambda \}}}A(\tfrac{7A}{d_1})\\
		&+d_1(1-A\ell)+49\overline{ (d_1-d_1A\ell)_{\{8\times 2^\lambda \}}}(A^2+3A+1)+14\overline{ B_{\{8\times 2^\lambda \}}}A(\tfrac{A}{B})\\
		\equiv &-d_1A\ell +49A^2\cdot d_1A\ell \overline{ (d_1-d_1A\ell)_{\{8\times 2^\lambda \}}}\overline{ d_{1\{8\times 2^\lambda \}}}\\
		&+21A(6d_1+d_1A\ell)\overline{ (d_1-d_1A\ell)_{\{8\times 2^\lambda \}}}\overline{ d_{1\{8\times 2^\lambda \}}}\\
		&+(48d_1+d_1A\ell)\overline{ (d_1-d_1A\ell)_{\{8\times 2^\lambda \}}}\overline{ d_{1\{8\times 2^\lambda \}}}\\
		&+14A\(\overline{ B_{\{8\times 2^\lambda \}}}(\tfrac AB)-\overline{ d_{1\{8\times 2^\lambda \}}}(\tfrac{7A}{d_1})\)\Mod {8\times 2^\lambda }. 
	\end{align*}
	Since $2^\lambda\|A$ with $\lambda\geq 1$, we apply \eqref{Congru m1 x inverse m1m2 is m1 x inverse m2 mod m1m2} and $x^2\equiv 1\Mod 8$ for odd $x$ to get
	\begin{align*}
		-84A(s(d_1,7A)-s(B,A))\equiv & \  6d_1A+d_1A^2\ell(1+\ell)+48\overline{d_{1\{A\}}}\\
		&+6A\(B(\tfrac AB)-d_1(\tfrac{7A}{d_1})\)\Mod {8\times 2^\lambda }. 
	\end{align*}
	By \eqref{Al1mod7 84A mod 7, A odd divisible by 3}, 
	To determine $B(\tfrac AB)-d_1(\tfrac{7A}{d_1})\Mod 4$, we use the quadratic reciprocity \eqref{quadratic reciprocity}. 
	By $B<0$ odd and $A>0$, we compute 
	\begin{align}\label{Al1mod7 84A mod 8 2 Lambda A even}
		\begin{split}
			B(\tfrac AB)-d_1(\tfrac {7A}{d_1})&\equiv -d_1T(\tfrac {B}{A})(-1)^{\frac{\frac{A}{2^\lambda}-1}2\cdot\frac{B-1}2}-d_1(\tfrac{d_1}{7A})(-1)^{\frac{7\cdot\frac{A}{2^\lambda}-1}2\cdot\frac{d_1-1}2}\\
			& \equiv -d_1T(\tfrac {d_1-d_1A\ell}{A})(\tfrac 7A)(-1)^{\frac{\frac{A}{2^\lambda}-1}2\cdot\frac{B-1}2}-d_1(\tfrac{d_1}{A})(-1)^{\frac{7\cdot\frac{A}{2^\lambda}-1}2\cdot\frac{d_1-1}2}\Mod 4
		\end{split}
	\end{align}
	Here are the cases: 
	\begin{enumerate}
		\item If $4|A$, then we have $T\equiv 1\Mod 4$, $B\equiv -d_1\Mod 4$. Moreover, $(\frac{d_1-d_1A\ell}A)=(\frac {d_1}A)$ always (note that $A$ is even and we have to consider $(\frac{d_1}2)$). Now \eqref{Al1mod7 84A mod 8 2 Lambda A even} simplifies to $(\frac \ell 7)d_1+1\Mod 4$. In this case $d_1A^2\ell(1+\ell)\equiv 0\Mod{8\times 2^\lambda}$ and  we conclude 
		\begin{equation}\label{Al1mod7 84A mod 8 2 Lambda A even, result}
			-84A(s(d_1,7A)-s(B,A))\equiv \left\{\begin{array}{ll}
				2A+48\overline{d_{1\{A\}}}\Mod{8\times 2^\lambda },& \ell=1,2;\\
				6A+48\overline{d_{1\{A\}}}\Mod{8\times 2^\lambda },& \ell=3.
			\end{array}
			\right.
		\end{equation}
		\item If $2\|A$ and $\ell=1$, then $T\equiv 3\Mod 4$, $B\equiv d_1\Mod 4$ and the above \eqref{Al1mod7 84A mod 8 2 Lambda A even} simplifies to $d_1-1\Mod 4$. Then as $A(12d_1-6+2d_1A)\equiv 2A\Mod {8\times 2^\lambda}$, we conclude the same as the first line of \eqref{Al1mod7 84A mod 8 2 Lambda A even, result}. 
		\item If $2\|A$ and $\ell=2$, then $T\equiv 1\Mod 4$, $B\equiv -d_1\Mod 4$, and $(\frac{d_1-d_1A\ell}A)=-(\frac {d_1}A)$. Now \eqref{Al1mod7 84A mod 8 2 Lambda A even} gives $d_1-1\Mod 4$ and we again get the first line of \eqref{Al1mod7 84A mod 8 2 Lambda A even, result}. 
		\item If $2\|A$ and $\ell=3$, then $T\equiv 3\Mod 4$, $B\equiv d_1\Mod 4$, and $(\frac A7)=(\frac{3A}7)(\frac 37)=-1$. Here \eqref{Al1mod7 84A mod 8 2 Lambda A even} results in $d_1-1\Mod 4$ again. Note that $d_1A^2\ell(1+\ell)\equiv 0\Mod{8\times 2^\lambda }$ and we get the second line of \eqref{Al1mod7 84A mod 8 2 Lambda A even, result}. 
	\end{enumerate}
	
	Next we check the part of $Q_2$ other than $e^{-\pi i s(d,c)}$. In this case $A$ is even, so $3T+1$ is even and we have
	\[
	\(\tfrac 32 T^2+\tfrac 12 T \)C\equiv \tfrac {3T+1}2\cdot T(-7\overline{d_{1\{A\}}}) \equiv \frac{3T+1}2\overline{d_{1\{A\}}}\Mod A.
	\]
	When written with denominator $24\cdot7A$, we have
	\[
	e\(\frac{(\frac 32 T^2+\frac 12 T )C}A\)=e\(\frac{36A\ell\overline{ d_{1\{A\}}}+48\overline{ d_{1\{A\}}}}{24\cdot 7A} \)
	\]
	whose numerator is
	\begin{equation}\label{Al1mod7 TCA mod 168A, A even}
		36A\ell\overline{ d_{1\{A\}}}+48\overline{ d_{1\{A\}}}\equiv \left\{\begin{array}{ll}
			0\Mod 7,&\\
			48\overline{d_{1\{A\}}}\Mod {3A},&\\
			4A+48\overline{d_{1\{A\}}}\Mod{8\times 2^\lambda},&\ell=1,3,\\
			48\overline{d_{1\{A\}}}\Mod{8\times 2^\lambda},&\ell=2.\\
		\end{array}
		\right.  
	\end{equation}
	Combining the above computation with \eqref{Al1mod7 84A mod 7 mod A}, \eqref{Al1mod7 84A mod 8 2 Lambda A even, result} and \eqref{Congru Dedekind theta is or not 3}, we get
	\begin{equation}\label{Al1mod7 Q2 to P2, A even}
		\Arg(Q_2\rightarrow P_2;\ell)=\frac{9,11,27}{28}\quad \text{for }\ell=1,2,3. 
	\end{equation}
	Then we compute $\Arg(Q_1\rightarrow P_1;\ell)$. When $\ell=1$, since $A$ is even, $\frac A2\equiv 4\Mod {7}$. Note that $a_1\equiv 1\Mod{14}$ because $a_1$ is odd. It is direct to get (remember $Q_1=i$)
	\begin{equation}\label{Al1mod7 Q1 to P1 A even l=1}
		\Arg(Q_1\rightarrow P_1;1)=\frac12-\frac 5{7}-\frac 14=-\frac{13}{28}. 
	\end{equation}
	When $\ell=2$, we get $\frac A2\equiv 2\Mod 7$ and 
	\begin{equation}\label{Al1mod7 Q1 to P1 A even l=2}
		\Arg(Q_1\rightarrow P_1;2)=\frac12-\frac 3{7}-\frac 14=-\frac{5}{28}. 
	\end{equation}
	When $\ell=3$, we have $\frac A2\equiv 6\Mod {14}$ and 
	\begin{equation}\label{Al1mod7 Q1 to P1 A even l=3}
		\Arg(Q_1\rightarrow P_1;3)=\frac12-\frac {1}{7}-\frac 14=\frac{3}{28}. 
	\end{equation}
	
	Combining \eqref{Al1mod7 Q1 to P1 A even l=1}, \eqref{Al1mod7 Q1 to P1 A even l=2}, \eqref{Al1mod7 Q1 to P1 A even l=3}, \eqref{Al1mod7 Q2 to P2, A even}, and \eqref{Al1mod7 Q3 to P3}, we get 
	\begin{equation}
		\Arg(Q\rightarrow P;\ell)=-\frac 37, -\frac 5{14}, \frac 3{14} \quad \text{for }\ell=1,2,3. 
	\end{equation}
	This proves \eqref{claim 7T+1 Arg diff}.

	\subsection{Case \texorpdfstring{$c'\ell\equiv 1\Mod 7, \ 2| A,$ and $3| A$}{1 Mod 7, A even, 3 divides A}}\label{Case 1 Mod 7, A even, 3 divides A}
	Comparing to the former case, the only difference in getting $\Arg(Q_2\rightarrow P_2;\ell)$ in \eqref{Al1mod7 Q2 to P2, A even} is that we should using \eqref{Al1mod7 84A mod 3A, A odd divisible by 3} instead of \eqref{Al1mod7 84A mod 7 mod A}. The result \eqref{Al1mod7 Q2 to P2, A even} still holds in this case. The condition $3|A$ or $3\nmid A$ does not affect the computation for $\Arg(Q_1\rightarrow P_1;\ell)$ and $\Arg(Q_3\rightarrow P_3;\ell)$, hence we still have \eqref{claim 7T+1 Arg diff}: 
	\begin{equation}
		\Arg(Q\rightarrow P;\ell)=-\frac 37, -\frac 5{14}, \frac 3{14} \quad \text{for }\ell=1,2,3. 
	\end{equation}
	
	Now we have finished the proof in all the cases for $A$ when $A\ell\equiv 1\Mod 7$ and proved \eqref{claim 7T+1 Arg diff} in Proposition~\ref{claim 7T+1 Arg diff proposition}. For the other case $A\ell\equiv -1\Mod 7$, we will not repeat the same process but just list the key argument differences below. For every $r\Mod{c'}^*$, we compare $P(d_1)$ \eqref{P d1} given $d_1\in V(r,c)$ and $Q(B)$ \eqref{Q B -1} given 
	\[T\defeq \frac{A\ell+1}{7}>0,\quad B=d_1T\quad \text{and \ \ } C=-7\overline{ d_{1\{A\}}}.\]
	Now $7B=d_1+d_1A\ell$. We shall get Table~\ref{table: Almod -1 all results}. 
	\begin{table}[!htbp]
		\centering
		\begin{tabular}{|c|c|c|c|}
			\hline
			Case $2\nmid A$: & $\ell=1$ & $\ell=2$ & $\ell=3$\\
			\hline
			$\Arg(Q_1\rightarrow P_1;\ell)$ & $-\frac 1{28}$ & $\frac 5{28}$ & $\frac {11}{28}$ \\
			[1ex]
			$\Arg(Q_2\rightarrow P_2;\ell)$ & $\frac 5{28}$ & $-\frac {11}{28}$ & $-\frac {13}{28}$ \\
			[1ex]
			$\Arg(Q_3\rightarrow P_3;\ell)$ & $\frac 2{7}$ & $-\frac 3{7}$ & $-\frac 1{7}$ \\
			[1ex]
			$\Arg(Q\rightarrow P;\ell)$ & $\frac 3{7}$ & $\frac 5{14}$ & $-\frac 3{14}$ \\
			\hhline{|=|=|=|=|}
			Case $2| A$: & $\ell=1$ & $\ell=2$ & $\ell=3$\\
			\hline
			$\Arg(Q_1\rightarrow P_1;\ell)$ & $\frac {13}{28}$ & $\frac 5{28}$ & $-\frac 3{28}$ \\
			[1ex]
			$\Arg(Q_2\rightarrow P_2;\ell)$ & $-\frac 9{28}$ & $-\frac {11}{28}$ & $\frac 1{28}$ \\
			[1ex]
			$\Arg(Q_3\rightarrow P_3;\ell)$ & $\frac 2{7}$ & $-\frac 3{7}$ & $-\frac 1{7}$ \\
			[1ex]
			$\Arg(Q\rightarrow P;\ell)$ & $\frac 3{7}$ & $\frac 5{14}$ & $-\frac 3{14}$ \\
			\hline 
		\end{tabular}
		\vspace{0.5ex}
		\caption{Table for the case $A\ell\equiv -1\Mod 7$ }
		\label{table: Almod -1 all results}
	\end{table}
	
	We have finished the proof of (7-5,2) of Theorem~\ref{Kloosterman sums vanish}.

	\section{Part (ii) of Theorem~\ref{Kloosterman sums vanish}}
	\label{Section part ii}
	
	For prime $p=5,7$ and integers $a,b$, recall the notation 
	\[C_p^{a,b}= \cos(\tfrac{a\pi}p)-\cos(\tfrac{b\pi}p). \]
	
	\subsection{(5-1) and (5-2) of Theorem~\ref{Kloosterman sums vanish}}
	
	We still denote $c'=c/5$ and first deal with the case $25|c$. For $(r,c)=1$, recall \eqref{Arg diff d to d star, 5n+4} with $V(r,c)$, $d$ and $d_*$ in that subsection. By \eqref{Arg diff d to d star, 5n+4}, we have $\Arg(d\rightarrow d_*;\ell)=-\frac 15$ for the $5n+4$ case. Since we have the $5n+1$ case here, we obtain
	\[\Arg(d\rightarrow d_*;\ell)=-\tfrac 15-\tfrac {3c'}c=-\tfrac 45. \]
	Hence we get $S_{\infty\infty}^{(\ell)}(0,5n+1,c,\mu_5)=0$ for $\ell\in \{1,2\}$, $25|c$, and every $n\geq 0$. This proves (5-1) when $25|c$. 
	
	Similarly, in the $5n+2$ case, we obtain
	\[\Arg(d\rightarrow d_*;\ell)=-\tfrac 15-\tfrac {2c'}c=-\tfrac 35. \]
	We still get $S_{\infty\infty}^{(\ell)}(0,5n+2,c,\mu_5)=0$ for $\ell\in \{1,2\}$, $25|c$, and every $n\geq 0$, which proves (5-2) when $25|c$. 
	
	Now we focus on the case $5\|c$. For $(r,c)=1$, recall the notation of $V(r,c)$, $d_j$ and $a_j$ in \eqref{Vrc dj choice jbetac' p=5}. 
	By \eqref{P1P2P3} and Condition~\ref{Styles four points mod 5}, we find that the argument differences of $P(d)$ for $d\in V(r,c)$ only depends on $c'\Mod 5$. Moreover, in \eqref{S infty infty for mod 5 using}, if we change $5n+4$ to $5n+1$ or to $5n+2$, then only the argument of $P_3(d)$ is affected. Recall that we define $\beta\in \{1,2,3,4\}$ by $\beta c'\equiv 1\Mod 5$. 
	
	\subsubsection{The $5n+1$ case}
	As Proposition~\ref{Vrc sum equals 0}, we denote
	\begin{equation}{\label{Vrc sum equals 0, 5n+1 case}}
		s_{r,c}^{(\ell)}=\sum_{d\in V(r,c)}P_1(d)P_2(d)P_3(d)\defeq \sum_{d\in V(r,c)} \frac{e\(-\frac{3c'a\ell^2}{10}\)}{\sin(\frac{\pi a\ell}{5})}e^{-\pi is(d,c)}e\(\frac{d}c\).
	\end{equation}
	Here $P_3(d)=e(\frac{d}c)$ instead of $e(\frac{4d}c)$ in the $5n+4$ case Proposition~\ref{Vrc sum equals 0}. To prove (5-1) of Theorem~\ref{Kloosterman sums vanish}, it suffices to show
	\begin{equation}\label{src sum to 0, 5n+1 case}
		C_5^{2,4}\sin(\tfrac \pi 5) s_{r,c}^{(1)}+C_5^{4,2}\sin(\tfrac {2\pi} 5) s_{r,c}^{(2)}=0. 
	\end{equation}
	
	When we compute $\Arg_3(d_j\rightarrow d_{j+1};\ell)$ for $j=1,2,3$, previously it was $e(\frac{4\beta}5)$ and now it should be $e(\frac \beta 5)$. Therefore, we need to subtract $\frac {3\beta}5$ from the argument differences in Condition~\ref{Styles four points mod 5} in each case. 
	
	Since we need both $\ell=1$ and $\ell=2$ appears at the same time, we write our new condition in the following way. It is important to note that the way we compute $\Arg(d_4\rightarrow d_1;\ell)$ is by 
	\[\sum_{j=1}^{3}\Arg(d_j\rightarrow d_{j+1};\ell)+\Arg(d_4\rightarrow d_1;\ell)=0 \]
	but not by subtracting $\frac{3\beta}5$. 
	
	\begin{condition}\label{Styles 5n+1}
		For the $5n+1$ case, we have the following styles of argument differences:
		\begin{enumerate}
			\item[$\bullet$] $c'\equiv 1\Mod 5$, $\beta=1$;
			\begin{table}[!htbp]
				\centering
				\begin{tabular}{|c|ccccccccc|}
					\hline
					$c'\equiv 1\Mod 5$ & $d_1$ & $\rightarrow$ & $d_2$ & $\rightarrow$ & $d_3$ & $\rightarrow$ & $d_4$ & $\rightarrow$ & $d_1$ \\
					\hline
					$\Arg(d_u\rightarrow d_v;1)$  & & $-\frac 3{10}$ & & $\frac 1{2}$ & & $-\frac 3{10}$ & & $\frac 1{10}$ & \ \\
					[1ex]
					$\Arg(d_u\rightarrow d_v;2)$ & & $\frac 2{5}$ & & $-\frac 1{10}$ & & $\frac 2{5}$ & & $\frac 3{10}$ & \ \\
					\hline
				\end{tabular}
			\end{table}
			
			\item[$\bullet$] $c'\equiv 2\Mod 5$, $\beta =3$;
			\begin{table}[!htbp]
				\centering
				\begin{tabular}{|c|ccccccccc|}
					\hline
					$c'\equiv 2\Mod 5$ & $d_1$ & $\rightarrow$ & $d_2$ & $\rightarrow$ & $d_3$ & $\rightarrow$ & $d_4$ & $\rightarrow$ & $d_1$ \\
					\hline
					$\Arg(d_u\rightarrow d_v;1)$  & & $-\frac 3{10}$ & & $-\frac 3{10}$ & & $-\frac 3{10}$ & & $-\frac 1{10}$ & \ \\
					[1ex]
					$\Arg(d_u\rightarrow d_v;2)$ & & $-\frac 2{5}$ & & $\frac 1{2}$ & & $-\frac 2{5}$ & & $\frac 3{10}$ & \ \\
					\hline
				\end{tabular}
			\end{table}
			
			\item[$\bullet$] $c'\equiv 3\Mod 5$, $\beta =2$; 
			\begin{table}[!htbp]
				\centering
				\begin{tabular}{|c|ccccccccc|}
					\hline
					$c'\equiv 3\Mod 5$ & $d_1$ & $\rightarrow$ & $d_2$ & $\rightarrow$ & $d_3$ & $\rightarrow$ & $d_4$ & $\rightarrow$ & $d_1$ \\
					\hline
					$\Arg(d_u\rightarrow d_v;1)$  & & $\frac 3{10}$ & & $\frac 3{10}$ & & $\frac 3{10}$ & & $\frac 1{10}$ & \ \\
					[1ex]
					$\Arg(d_u\rightarrow d_v;2)$ & & $\frac 2{5}$ & & $\frac 1{2}$ & & $\frac 2{5}$ & & $-\frac 3{10}$ & \ \\
					\hline
				\end{tabular}
			\end{table} 
			
			\item[$\bullet$] $c'\equiv 4\Mod 5$, $\beta =4$.
			\begin{table}[!htbp]
				\centering
				\begin{tabular}{|c|ccccccccc|}
					\hline
					$c'\equiv 4\Mod 5$ & $d_1$ & $\rightarrow$ & $d_2$ & $\rightarrow$ & $d_3$ & $\rightarrow$ & $d_4$ & $\rightarrow$ & $d_1$ \\
					\hline
					$\Arg(d_u\rightarrow d_v;1)$  & & $\frac 3{10}$ & & $\frac 1{2}$ & & $\frac 3{10}$ & & $-\frac 1{10}$ & \ \\
					[1ex]
					$\Arg(d_u\rightarrow d_v;2)$ & & $-\frac 2{5}$ & & $\frac 1{10}$ & & $-\frac 2{5}$ & & $-\frac 3{10}$ & \ \\
					\hline
				\end{tabular}
			\end{table}
		\end{enumerate}
	\end{condition}

	The condition above corresponding to the following styles of $P(d)$ for $d\in V(r,c)$:  
	\begin{center}
		\includegraphics[scale=0.3]{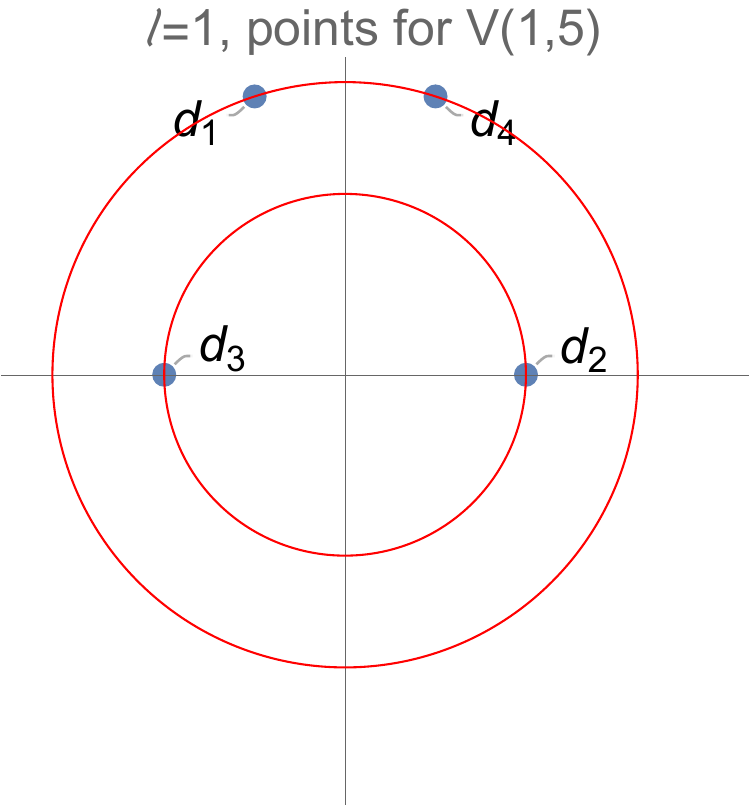}
		\includegraphics[scale=0.3]{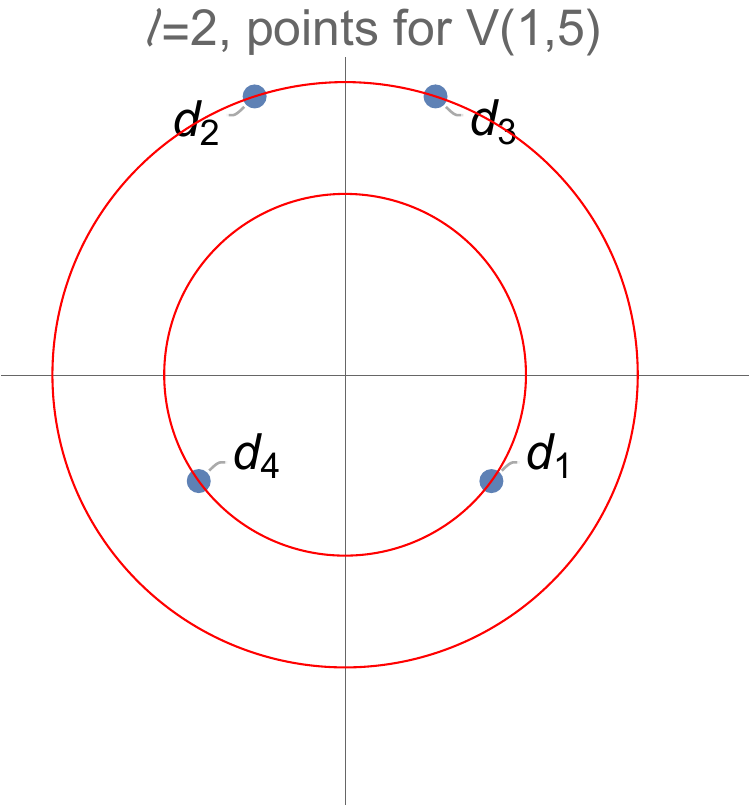}
		\qquad
		\includegraphics[scale=0.3]{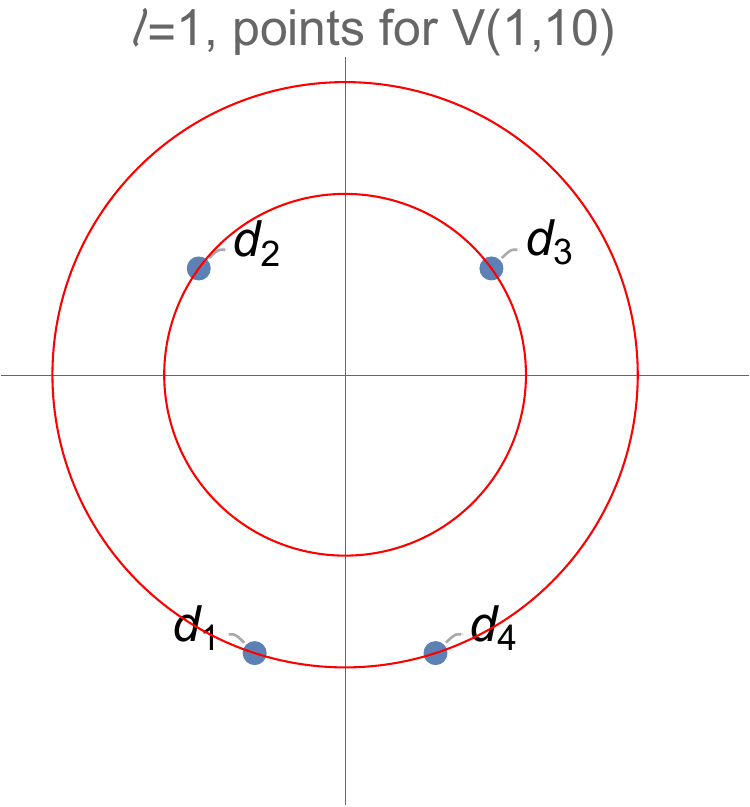}
		\includegraphics[scale=0.3]{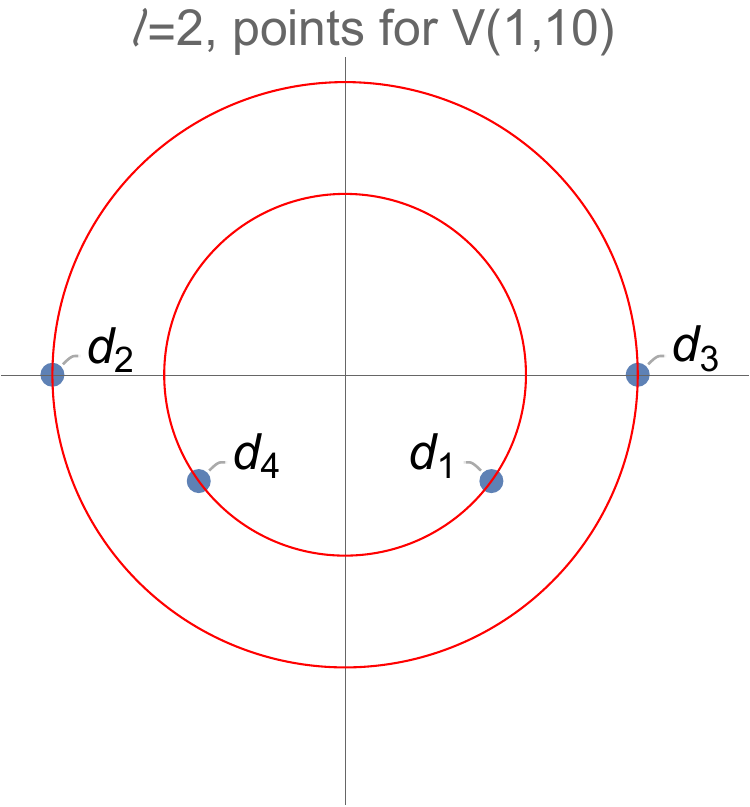}
		\\
		\includegraphics[scale=0.3]{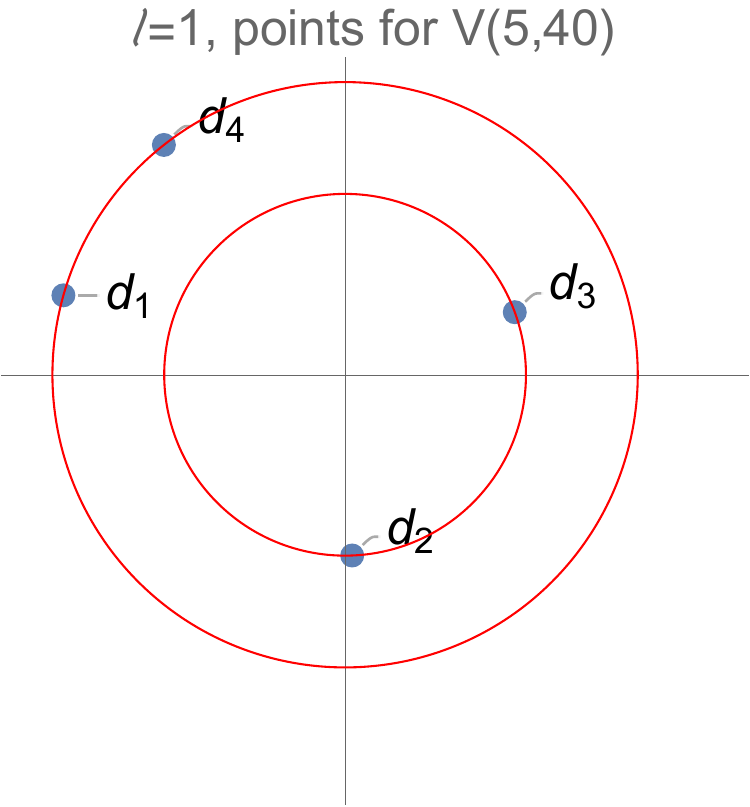}
		\includegraphics[scale=0.3]{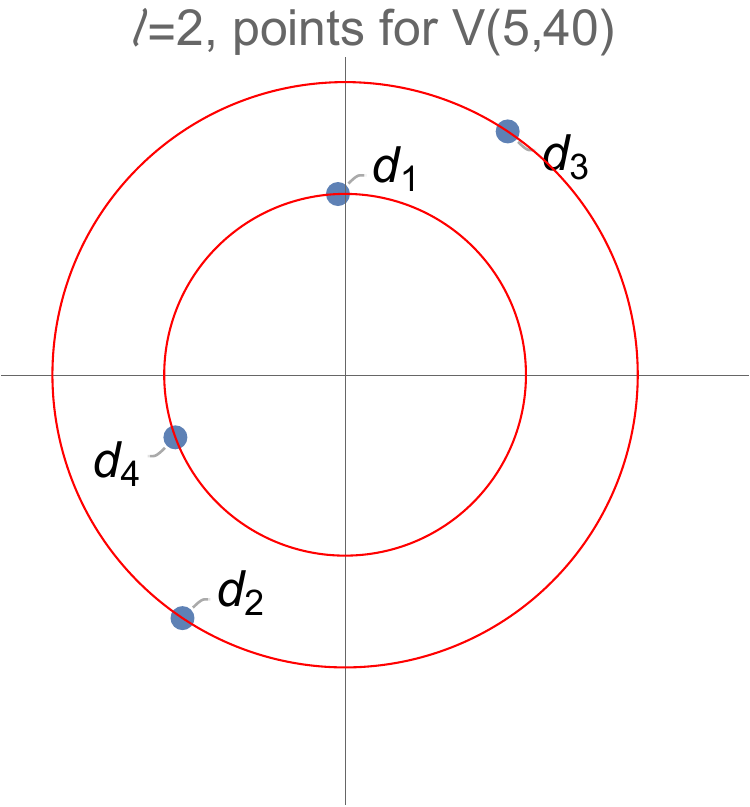}
		\qquad
		\includegraphics[scale=0.3]{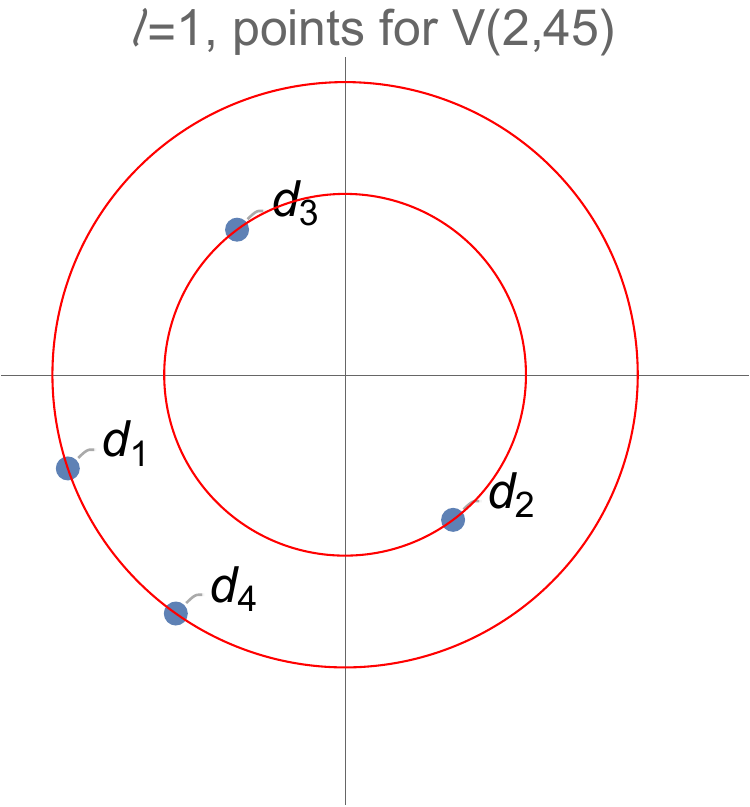}
		\includegraphics[scale=0.3]{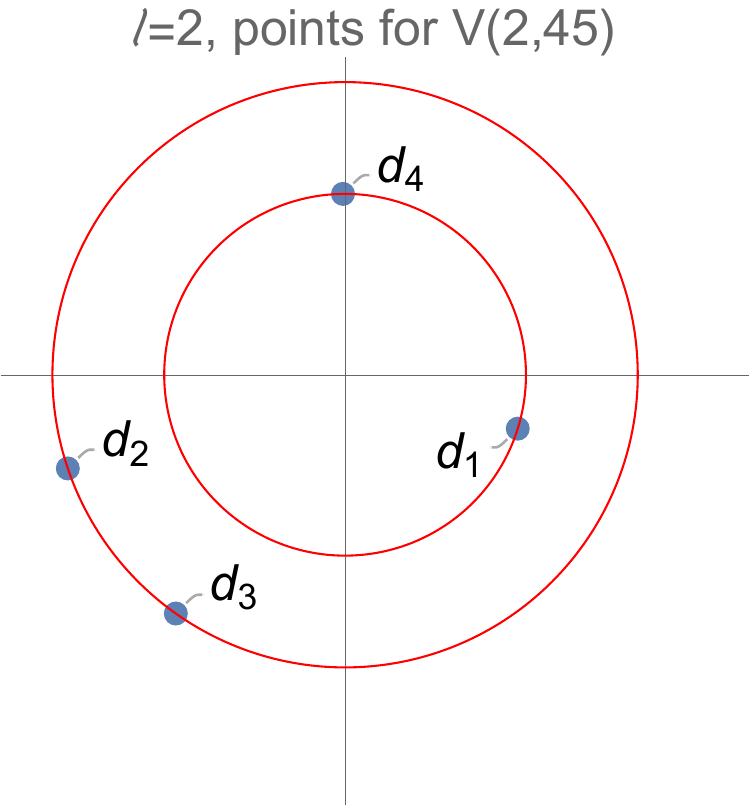}
	\end{center}
	For every two graphs close to each other in a row which satisfy the argument differences in the corresponding cases in Condition~\ref{Styles 5n+1}, it proves \eqref{src sum to 0, 5n+1 case} due to the following equations: 
	\begin{align*}
		C_5^{2,4}\cos(\tfrac{\pi}{10})+C_5^{4,2}\sin(\tfrac{2\pi}{5})\(\frac{\cos(\frac{\pi}{10})}{\sin(\frac{\pi}5)}-\frac{\cos(\frac{3\pi}{10})}{\sin(\frac{2\pi}5)}\)&=0,\quad \text{for }c'\equiv 1,4\Mod 5;\\
		C_5^{2,4}\sin(\tfrac{\pi}5)\(\frac{\cos(\frac{\pi}{10})}{\sin(\frac{\pi}5)}-\frac{\cos(\frac{3\pi}{10})}{\sin(\frac{2\pi}5)}\)+C_5^{4,2}\cos(\tfrac{3\pi}{10})&=0,\quad \text{for }c'\equiv 2,3\Mod 5. 
	\end{align*}
	This proves (5-1) of Theorem~\ref{Kloosterman sums vanish}.

	\subsubsection{The $5n+2$ case}
	As \eqref{Vrc sum equals 0, 5n+1 case}, we denote $P_3(d)=e(\frac{2d}c)$ instead of $e(\frac{4d}c)$ in the $5n+4$ case Proposition~\ref{Vrc sum equals 0} and instead of $e(\frac{d}c)$ in the $5n+1$ case \eqref{Vrc sum equals 0, 5n+1 case}. To prove (5-2) of Theorem~\ref{Kloosterman sums vanish}, it suffices to show
	\begin{equation}\label{src sum to 0, 5n+2 case}
		C_5^{0,4}\sin(\tfrac \pi 5) s_{r,c}^{(1)}+C_5^{0,2}\sin(\tfrac {2\pi} 5) s_{r,c}^{(2)}=0. 
	\end{equation}
	
	When we compute $\Arg_3(d_j\rightarrow d_{j+1};\ell)$ for $j=1,2,3$, in \eqref{Vrc sum equals 0, 5n+1 case} it was $e(\frac{\beta}5)$ and now it should be $e(\frac{2\beta}5)$. Therefore, we need to add $\frac {\beta}5$ to the argument differences in Condition~\ref{Styles 5n+1} in each case to get the following condition. 
	\begin{condition}\label{Styles 5n+2}
		For the $5n+2$ case, we have the following styles of argument differences:
		\begin{enumerate}
			\item[$\bullet$] $c'\equiv 1\Mod 5$, $\beta=1$; 
			\begin{table}[!htbp]
				\centering
				\begin{tabular}{|c|ccccccccc|}
					\hline
					$c'\equiv 1\Mod 5$ & $d_1$ & $\rightarrow$ & $d_2$ & $\rightarrow$ & $d_3$ & $\rightarrow$ & $d_4$ & $\rightarrow$ & $d_1$ \\
					\hline
					$\Arg(d_u\rightarrow d_v;1)$  & & $-\frac 1{10}$ & & $-\frac 3{10}$ & & $-\frac 1{10}$ & & $\frac 1{2}$ & \ \\
					[1ex]
					$\Arg(d_u\rightarrow d_v;2)$ & & $-\frac 2{5}$ & & $\frac 1{10}$ & & $-\frac 2{5}$ & & $-\frac 3{10}$ & \ \\
					\hline
				\end{tabular}
			\end{table}
			
			\item[$\bullet$] $c'\equiv 2\Mod 5$, $\beta =3$;
			\begin{table}[!htbp]
				\centering
				\begin{tabular}{|c|ccccccccc|}
					\hline
					$c'\equiv 2\Mod 5$ & $d_1$ & $\rightarrow$ & $d_2$ & $\rightarrow$ & $d_3$ & $\rightarrow$ & $d_4$ & $\rightarrow$ & $d_1$ \\
					\hline
					$\Arg(d_u\rightarrow d_v;1)$  & & $\frac 3{10}$ & & $\frac 3{10}$ & & $\frac 3{10}$ & & $\frac 1{10}$ & \ \\
					[1ex]
					$\Arg(d_u\rightarrow d_v;2)$ & & $\frac 1{5}$ & & $\frac 1{10}$ & & $\frac 1{5}$ & & $\frac 1{2}$ & \ \\
					\hline
				\end{tabular}
			\end{table}
			
			\item[$\bullet$] $c'\equiv 3\Mod 5$, $\beta =2$;
			\begin{table}[!htbp]
				\centering
				\begin{tabular}{|c|ccccccccc|}
					\hline
					$c'\equiv 3\Mod 5$ & $d_1$ & $\rightarrow$ & $d_2$ & $\rightarrow$ & $d_3$ & $\rightarrow$ & $d_4$ & $\rightarrow$ & $d_1$ \\
					\hline
					$\Arg(d_u\rightarrow d_v;1)$  & & $-\frac 3{10}$ & & $-\frac 3{10}$ & & $-\frac 3{10}$ & & $-\frac 1{10}$ & \ \\
					[1ex]
					$\Arg(d_u\rightarrow d_v;2)$ & & $-\frac 1{5}$ & & $-\frac 1{10}$ & & $-\frac 1{5}$ & & $\frac 1{2}$ & \ \\
					\hline
				\end{tabular}
			\end{table} 
			
			\item[$\bullet$] $c'\equiv 4\Mod 5$, $\beta =4$. 
			\begin{table}[!htbp]
				\centering
				\begin{tabular}{|c|ccccccccc|}
					\hline
					$c'\equiv 4\Mod 5$ & $d_1$ & $\rightarrow$ & $d_2$ & $\rightarrow$ & $d_3$ & $\rightarrow$ & $d_4$ & $\rightarrow$ & $d_1$ \\
					\hline
					$\Arg(d_u\rightarrow d_v;1)$  & & $\frac 1{10}$ & & $\frac 3{10}$ & & $\frac 1{10}$ & & $\frac 1{2}$ & \ \\
					[1ex]
					$\Arg(d_u\rightarrow d_v;2)$ & & $\frac 2{5}$ & & $-\frac 1{10}$ & & $\frac 2{5}$ & & $\frac 3{10}$ & \ \\
					\hline
				\end{tabular}
			\end{table}
		\end{enumerate}
	\end{condition}
	The condition above corresponding to the following styles of $P(d)$ for $d\in V(r,c)$. 
	\begin{center}
		\includegraphics[scale=0.3]{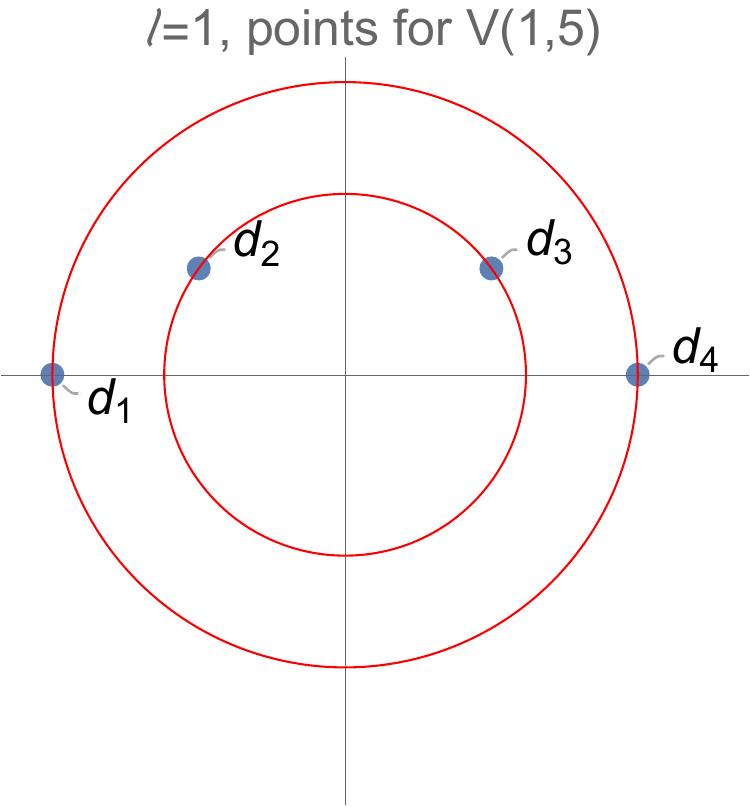}
		\includegraphics[scale=0.3]{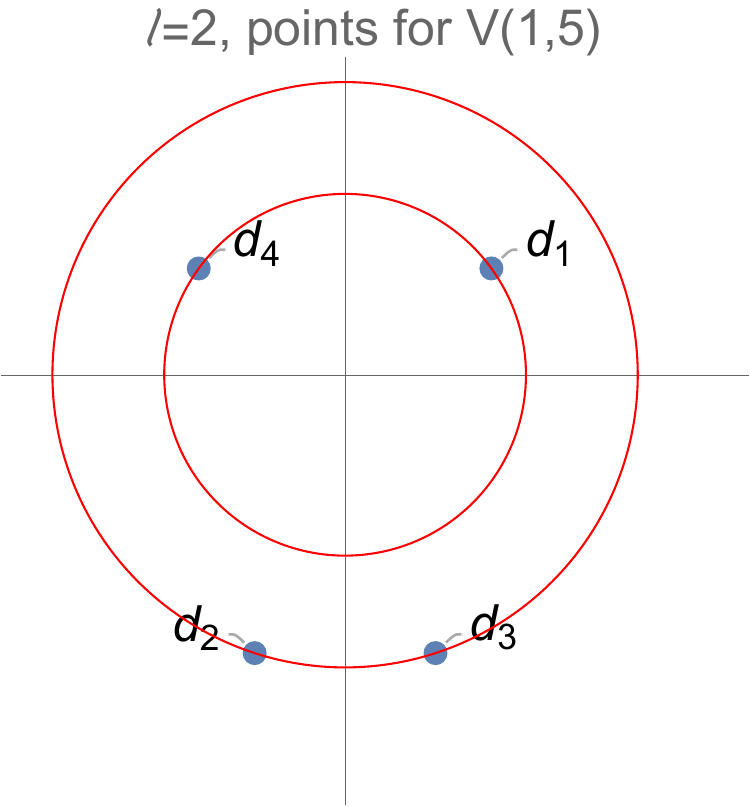}
		\qquad
		\includegraphics[scale=0.3]{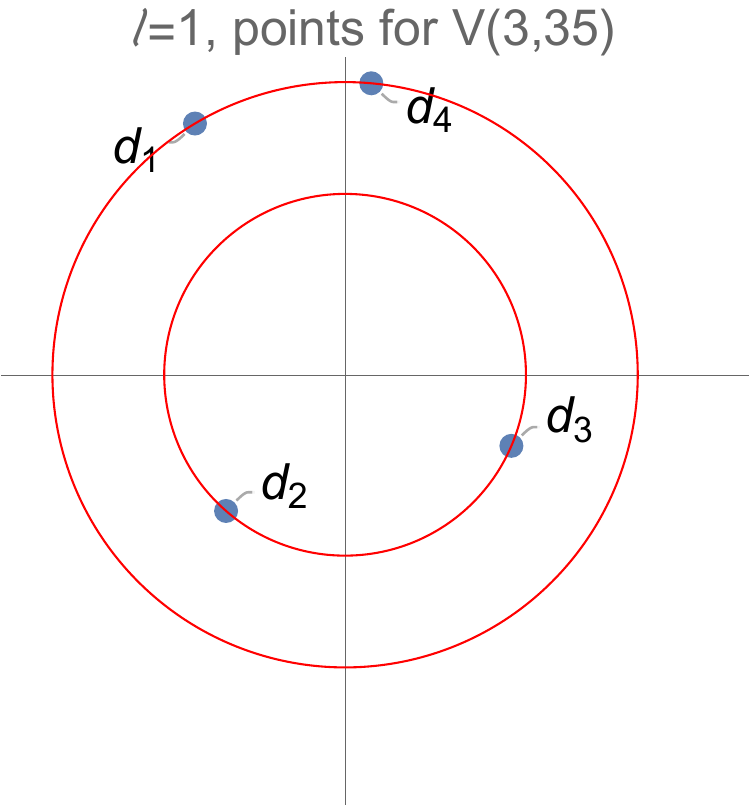}
		\includegraphics[scale=0.3]{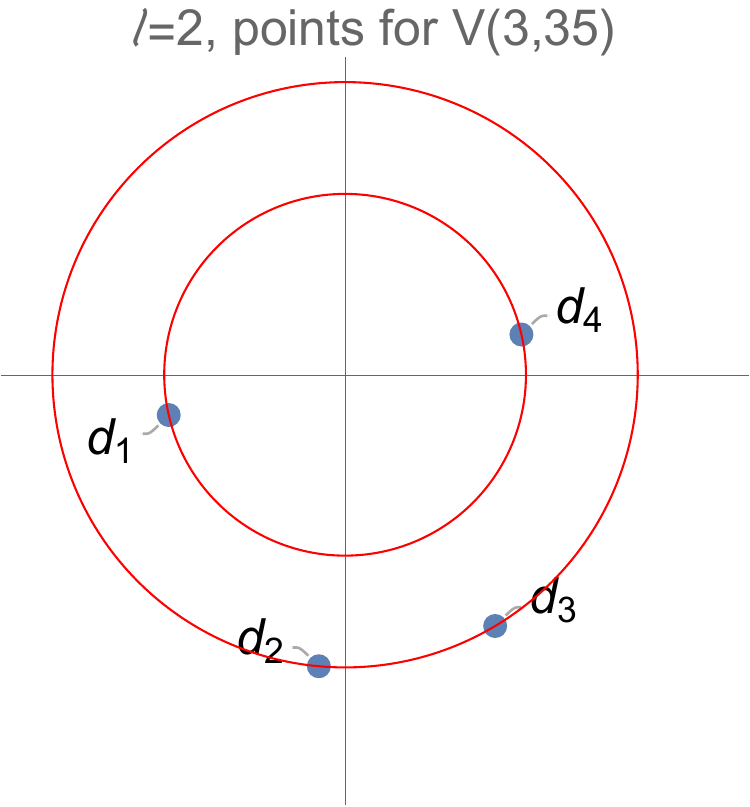}
		\\
		\includegraphics[scale=0.3]{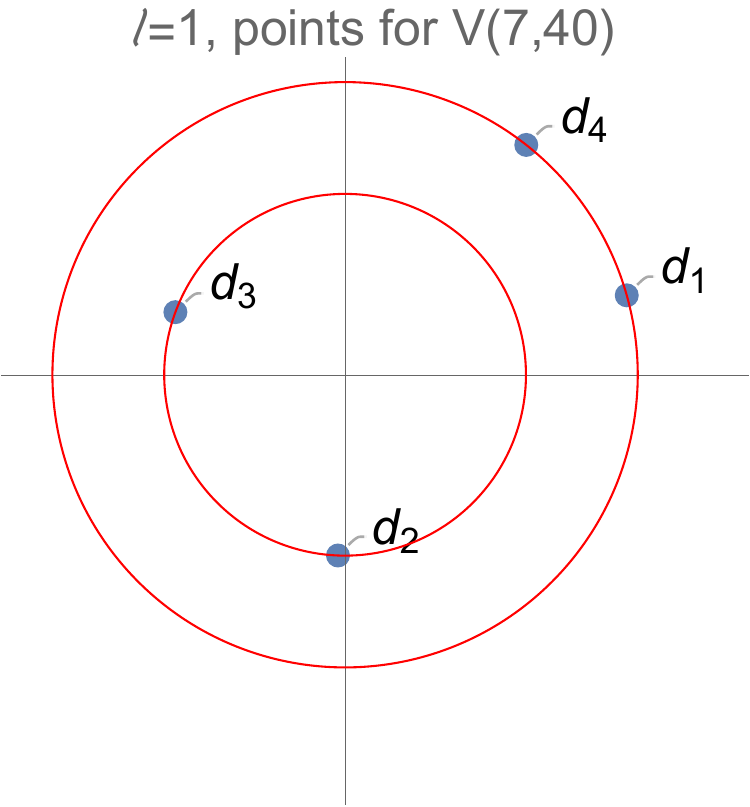}
		\includegraphics[scale=0.3]{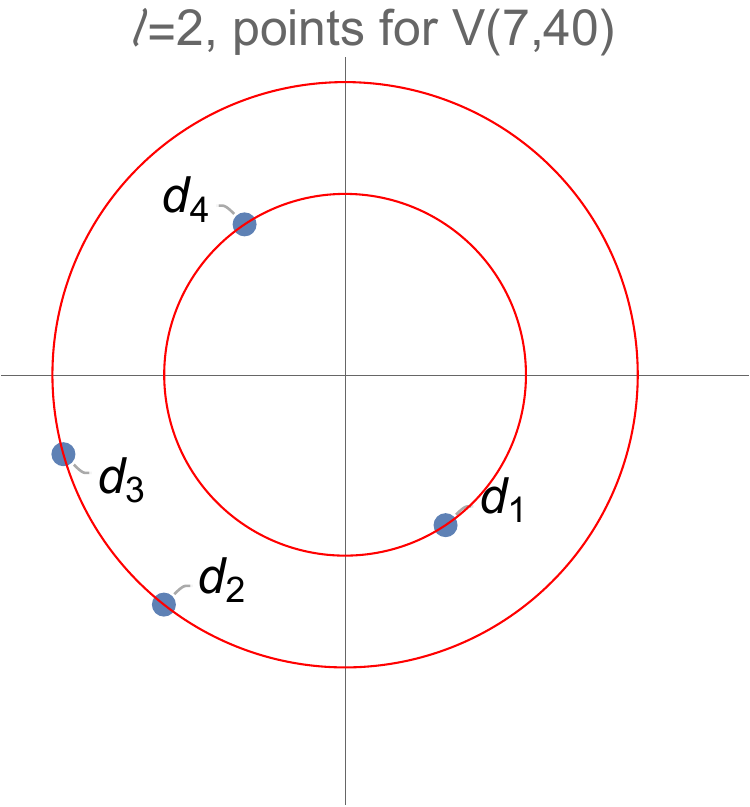}
		\qquad 
		\includegraphics[scale=0.3]{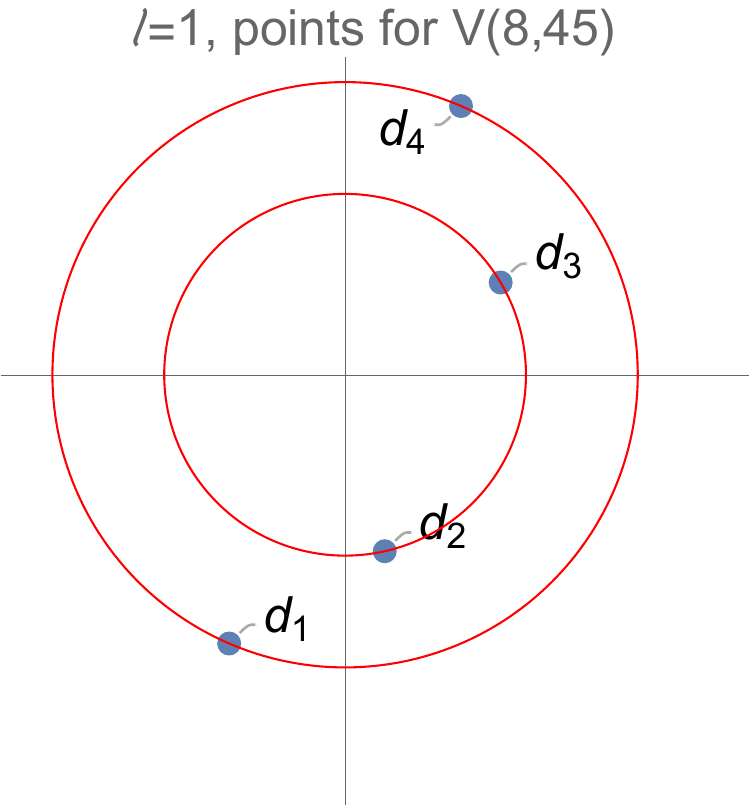}
		\includegraphics[scale=0.3]{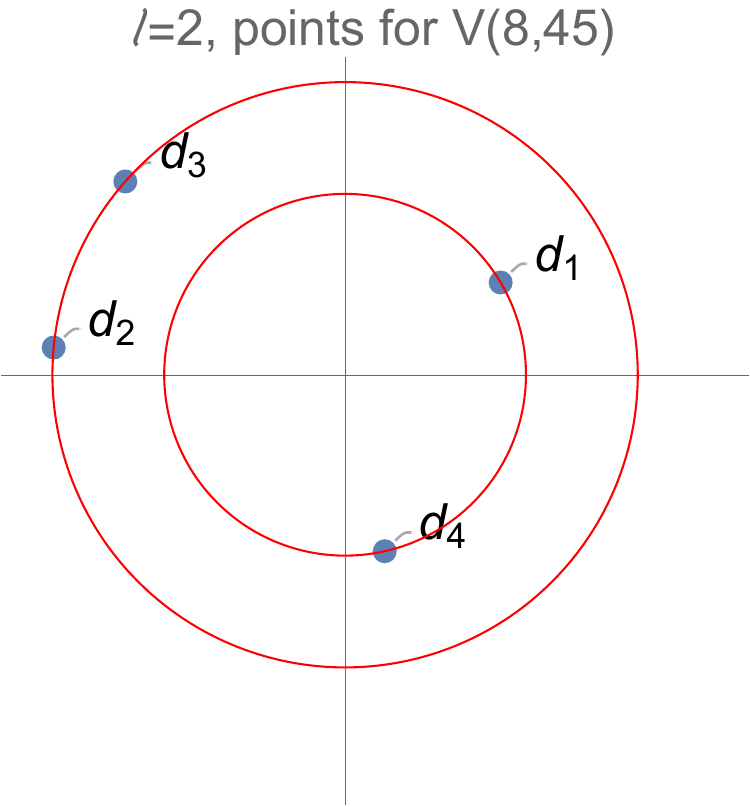}
	\end{center}
	For every two graphs with same $c$ which satisfy the argument differences in the corresponding cases in Condition~\ref{Styles 5n+2}, it proves \eqref{src sum to 0, 5n+2 case} due to the following equations: 
	\begin{align*}
		C_5^{0,4}\sin(\tfrac \pi 5)\cdot\frac{\cos(\tfrac{3\pi}{10})}{\sin(\frac{2\pi}5)}+C_5^{0,2}\sin(\tfrac{2\pi}{5})\(\frac{\cos(\frac{3\pi}{10})}{\sin(\frac{2\pi}5)}-\frac{\cos(\frac{\pi}{10})}{\sin(\frac{\pi}5)}\)&=0,\quad \text{for }c'\equiv 1,4\Mod 5;\\
		C_5^{0,4}\sin(\tfrac{\pi}{5})\(\frac{\cos(\frac{\pi}{10})}{\sin(\frac{\pi}5)}-\frac{\cos(\frac{3\pi}{10})}{\sin(\frac{2\pi}5)}\)-C_5^{0,2}\sin(\tfrac {2\pi} 5)\cdot\frac{\cos(\tfrac{\pi}{10})}{\sin(\frac{\pi}5)}&=0,\quad \text{for }c'\equiv 2,3\Mod 5. 
	\end{align*}
	This proves (5-2) of Theorem~\ref{Kloosterman sums vanish}.

	\subsection{Restate the condition for (7-5) of Theorem~\ref{Kloosterman sums vanish}}
	We still denote $c=7A=7c'$. 
	When $49|c$, recall the notation in \S\ref{Subsection 7n+5: 49|c} and we have \eqref{Arg diff 7n+5 case 49|c} for any $d\in V(r,c)$ and $d_*=d+c'$: 
	\begin{equation}\label{Arg diff 7n+5 case 49|c, last section using}
		\Arg(d\rightarrow d_*;\ell)=\left\{\begin{array}{ll}
			-\frac 27 &\ d\equiv 1,6\Mod 7;\vspace{1ex}\\
			\frac 37 &\ d\equiv 2,5\Mod 7;\vspace{1ex}\\ 
			-\frac 17 &\ d\equiv 3,4\Mod 7. 
		\end{array}
		\right.
	\end{equation}

	When $7\|c'$, denote $A=c'=c/7$ and recall the notation of $V(r,c)$ before \eqref{Vrc sum equals 0, mod 7} and $d_j$ and $a_j$ in \eqref{ajinverse7 dj =1}. We combine Condition~\ref{Styles six points mod 7}, Condition~\ref{Styles six points mod 7, 7Tpm1 case}, \eqref{claim 7T+1 Arg diff} and \eqref{claim 7T-1 Arg diff} and get the following condition: 
	\begin{condition}\label{Styles six points mod 7 7n+5}
		For the $7n+5$ case, we have the following conditions on $\Arg(Q\rightarrow P;\ell)$ when $A\ell\equiv \pm 1\Mod 7$ with tables for $\Arg_j(d_u\rightarrow d_v;\ell)$.  
		\begin{enumerate}
			\item[$\bullet$] $c'\equiv 1\Mod 7$. $A\cdot 1=7T+1$, $\Arg(Q\rightarrow P;1)=-\frac 37$; 
			\begin{table}[!htbp]
				\centering
				\begin{tabular}{|c|ccccccccccccc|}
					\hline
					$c'\equiv 1\Mod 7$ & $d_1$ & $\rightarrow$ & $d_2$ & $\rightarrow$ & $d_3$ & $\rightarrow$ & $d_4$ & $\rightarrow$ & $d_5$ & $\rightarrow$ & $d_6$ & $\rightarrow$ & $d_1$\\
					\hline
					$\Arg(d_u\rightarrow d_v;1)$  & & $\frac 3{14}$ & & $-\frac 3{7}$ & & $\frac 2{7}$ & & $-\frac 3{7}$ &  & $\frac 3{14}$ & & $\frac 1{7}$ & \ \\
					[1ex]
					$\Arg(d_u\rightarrow d_v;2)$ & & $-\frac 3{14}$ & & $-\frac 1{14}$ & & $-\frac 2{7}$ & & $-\frac 1{14}$  &  & $-\frac 3{14}$ & & $-\frac 1{7}$ & \ \\
					[1ex]
					$\Arg(d_u\rightarrow d_v;3)$ & & $-\frac 				3{7}$ & & $\frac 5{14}$ & & $\frac 3{7}$ & & $\frac 5{14}$  &  & $-\frac 3{7}$ & & $-\frac 2{7}$ & \ \\
					\hline
				\end{tabular}
			\end{table}
			
			\item[$\bullet$] $c'\equiv 2\Mod 7$. $A\cdot 3=7T-1$, $\Arg(Q\rightarrow P;3)=-\frac 3{14}$; 
			\begin{table}[!htbp]
				\centering
				\begin{tabular}{|c|ccccccccccccc|}
					\hline
					$c'\equiv 2\Mod 7$ & $d_1$ & $\rightarrow$ & $d_2$ & $\rightarrow$ & $d_3$ & $\rightarrow$ & $d_4$ & $\rightarrow$ & $d_5$ & $\rightarrow$ & $d_6$ & $\rightarrow$ & $d_1$\\
					\hline
					$\Arg(d_u\rightarrow d_v;1)$  & & $-\frac 5{14}$ & & $-\frac 2{7}$ & & $-\frac 1{7}$ & & $-\frac 2{7}$ &  & $-\frac 5{14}$ & & $\frac 3{7}$ & \ \\
					[1ex]
					$\Arg(d_u\rightarrow d_v;2)$ & & $-\frac 3{14}$ & & $-\frac 1{14}$ & & $-\frac 2{7}$ & & $-\frac 1{14}$  &  & $-\frac 3{14}$ & & $-\frac 1{7}$ & \ \\
					[1ex]
					$\Arg(d_u\rightarrow d_v;3)$ & & $-\frac 1{7}$ & & $-\frac 3{14}$ & & $\frac 1{7}$ & & $-\frac 3{14}$  &  & $-\frac 1{7}$ & & $-\frac 3{7}$ & \ \\
					\hline
				\end{tabular}
			\end{table}
			
			\item[$\bullet$] $c'\equiv 3\Mod 7$. $A\cdot2=7T-1$, $\Arg(Q\rightarrow P;2)=\frac 5{14}$;
			\begin{table}[!htbp]
				\centering
				\begin{tabular}{|c|ccccccccccccc|}
					\hline
					$c'\equiv 3\Mod 7$ & $d_1$ & $\rightarrow$ & $d_2$ & $\rightarrow$ & $d_3$ & $\rightarrow$ & $d_4$ & $\rightarrow$ & $d_5$ & $\rightarrow$ & $d_6$ & $\rightarrow$ & $d_1$\\
					\hline
					$\Arg(d_u\rightarrow d_v;1)$  & & $\frac 5{14}$ & & $\frac 2{7}$ & & $\frac 1{7}$ & & $\frac 2{7}$ &  & $\frac 5{14}$ & & $-\frac 3{7}$ & \ \\
					[1ex]
					$\Arg(d_u\rightarrow d_v;2)$ & & $\frac 1{14}$ & & $\frac 5{14}$ & & $\frac 3{7}$ & & $\frac 5{14}$  &  & $\frac 1{14}$ & & $-\frac 2{7}$ & \ \\
					[1ex]
					$\Arg(d_u\rightarrow d_v;3)$ & & $\frac3{7}$ & & $-\frac 5{14}$ & & $-\frac 3{7}$ & & $-\frac 5{14}$  &  & $\frac 3{7}$ & & $\frac 2{7}$ & \ \\
					\hline
				\end{tabular}
			\end{table}
			
			\item[$\bullet$] $c'\equiv 4\Mod 7$. $A\cdot 2=7T+1$, $\Arg(Q\rightarrow P;2)=-\frac 5{14}$; 
			\begin{table}[!htbp]
				\centering
				\begin{tabular}{|c|ccccccccccccc|}
					\hline
					$c'\equiv 4\Mod 7$ & $d_1$ & $\rightarrow$ & $d_2$ & $\rightarrow$ & $d_3$ & $\rightarrow$ & $d_4$ & $\rightarrow$ & $d_5$ & $\rightarrow$ & $d_6$ & $\rightarrow$ & $d_1$\\
					\hline
					$\Arg(d_u\rightarrow d_v;1)$  & & $-\frac 5{14}$ & & $-\frac 2{7}$ & & $-\frac 1{7}$ & & $-\frac 2{7}$ &  & $-\frac 5{14}$ & & $-\frac 3{7}$ & \ \\
					[1ex]
					$\Arg(d_u\rightarrow d_v;2)$ & & $-\frac 1{14}$ & & $-\frac 5{14}$ & & $-\frac 3{7}$ & & $-\frac 5{14}$  &  & $-\frac 1{14}$ & & $\frac 2{7}$ & \ \\
					[1ex]
					$\Arg(d_u\rightarrow d_v;3)$ & & $-\frac 3{7}$ & & $\frac 5{14}$ & & $\frac 3{7}$ & & $\frac 5{14}$  &  & $-\frac 3{7}$ & & $-\frac 2{7}$ & \ \\
					\hline
				\end{tabular}
			\end{table}
			
			\item[$\bullet$] $c'\equiv 5\Mod 7$. $A\cdot 3=7T+1$, $\Arg(Q\rightarrow P;3)=\frac 3{14}$; 
			\begin{table}[!htbp]
				\centering
				\begin{tabular}{|c|ccccccccccccc|}
					\hline
					$c'\equiv 5\Mod 7$ & $d_1$ & $\rightarrow$ & $d_2$ & $\rightarrow$ & $d_3$ & $\rightarrow$ & $d_4$ & $\rightarrow$ & $d_5$ & $\rightarrow$ & $d_6$ & $\rightarrow$ & $d_1$\\
					\hline
					$\Arg(d_u\rightarrow d_v;1)$  & & $\frac 5{14}$ & & $\frac 2{7}$ & & $\frac 1{7}$ & & $\frac 2{7}$ &  & $\frac 5{14}$ & & $-\frac 3{7}$ & \ \\
					[1ex]
					$\Arg(d_u\rightarrow d_v;2)$ & & $\frac 3{14}$ & & $\frac 1{14}$ & & $\frac 2{7}$ & & $\frac 1{14}$  &  & $\frac 3{14}$ & & $\frac 1{7}$ & \ \\
					[1ex]
					$\Arg(d_u\rightarrow d_v;3)$ & & $\frac 1{7}$ & & $\frac 3{14}$ & & $-\frac 1{7}$ & & $\frac 3{14}$  &  & $\frac 1{7}$ & & $\frac 3{7}$ & \ \\
					\hline
				\end{tabular}
			\end{table}
			
			\item[$\bullet$] $c'\equiv 6\Mod 7$. $A\cdot 1=7T-1$, $\Arg(Q\rightarrow P;1)=\frac 37$. 
			\begin{table}[!htbp]
				\centering
				\begin{tabular}{|c|ccccccccccccc|}
					\hline
					$c'\equiv 6\Mod 7$ & $d_1$ & $\rightarrow$ & $d_2$ & $\rightarrow$ & $d_3$ & $\rightarrow$ & $d_4$ & $\rightarrow$ & $d_5$ & $\rightarrow$ & $d_6$ & $\rightarrow$ & $d_1$\\
					\hline
					$\Arg(d_u\rightarrow d_v;1)$  & & $-\frac 3{14}$ & & $\frac 3{7}$ & & $-\frac 2{7}$ & & $\frac 3{7}$ &  & $-\frac 3{14}$ & & $-\frac 1{7}$ & \ \\
					[1ex]
					$\Arg(d_u\rightarrow d_v;2)$ & & $\frac 3{14}$ & & $\frac 1{14}$ & & $\frac 2{7}$ & & $\frac 1{14}$  &  & $\frac 3{14}$ & & $\frac 1{7}$ & \ \\
					[1ex]
					$\Arg(d_u\rightarrow d_v;3)$ & & $\frac 3{7}$ & & $-\frac 5{14}$ & & $-\frac 3{7}$ & & $-\frac 5{14}$  &  & $\frac 3{7}$ & & $\frac 2{7}$ & \ \\
					\hline
				\end{tabular}
			\end{table}
		\end{enumerate}
	\end{condition}

	%Now we start to prove the ($7$-$k$) cases for $k\in \{0,1,2,3,4,6\}$. 
	
	We only prove (7-0) of Theorem~\ref{Kloosterman sums vanish} and omit the proof of the other ($7$-$k$) cases because the proofs are essentially the same.

	\subsection{(7-0) of Theorem~\ref{Kloosterman sums vanish}}
	
	As \eqref{Vrc sum equals 0, mod 7}, we still denote $V(r,c)=\{d\Mod c^*: d\equiv r\Mod c'\}$, $d_j\in V(r,c)$ by $d_j\equiv j\Mod 7$. Recall the Kloosterman sums defined at \eqref{S infty infty for mod p}, \eqref{S 0 infty mod 7, +1} and \eqref{S 0 infty mod 7, -1}.

	For $A\defeq c'=c/7$, when $A\ell=7T+1$ for some integer $T\geq 0$, as \eqref{Q B} we define 
	\begin{equation}
		Q(B)\defeq 2\sqrt 7Q_1(B)Q_2(B)Q_3(B)
	\end{equation} 
	with $Q_1(B)\defeq (-1)^{[A\ell]}i$ and
	\begin{equation}\label{Q B 7n, +1}
		Q_2(B)\defeq e\(\frac{(\frac 32 T^2+\frac 12 T)C}A\),\quad \text{and}\quad Q_3(B)=e\(\frac{0\cdot B}A\)=1
	\end{equation}
	and let $B=-d_1T$ with $C=-7\overline{d_{1\{A\}}}$. When $A\ell=7T-1$ for some $T\geq 0$, we still define $Q(B)$ as above while we take 
	\begin{equation}\label{Q B 7n, -1}
		Q_2(B)\defeq e\(\frac{(\frac 32 (T-1)^2+\frac 52 T+1)C}A\),\quad \text{and}\quad B=d_1T. 
	\end{equation}
	instead. Note that when $A$ is fixed, $\ell$ is also fixed, i.e. there is only one corresponding $Q(B)$ for every fixed $c$.

	We define the sum on $V(r,c)$ as
	\begin{align}{\label{Vrc sum equals 0, 7n}}
		\begin{split}
			s_{r,c}^{(\ell)}&\defeq \sin(\tfrac{\pi\ell}7)\sum_{d\in V(r,c)}P_1(d)P_2(d)P_3(d)+\sin(\tfrac{\pi\ell}7) \mathbf{1}_{\substack{A\defeq c/7 \\ {[A\ell]=1,6}}} Q(B),\quad\text{where}\\
			P_1(d)&\defeq \frac{(-1)^{\ell c}e\(-\frac{3c'a\ell^2}{14}\)}{\sin(\frac{\pi a\ell}{7})},\quad P_2(d)\defeq e\(-\frac{12 cs(d,c)}{24c}\),\quad  P_3(d)\defeq e\(\frac{0\cdot d}c\)=1. 
		\end{split}
	\end{align}
	Here $\textbf{1}_{\text{condition}}$ equals $1$ if the condition meets and equals $0$ otherwise. 
	
	To prove (7-0) of Theorem~\ref{Kloosterman sums vanish}, it suffices to show
	\begin{equation}\label{src sum to 0, 7n case}
		C_7^{4,6}s_{r,c}^{(1)}+C_7^{6,2}s_{r,c}^{(2)}+C_7^{2,4}s_{r,c}^{(3)} =0. 
	\end{equation}

	First we deal with the case $49|c$ and there is no $Q(B)$. We need to subtract $\frac 57$ from \eqref{Arg diff 7n+5 case 49|c, last section using} and get
	\begin{equation*}
		\Arg(d\rightarrow d_*;\ell)=\left\{\begin{array}{ll}
			0 &\ d\equiv 1,6\Mod 7;\vspace{1ex}\\
			-\frac 27 &\ d\equiv 2,5\Mod 7;\vspace{1ex}\\ 
			\frac 17 &\ d\equiv 3,4\Mod 7. 
		\end{array}
		\right.
	\end{equation*} 
	When $r\equiv d\equiv 2,3,4,5\Mod 7$, we get equi-distribution and \eqref{src sum to 0, 7n case} follows. When $r\equiv d\equiv 1,6\Mod 7$, note that $P_1(d)=(-1)^{(a+1)c\ell}/\sin(\frac{\pi a \ell}7)$ for $ad\equiv 1\Mod c$ has the same $\sgn P_1(d)$ for $\ell=1,2,3$. Hence every summand for $d\in V(r,c)$ in \eqref{src sum to 0, 7n case} has the same argument and we get \eqref{src sum to 0, 7n case} by
	\[C_7^{4,6}\frac{\sin(\frac{\pi}7)}{\sin(\frac{\pi}7)}+C_7^{6,2}\frac{\sin(\frac{2\pi}7)}{\sin(\frac{2\pi}7)}+C_7^{2,4}\frac{\sin(\frac{3\pi}7)}{\sin(\frac{3\pi}7)}=0. \]

	Next we check the condition for $7\|c$. 
	Comparing with Condition~\ref{Styles six points mod 7 7n+5}, since we have different $P_3(d)$ and $Q_3(B)$ in this case, we need to subtract $\frac{5\beta}7$ in $\Arg(d_j\rightarrow d_{j+1};\ell)$, $1\leq j\leq 5$ from Condition~\ref{Styles six points mod 7 7n+5}. We also need to add $\mp\frac{5\ell}7$ to $\Arg(Q\rightarrow P;\ell)$ when $A\ell\equiv \pm 1\Mod 7$. It is important to note that we compute $\Arg(d_6\rightarrow d_1;\ell)$ by
	\[\sum_{j=1}^5\Arg(d_j\rightarrow d_{j+1};\ell)+\Arg(d_6\rightarrow d_1;\ell)=0\]
	instead of adding $\mp\frac{5\ell}7$. 
	\begin{condition}\label{Styles six points mod 7 7n}
		For the $7n$ case, we have the following conditions on $\Arg(Q\rightarrow P;\ell)$ when $A\ell\equiv \pm 1\Mod 7$ with tables for $\Arg_j(d_u\rightarrow d_v;\ell)$.  
		\begin{enumerate}
			\item[$\bullet$] $c'\equiv 1\Mod 7$, $\beta =1$. $A\cdot 1=7T+1$, $\Arg(Q\rightarrow P;1)=-\frac 17$; 
			\begin{table}[!htbp]
				\centering
				\begin{tabular}{|c|ccccccccccccc|}
					\hline
					$c'\equiv 1\Mod 7$ & $d_1$ & $\rightarrow$ & $d_2$ & $\rightarrow$ & $d_3$ & $\rightarrow$ & $d_4$ & $\rightarrow$ & $d_5$ & $\rightarrow$ & $d_6$ & $\rightarrow$ & $d_1$\\
					\hline
					$\Arg(d_u\rightarrow d_v;1)$  & & $\frac 1{2}$ & & $-\frac 1{7}$ & & $-\frac 3{7}$ & & $-\frac 1{7}$ &  & $\frac 1{2}$ & & $-\frac 2{7}$ & \ \\
					[1ex]
					$\Arg(d_u\rightarrow d_v;2)$ & & $\frac 1{14}$ & & $\frac 3{14}$ & & $0$ & & $\frac 3{14}$  &  & $\frac 1{14}$ & & $\frac 3{7}$ & \ \\
					[1ex]
					$\Arg(d_u\rightarrow d_v;3)$ & & $-\frac 				1{7}$ & & $-\frac 5{14}$ & & $-\frac 2{7}$ & & $-\frac 5{14}$  &  & $-\frac 1{7}$ & & $\frac 2{7}$ & \ \\
					\hline
				\end{tabular}
			\end{table}

			\item[$\bullet$] $c'\equiv 2\Mod 7$, $\beta=4$. $A\cdot 3=7T-1$, $\Arg(Q\rightarrow P;3)=-\frac 1{14}$;  
			\begin{table}[!htbp]
				\centering
				\begin{tabular}{|c|ccccccccccccc|}
					\hline
					$c'\equiv 2\Mod 7$ & $d_1$ & $\rightarrow$ & $d_2$ & $\rightarrow$ & $d_3$ & $\rightarrow$ & $d_4$ & $\rightarrow$ & $d_5$ & $\rightarrow$ & $d_6$ & $\rightarrow$ & $d_1$\\
					\hline
					$\Arg(d_u\rightarrow d_v;1)$  & & $-\frac 3{14}$ & & $-\frac 1{7}$ & & $0$ & & $-\frac 1{7}$ &  & $-\frac 3{14}$ & & $-\frac 2{7}$ & \ \\
					[1ex]
					$\Arg(d_u\rightarrow d_v;2)$ & & $-\frac 1{14}$ & & $\frac 1{14}$ & & $-\frac 1{7}$ & & $\frac 1{14}$  &  & $-\frac 1{14}$ & & $\frac 1{7}$ & \ \\
					[1ex]
					$\Arg(d_u\rightarrow d_v;3)$ & & $0$ & & $-\frac 1{14}$ & & $\frac 2{7}$ & & $-\frac 1{14}$  &  & $0$ & & $-\frac 1{7}$ & \ \\
					\hline
				\end{tabular}
			\end{table}
			
			\item[$\bullet$] $c'\equiv 3\Mod 7$. $A\cdot2=7T-1$, $\Arg(Q\rightarrow P;2)=-\frac 3{14}$; 
			\begin{table}[!htbp]
				\centering
				\begin{tabular}{|c|ccccccccccccc|}
					\hline
					$c'\equiv 3\Mod 7$ & $d_1$ & $\rightarrow$ & $d_2$ & $\rightarrow$ & $d_3$ & $\rightarrow$ & $d_4$ & $\rightarrow$ & $d_5$ & $\rightarrow$ & $d_6$ & $\rightarrow$ & $d_1$\\
					\hline
					$\Arg(d_u\rightarrow d_v;1)$  & & $-\frac 3{14}$ & & $-\frac 2{7}$ & & $-\frac 3{7}$ & & $-\frac 2{7}$ &  & $-\frac 3{14}$ & & $\frac 3{7}$ & \ \\
					[1ex]
					$\Arg(d_u\rightarrow d_v;2)$ & & $\frac 1{2}$ & & $-\frac 3{14}$ & & $-\frac 1{7}$ & & $-\frac 3{14}$  &  & $\frac 1{2}$ & & $-\frac 3{7}$ & \ \\
					[1ex]
					$\Arg(d_u\rightarrow d_v;3)$ & & $-\frac1{7}$ & & $\frac 1{14}$ & & $0$ & & $\frac 1{14}$  &  & $-\frac 1{7}$ & & $\frac 1{7}$ & \ \\
					\hline
				\end{tabular}
			\end{table}
			
			\item[$\bullet$] $c'\equiv 4\Mod 7$, $\beta =2$. $A\cdot 2=7T+1$, $\Arg(Q\rightarrow P;2)=\frac 3{14}$; 
			\begin{table}[!htbp]
				\centering
				\begin{tabular}{|c|ccccccccccccc|}
					\hline
					$c'\equiv 4\Mod 7$ & $d_1$ & $\rightarrow$ & $d_2$ & $\rightarrow$ & $d_3$ & $\rightarrow$ & $d_4$ & $\rightarrow$ & $d_5$ & $\rightarrow$ & $d_6$ & $\rightarrow$ & $d_1$\\
					\hline
					$\Arg(d_u\rightarrow d_v;1)$  & & $\frac 3{14}$ & & $\frac 2{7}$ & & $\frac 3{7}$ & & $\frac 2{7}$ &  & $\frac 3{14}$ & & $-\frac 3{7}$ & \ \\
					[1ex]
					$\Arg(d_u\rightarrow d_v;2)$ & & $\frac 1{2}$ & & $\frac 3{14}$ & & $\frac 1{7}$ & & $\frac 3{14}$  &  & $\frac 1{2}$ & & $\frac 3{7}$ & \ \\
					[1ex]
					$\Arg(d_u\rightarrow d_v;3)$ & & $\frac1{7}$ & & $-\frac 1{14}$ & & $0$ & & $-\frac 1{14}$  &  & $\frac 1{7}$ & & $-\frac 1{7}$ & \ \\
					\hline
				\end{tabular}
			\end{table}
			
			\item[$\bullet$] $c'\equiv 5\Mod 7$, $\beta =3$. $A\cdot 3=7T+1$, $\Arg(Q\rightarrow P;3)=\frac 1{14}$; 
			\begin{table}[!htbp]
				\centering
				\begin{tabular}{|c|ccccccccccccc|}
					\hline
					$c'\equiv 5\Mod 7$ & $d_1$ & $\rightarrow$ & $d_2$ & $\rightarrow$ & $d_3$ & $\rightarrow$ & $d_4$ & $\rightarrow$ & $d_5$ & $\rightarrow$ & $d_6$ & $\rightarrow$ & $d_1$\\
					\hline
					$\Arg(d_u\rightarrow d_v;1)$  & & $\frac 3{14}$ & & $\frac 1{7}$ & & $0$ & & $\frac 1{7}$ &  & $\frac 3{14}$ & & $\frac 2{7}$ & \ \\
					[1ex]
					$\Arg(d_u\rightarrow d_v;2)$ & & $\frac 1{14}$ & & $-\frac 1{14}$ & & $\frac 1{7}$ & & $-\frac 1{14}$  &  & $\frac 1{14}$ & & $-\frac 1{7}$ & \ \\
					[1ex]
					$\Arg(d_u\rightarrow d_v;3)$ & & $0$ & & $\frac 1{14}$ & & $-\frac 2{7}$ & & $\frac 1{14}$  &  & $0$ & & $\frac 1{7}$ & \ \\
					\hline
				\end{tabular}
			\end{table}
			
			\item[$\bullet$] $c'\equiv 6\Mod 7$, $\beta =6$. $A\cdot 1=7T-1$, $\Arg(Q\rightarrow P;1)=\frac 17$.  
			\begin{table}[!htbp]
				\centering
				\begin{tabular}{|c|ccccccccccccc|}
					\hline
					$c'\equiv 6\Mod 7$ & $d_1$ & $\rightarrow$ & $d_2$ & $\rightarrow$ & $d_3$ & $\rightarrow$ & $d_4$ & $\rightarrow$ & $d_5$ & $\rightarrow$ & $d_6$ & $\rightarrow$ & $d_1$\\
					\hline
					$\Arg(d_u\rightarrow d_v;1)$  & & $\frac 1{2}$ & & $\frac 1{7}$ & & $\frac 3{7}$ & & $\frac 1{7}$ &  & $\frac 1{2}$ & & $\frac 2{7}$ & \ \\
					[1ex]
					$\Arg(d_u\rightarrow d_v;2)$ & & $-\frac 1{14}$ & & $-\frac 3{14}$ & & $0$ & & $-\frac 3{14}$  &  & $-\frac 1{14}$ & & $-\frac 3{7}$ & \ \\
					[1ex]
					$\Arg(d_u\rightarrow d_v;3)$ & & $\frac 1{7}$ & & $\frac 5{14}$ & & $\frac 2{7}$ & & $\frac 5{14}$  &  & $\frac 1{7}$ & & $-\frac 2{7}$ & \ \\
					\hline
				\end{tabular}
			\end{table}
		\end{enumerate}
	\end{condition}
	Note that the condition for $c'\Mod 7$ is the same as the reversed condition for $-c'\Mod 7$. Hence we only need show the corresponding graphs for $c'\equiv 1,2,3\Mod 7$, and also for the other $7n+k$ cases in the remaining subsections. In each of the following graphs, if $d_u$ and $d_v$ are not shown, then $P(d_u)=P(d_v)$ are both the remaining non-labeled point.

	\begin{center}
		\includegraphics[scale=0.32]{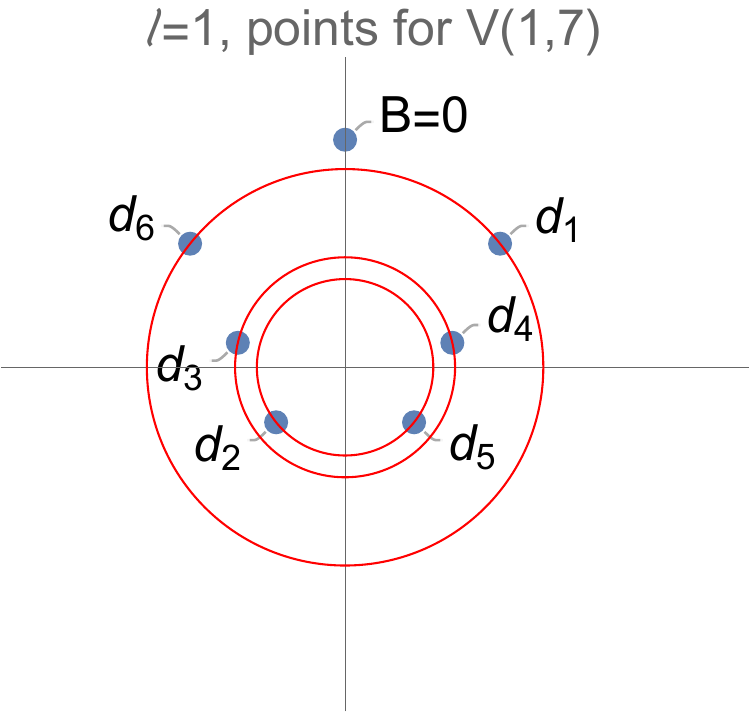}
		\includegraphics[scale=0.32]{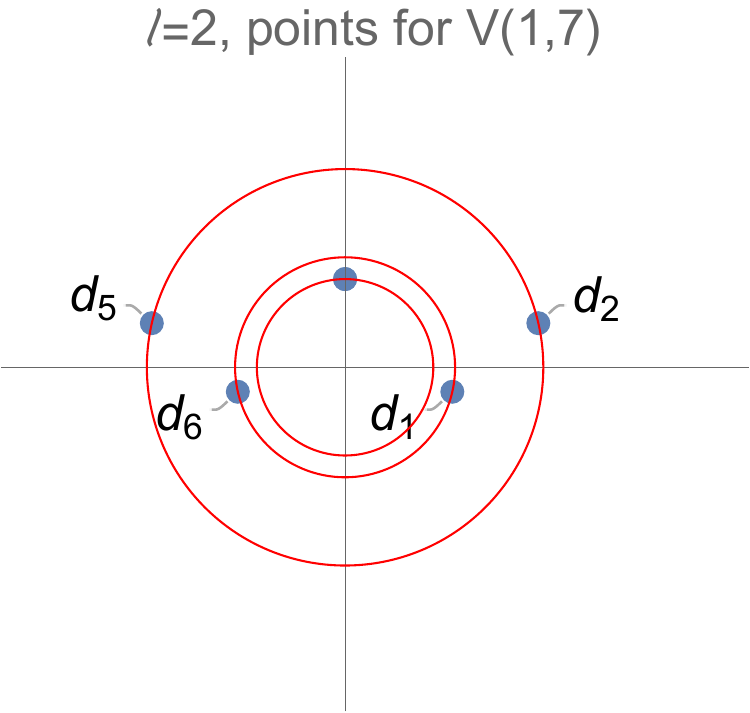}
		\includegraphics[scale=0.32]{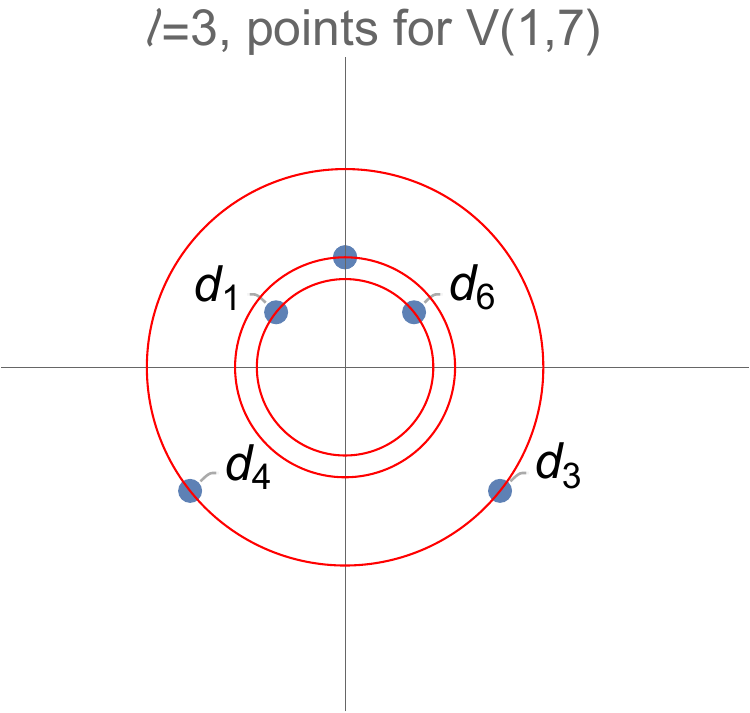}
		\\
		\includegraphics[scale=0.32]{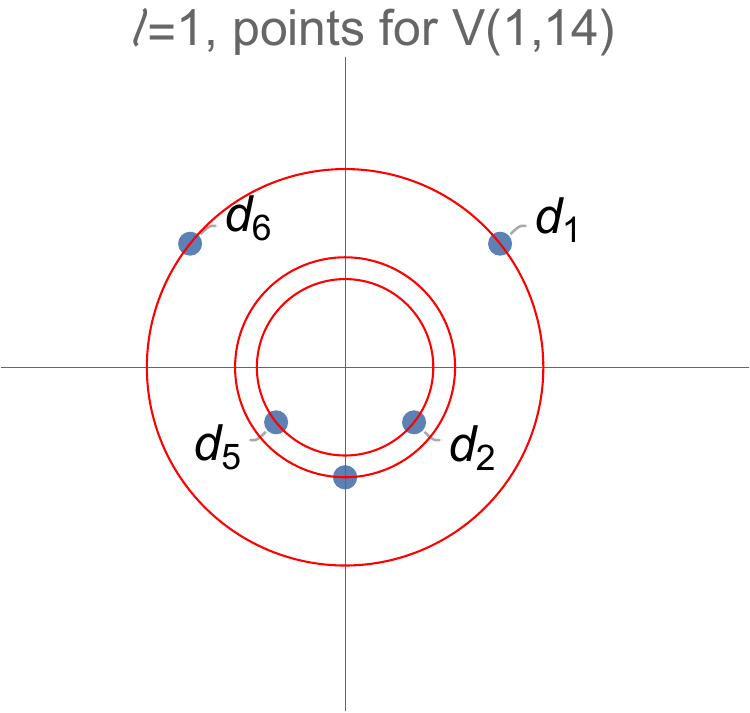}
		\includegraphics[scale=0.32]{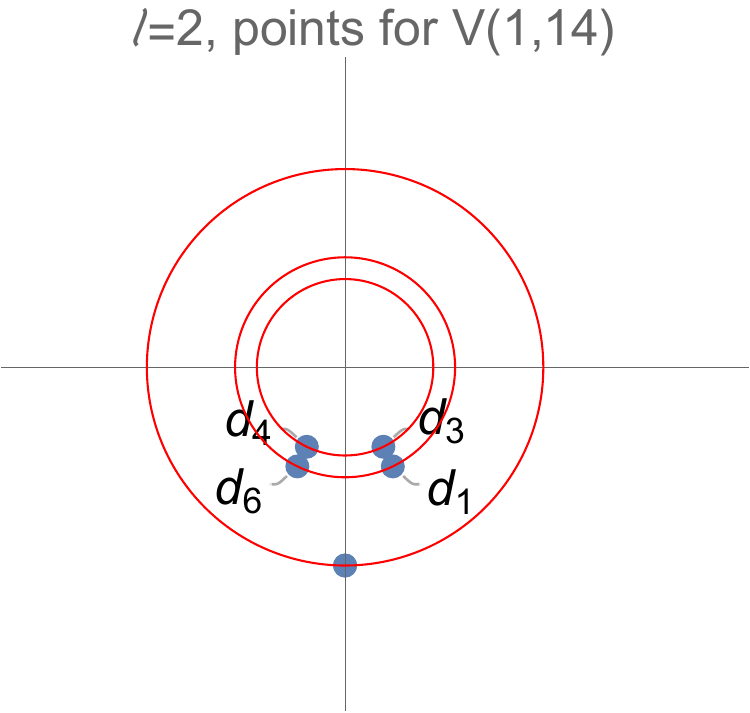}
		\includegraphics[scale=0.32]{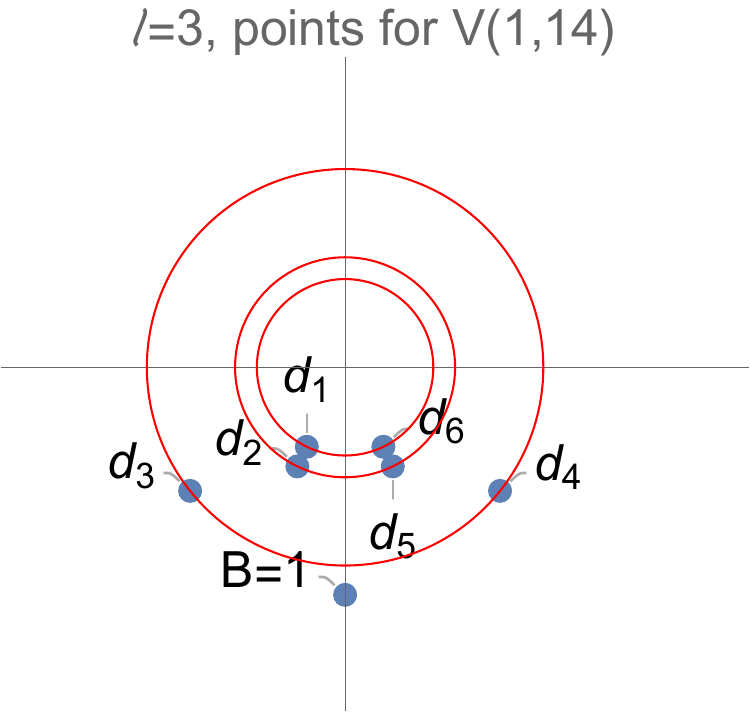}
		\\
		\includegraphics[scale=0.32]{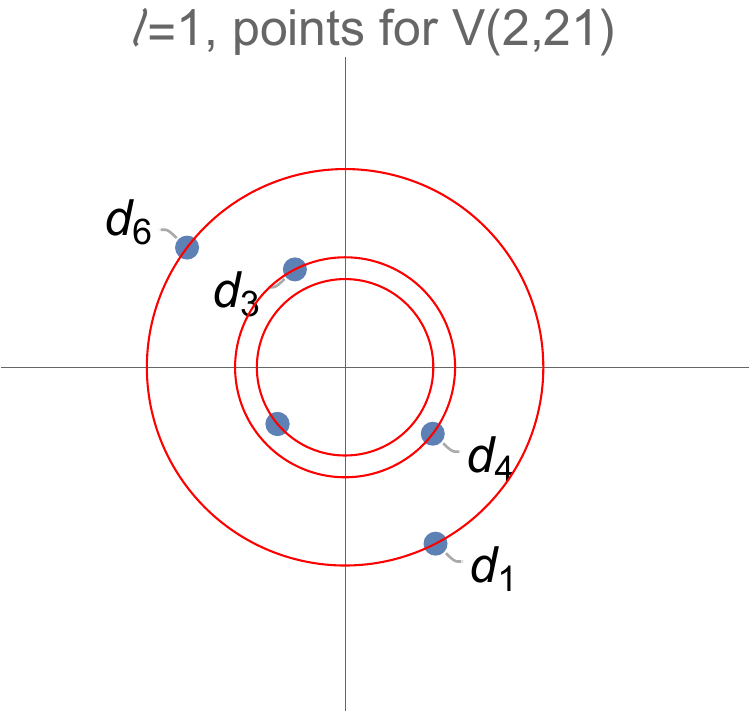}
		\includegraphics[scale=0.32]{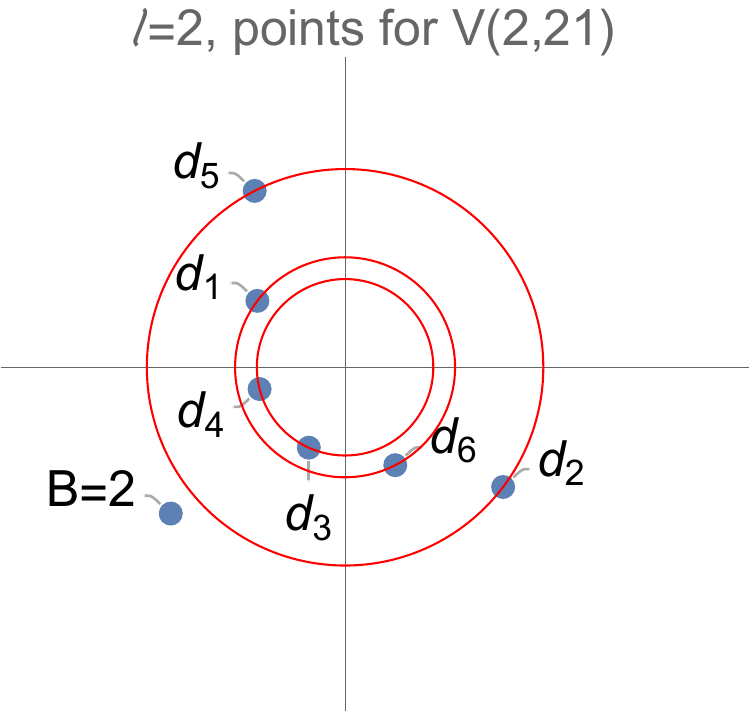}
		\includegraphics[scale=0.32]{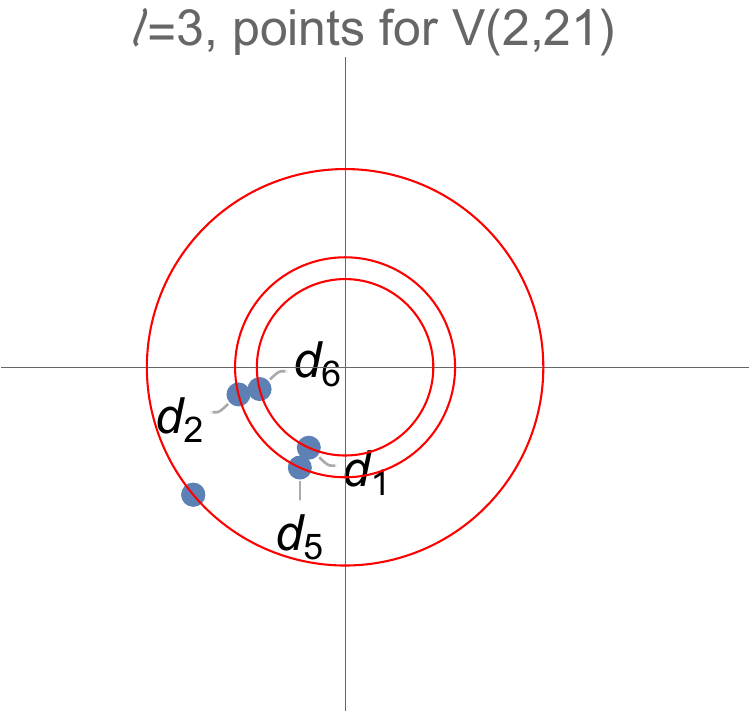}
	\end{center}
	
	Visualizing by the above graphs, \eqref{src sum to 0, 7n case} is proved by the following equations: 
	\begin{align*}
		C_7^{4,6}\sin(\tfrac{\pi}7)&\(\frac{\cos(\frac{2\pi}7)}{\sin(\frac {\pi}7)}+\frac{\cos(\frac{3\pi}7)}{\sin(\frac {2\pi}7)}-\frac{\cos(\frac{2\pi}7)}{\sin(\frac {3\pi}7)}\)+C_7^{6,2}\sin(\tfrac{2\pi}7)\(\frac{\cos(\frac{3\pi}7)}{\sin(\frac {\pi}7)}-\frac{\cos(\frac{3\pi}7)}{\sin(\frac {2\pi}7)}+\frac{1}{\sin(\frac {3\pi}7)}\)\\
		&+C_7^{2,4}\sin(\tfrac{3\pi}7)\(-\frac{\cos(\frac{2\pi}7)}{\sin(\frac {\pi}7)}+\frac{1}{\sin(\frac {2\pi}7)}+\frac{\cos(\frac{2\pi}7)}{\sin(\frac {3\pi}7)}\)=-C_7^{4,6}\sin(\tfrac{\pi}7)\sqrt 7,\\	
		C_7^{4,6}\sin(\tfrac{\pi}7)&\(-\frac{\cos(\frac{2\pi}7)}{\sin(\frac {\pi}7)}+\frac{1}{\sin(\frac {2\pi}7)}+\frac{\cos(\frac{2\pi}7)}{\sin(\frac {3\pi}7)}\)+C_7^{6,2}\sin(\tfrac{2\pi}7)\(\frac{1}{\sin(\frac {\pi}7)}+\frac{\cos(\frac{\pi}7)}{\sin(\frac {2\pi}7)}+\frac{\cos(\frac{\pi}7)}{\sin(\frac {3\pi}7)}\)\\
		&+C_7^{2,4}\sin(\tfrac{3\pi}7)\(\frac{\cos(\frac{2\pi}7)}{\sin(\frac {\pi}7)}+\frac{\cos(\frac{\pi}7)}{\sin(\frac {2\pi}7)}+\frac{\cos(\frac{\pi}7)}{\sin(\frac {3\pi}7)}\)=-C_7^{2,4}\sin(\tfrac{3\pi}7)\sqrt 7,\\	
		C_7^{4,6}\sin(\tfrac{\pi}7)&\(\frac{\cos(\frac{3\pi}7)}{\sin(\frac {\pi}7)}-\frac{\cos(\frac{3\pi}7)}{\sin(\frac {2\pi}7)}+\frac{1}{\sin(\frac {3\pi}7)}\)+C_7^{6,2}\sin(\tfrac{2\pi}7)\(-\frac{\cos(\frac{3\pi}7)}{\sin(\frac {\pi}7)}+\frac{\cos(\frac{3\pi}7)}{\sin(\frac {2\pi}7)}+\frac{\cos(\frac{\pi}7)}{\sin(\frac {3\pi}7)}\)\\
		&+C_7^{2,4}\sin(\tfrac{3\pi}7)\(\frac{1}{\sin(\frac {\pi}7)}+\frac{\cos(\frac{\pi}7)}{\sin(\frac {2\pi}7)}+\frac{\cos(\frac{\pi}7)}{\sin(\frac {3\pi}7)}\)=-C_7^{6,2}\sin(\tfrac{2\pi}7)\sqrt 7. \\	
	\end{align*} 
	This proves (7-0) of Theorem~\ref{Kloosterman sums vanish}. 
We omit the similar proofs of (7-$k$) for $k\in \{1,2,3,4,6\}$ here, which can be found in \cite{QihangKLsumsGitHub}.

	\section*{Acknowledgement}
	The author extends sincere thanks to the referee for their careful review and valuable
    comments. The author thanks Professor Scott Ahlgren for his careful reading in a previous version of this paper and for his plenty of insightful suggestions. The author also thanks Nick Andersen and Alexander Dunn for helpful comments.

	\bibliographystyle{alpha}
	\bibliography{allrefs}
	%\bibliography{../../TeX/allrefs}

\end{document}